\documentclass{amsart}

\usepackage{amsfonts, amssymb, amsmath, eucal, verbatim, amsthm, amscd, enumerate, mathtools}

\usepackage{graphicx}
\usepackage{framed}
\usepackage{tikz}
\usepackage{enumitem}
\usepackage{bbm}

\usepackage{ascmac}

\newtheorem{theorem}{Theorem}[section]
\newtheorem{lemma}[theorem]{Lemma}
\newtheorem{proposition}[theorem]{Proposition}

\newtheorem{corollary}[theorem]{Corollary}
\theoremstyle{definition}
\newtheorem{definition}[theorem]{Definition}
\newtheorem{example}[theorem]{Example}

\theoremstyle{remark}
\newtheorem*{remark}{Remark}

\newcommand{\vertiii}[1]{{\left\vert\kern-0.25ex\left\vert\kern-0.25ex\left\vert #1 
\right\vert\kern-0.25ex\right\vert\kern-0.25ex\right\vert}}
\newcommand{\R}{{\mathbb R}}

\numberwithin{equation}{section}

\usepackage{setspace}

\def\1{\textbf{\rm 1}}

\def\XXint#1#2#3{{\setbox0=\hbox{$#1{#2#3}{\int}$}
\vcenter{\hbox{$#2#3$}}\kern-.5\wd0}}

\makeatletter
\@namedef{subjclassname@2020}{\textup{2020} Mathematics Subject Classification}
\makeatother

\begin{document}

%\date{\today}
\keywords{Gaussian correlation inequality, inverse Brascamp--Lieb inequality}

%\subjclass[2020]{{39B62, 52A40 (primary); 52A38, 80A19 (secondary)}}
\subjclass[2020]{{39B62, 52A40 (primary); 60E15, 60G15 (secondary)}}

%\author[Bez]{Neal Bez}
%\address[Neal Bez]{Department of Mathematics, Graduate School of Science and Engineering,
%Saitama University, Saitama 338-8570, Japan}
%\email{nealbez@mail.saitama-u.ac.jp}
\author[Nakamura]{Shohei Nakamura}
\address[Shohei Nakamura]{School of Mathematics, The Watson Building, University of Birmingham, Edgbaston, Birmingham, B15 2TT, England.}
\email{s.nakamura@bham.ac.uk}
\author[Tsuji]{Hiroshi Tsuji}
\address[Hiroshi Tsuji]{Department of Mathematics, Graduate School of Science and Engineering, Saitama University, Saitama 338-8570, Japan and Department of Mathematics, Institute of Science Tokyo, 2-12-1 Ookayama, Meguro-ku, Tokyo 152-8551, Japan}
\email{tsujihiroshi@mail.saitama-u.ac.jp, tsujihiroshi@math.sci.isct.ac.jp}

\title[The Gaussian correlation inequality]{The Gaussian correlation inequality for centered convex sets and the case of equality}

\begin{abstract}
%{\color{red}
%Inspired by Milman's recent simple proof of the Gaussian correlation inequality, we give a generalization of our previous work on the inverse Brascamp--Lieb inequality. 
%This generalization yields another proof of the Gaussian correlation inequality that is slightly different from Milman's one. 
%We prove that the Gaussian correlation inequality holds for centered convex sets and characterize the case of equality. 
%}
Inspired by Milman's recent observation, we prove that the Gaussian correlation inequality holds for convex sets having the same barycenter, and especially for centered ones.
This gives an affirmative answer to the problem proposed by Szarek and Werner.
We also characterize the equality case. The study of the equality case in the non-symmetric Gaussian correlation inequality {relates to the following question: }  
%is motivated by the following question: 
Let $X$ be a standard Gaussian random vector in $\R^n$. For which convex sets $K_1,K_2 \subset \R^n$, are the two events $\{X\in K_1\}$ and $\{X\in K_2\}$ independent? By imposing an additional normalization that $K_1$ and $K_2$ have the same barycenter, we give the necessary and sufficient conditions for this independence. 
{The conditions also identify when $\|X\|_{K_1}$ and $\|X\|_{K_2}$ are independent as random variables. } 
\end{abstract}

\maketitle

\section{Introduction}
Let $(\Omega,\mathcal{F},\mathbb{P})$ be a probability space. 
The notion of \textit{independence} is one of the fundamental concepts in probability theory. Two events $\Omega_1,\Omega_2 \subset \Omega$ are said to be independent if $\mathbb{P}(\Omega_1\cap \Omega_2) = \mathbb{P}(\Omega_1)\mathbb{P}(\Omega_2)$. 
We are interested in the particular case where $\Omega_i = \{X\in K_i\}$, $i=1,2$, with $K_1,K_2 \subset \mathbb{R}^n$ being Borel sets and $X:\Omega \to \mathbb{R}^n$ being a standard Gaussian random vector. 
%In such a special case, $\Omega_1\cap \Omega_2 = \{ X\in K_1\cap K_2 \}$ and thus the independence of $\Omega_1,\Omega_2$ is equivalent to 
%$
%\gamma(K_1\cap K_2) = \gamma(K_1)\gamma(K_2), 
%$
%where $d\gamma = (2\pi)^{-\frac n2} e^{-\frac12 |x|^2}\, dx$ is the standard Gaussian measure in $\R^n$. 
A typical case of this independence occurs when considering orthogonal stripes. 
As the simplest example, we consider the case $n=2$ with $S_1:=[ -a,a ]\times \R$ and $S_2:= \mathbb{R}\times [-b,b]$, where $a,b>0$. Then it is easy to see that the events  $\Omega_i = \{X\in S_i\}$, $i=1,2$, are independent: 
$$
\mathbb{P}( X \in S_1\cap S_2 ) = \mathbb{P}(X\in S_1) \mathbb{P}(X\in S_2). 
$$
However, for a general choice of $K_1,K_2 \subset \mathbb{R}^n$, the events $ \{X\in K_i\}$ are not necessarily independent, and either $\mathbb{P}(\Omega_1\cap \Omega_2) > \mathbb{P}(\Omega_1)\mathbb{P}(\Omega_2)$ or $\mathbb{P}(\Omega_1\cap \Omega_2) < \mathbb{P}(\Omega_1)\mathbb{P}(\Omega_2)$ could occur. 
What is the expected relationship between these quantities? 
The trivial answer to this line is 
$$
\mathbb{P}(\Omega_1\cap \Omega_2) \le \mathbb{P}(\Omega_1)^\frac12\mathbb{P}(\Omega_2)^\frac12
$$
since $\Omega_1\cap \Omega_2 \subset \Omega_i$ for $i=1,2$. 
The exponent $\frac12$ on the right-hand side is observed to be optimal by taking $\Omega_1=\Omega_2$, since $\mathbb{P}$ is a probability measure. 
In contrast to this upper bound on $\mathbb{P}(\Omega_1\cap \Omega_2)$, there is no hope for a nontrivial lower bound that holds for all $\Omega_1,\Omega_2$ since $\mathbb{P}(\Omega_1\cap \Omega_2)=0$ whenever $\Omega_1 \cap \Omega_2 = \emptyset$.  Nevertheless, a nontrivial lower bound may be available by imposing some structural assumption on $K_1,K_2$. 
Royen's symmetric Gaussian correlation inequality \cite{Roy14} states that for any symmetric convex sets $K_1, K_2 \subset \R^n$ and $\Omega_i = \{X\in K_i\}$, it holds that $\mathbb{P}(\Omega_1\cap \Omega_2) \ge \mathbb{P}(\Omega_1)\mathbb{P}(\Omega_2)$ or equivalently, 
\begin{equation}\label{e:OriginGCI}
\gamma(K_1 \cap K_2) \ge \gamma(K_1) \gamma(K_2), 
\end{equation}
where $d\gamma = (2\pi)^{-\frac n2} e^{-\frac12 |x|^2}\, dx$ is the standard Gaussian measure in $\R^n$. 
Let us briefly review the history of \eqref{e:OriginGCI}. 
For this purpose, it is worth recalling an equivalent formulation of the Gaussian correlation inequality. 
Let $d_1,d_2\in \mathbb{N}$ and let $(Y_1,Y_2)$ be a random vector in $\R^{d_1}\times \R^{d_2}$ that is normally distributed with mean zero and arbitrary covariance. 
Then an equivalent form of the symmetric Gaussian correlation inequality states that for any symmetric convex sets $L_i \subset \R^{d_i}$, $i=1,2$, 
\begin{equation}\label{e:ProbGCI}
\mathbb{P} ( Y_1\in L_1,\; Y_2\in L_2 ) \ge \mathbb{P}(Y_1\in L_1) \mathbb{P} (Y_2 \in L_2). 
\end{equation}
In particular, if we denote $(Y_1,Y_2) = (Z_1,\ldots, Z_{d_1+d_2})$ and take $L_i= [-1,1]^{d_i}$, $i=1,2$, then \eqref{e:ProbGCI} can be written as 
\begin{equation}\label{e:ActualGCI} 
    \mathbb{P} ( \max_{i=1,\ldots, d_1+d_2} |Z_i| \le 1  ) \ge \mathbb{P}( \max_{i=1,\ldots, d_1} |Z_i| \le 1 ) \mathbb{P} ( \max_{i=d_1+1,\ldots, d_1+d_2} |Z_i |\le 1 ). 
\end{equation}
While there is a clear implication \eqref{e:OriginGCI} $\Rightarrow$ \eqref{e:ProbGCI} $\Rightarrow$ \eqref{e:ActualGCI}, all three formulations are in fact equivalent via approximating arguments; see \cite{LM}. 
The formulation of \eqref{e:ProbGCI} is due to Das Gupta et al. \cite{DEOPSS}, although the inequality itself may be traced back to the works of Khatri \cite{Kha} and \v{S}id\'{a}k \cite{Sid}, who independently proved \eqref{e:ActualGCI} when $\min\{d_1,d_2\}=1$. 
The formulation of \eqref{e:OriginGCI} may be found in the work of Pitt \cite{Pitt} and they proved \eqref{e:OriginGCI} for any symmetric convex sets $K_1,K_2$ when $n=2$. 
Regarding the inequality in higher dimensions, Schechtman, Schlumprecht, and Zinn \cite{SSZ}, Harg\'{e} \cite{Harge}, and Cordero-Erausquin \cite{C-E} investigated \eqref{e:OriginGCI} when either $K_1$ or $K_2$ is symmetric ellipsoid. 
After several partial results \cite{Borell, Harge2, Hu, SW}, a complete proof of the symmetric Gaussian correlation inequality in any dimension was provided by the celebrated work of Royen \cite{Roy14}; see also \cite{LM}.  
Recently, a new proof of \eqref{e:ActualGCI} (and thus of \eqref{e:OriginGCI}) was given by Milman \cite{Mil}, and our inspiration for this note comes from his observation. %and this is a source of our motivation of this paper. 
We refer to \cite{ACS,ENT,Tehr} for recent developments and generalizations of the Gaussian correlation inequality. 

In this note, we address the following two questions. 
First; (i) does the Gaussian correlation inequality \eqref{e:OriginGCI} hold for non-symmetric convex sets such as simplices? 
Second; (ii) is there any characterization of the equality case in this non-symmetric Gaussian correlation inequality? 
%{\color{red}The latter question is related to a problem that, while seemingly naive, appears to be fundamental in probability theory: Given a random vector $X$ in $\R^n$, for which Borel sets $K_1,K_2 \subset \R^n$, are the events $\Omega_i:= \{X \in K_i\}$, $i=1,2$, independent? $\leftarrow$Cor1.2の上で全く同じことを述べるのでこちらをカットしてよいか？}
%Let $X$ be a standard Gaussian random vector in $\R^n$. For given two convex sets $K_1,K_2 \subset \R^n$, when are the two events $\{X\in K_i\}$, $i=1,2$, independent? 
%We will give certain answers to these questions. 
%To the first question, we establish the Gaussian correlation inequality for convex sets having the same barycenter with respect to $d\gamma$. 
%This in particular confirms the non-symmetric Gaussian correlation inequality proposed by Szarek--Werner \cite{SW}.
%To the second question, we completely characterize the equality case for the non-symmetric inequality. As a consequence, we provide a characterization of convex sets having the same barycenter $K_1,K_2 \subset\R^n$ for which the two events $\{X\in K_i\}$, $i=1,2$, are independent. 
{In their current form, these questions are too naive to be of genuine mathematical significance.}
%These questions, in the current form, are indeed too naive to be formulated rigorously.
For instance, regarding the first question (i), given two arbitrary convex bodies $K_1, K_2$, we can make them disjoint by translating one of them sufficiently far away. Therefore, there is no hope of establishing \eqref{e:OriginGCI} for all convex sets, and some normalization is necessary. 
In order to make a meaningful normalization, we make use of the barycenter with respect to $\gamma$, which is defined for a Borel set $K$ by 
$$
{\rm bar}_\gamma(K):= \int_{K} x\, \frac{d\gamma}{\gamma(K)}.
$$
We say that $K$ is centered if ${\rm bar}_{\gamma}(K) =0$. 
In fact, this normalization was proposed by Szarek and Werner \cite{SW}. There, they proved \eqref{e:OriginGCI} for any convex body $K_1 \subset \R^n$ and strip $K_2 = \{ x \in \R^n \, :\, a \le \langle x, u \rangle \le b\}$, where $u \in \mathbb{S}^{n-1}$ and $a,b \in \R$, such that ${\rm bar}_\gamma(K_1)$ and ${\rm bar}_\gamma(K_2)$ lie in the same hyperplane $\{ x \in \R^n \, :\, \langle x, u \rangle =c\}$ for some $c \in \R$. 
Given their results, Szarek and Werner \cite[Problem 2]{SW} proposed a formulation of the non-symmetric version of the Gaussian correlation inequality and asked the following question:  If $K_1,K_2\subset \R^n$ are convex sets with ${\rm bar}_\gamma(K_1) = {\rm bar}_\gamma(K_2)$, then does \eqref{e:OriginGCI} hold? 
Our first main result gives an affirmative answer to this question, as well as the first question (i) for reasonable convex sets.  

\begin{theorem}\label{t:NonSymGCI}
    For any convex sets $K_1,K_2 \subset \R^n$ with ${\rm bar}_\gamma(K_1)={\rm bar}_\gamma(K_2)$, we have \eqref{e:OriginGCI}: 
        $$
        \gamma(K_1\cap K_2) \ge \gamma(K_1)\gamma(K_2). 
        $$
\end{theorem}

In particular, this shows that the Gaussian correlation inequality \eqref{e:OriginGCI} holds for all centered convex sets. 
We remark that another formulation of the non-symmetric version of the Gaussian correlation inequality was introduced by Cordero-Erausquin \cite{C-E}; see the forthcoming Corollary \ref{cor:Cordero}. 
Even before \cite{SW, C-E}, a certain functional form of the non-symmetric inequality had been studied by Hu \cite{Hu}; see also the work of Harg\'{e} \cite{Harge2}. However, their results imposed convexity on the input function, rather than log-concavity, and do not provide any geometric inequality. 

The second question (ii) concerns the equality case of Theorem \ref{t:NonSymGCI}. 
We mention that if $K_1,K_2$ are symmetric, then equality in \eqref{e:OriginGCI} holds if and only if there exists a subspace $E $ of $\R^n$ such that 
\begin{equation}\label{e:EqualGCISet}
    K_1 = \overline{K}_1 \times E,\quad K_2 = E^\perp \times \overline{K}_2
\end{equation}
    for some symmetric convex sets $\overline{K}_1 \subset E^\perp$ and $\overline{K}_2 \subset E$. 
This can be obtained from Royen's proof of \eqref{e:OriginGCI} for symmetric convex sets; see also \cite{Mil} for more details. 
%As we have briefly mentioned {\color{blue}in the beginning of our paper}, 
The second question (ii) has a simple interpretation in probability theory. %although our setting is more special and concrete. 
That is, we consider the following question: Let $X$ be a standard Gaussian random vector in $\R^n$. 
What kind of structures should convex sets $K_1,K_2 \subset \R^n$ have for independence of the events $\Omega_i:= \{X \in K_i\}$, $i=1,2$? 
One may naively expect that the same structure such as \eqref{e:EqualGCISet} holds to have this independence. 
However, according to Theorem \ref{t:NonSymGCI}, it turns out that \textit{any} convex bodies $K_1,K_2$ yield the independence of $\Omega_i$ by translating them suitably. 

\begin{corollary}\label{cor:Translate}
Let $X$ be a standard Gaussian random vector in $\R^n$. For any convex bodies $K_1,K_2 \subset \R^n$, there exist translations $a_1,a_2 \in \R^n$ such that the events $\{X\in K_1 + a_1\}$ and $\{X\in K_2 + a_2\}$ are independent. 
\end{corollary}
It seems likely that this property has been proven before, but as we could not find a suitable reference, we include a brief proof, appealing to the non-symmetric Gaussian correlation inequality, in Appendix. 
%Because of this corollary, again,  some normalization is needed for the second question {\color{red}(ii)} {\color{red}in the probabilistic perspective}.
Because of this corollary, we need to impose some normalization in order for investigating structures of $K_1,K_2$.
%
%（コメント：the structures \eqref{e:EqualGCISet}を期待するa prioriな理由はない）
%if we expect eventually the structures \eqref{e:EqualGCISet} as characterizations for independence of events corresponding to non-symmetric convex sets,  some normalization is required.
As in Theorem \ref{t:NonSymGCI}, we here impose ${\rm bar}_\gamma(K_1) = {\rm bar}_\gamma(K_2)$ as a normalization. 
We denote the covariance matrix of a probability measure $\mu$ with finite second moment by 
$$
{\rm Cov}\, (\mu):= \left( \int_{\R^n} x_ix_j\, d\mu - \int_{\R^n} x_i\, d\mu\int_{\R^n} x_j\, d\mu \right)_{i,j=1,\ldots,n}. 
$$
The following theorem provides an answer to the second question (ii) under the same conditions in Theorem \ref{t:NonSymGCI}.

\begin{theorem}\label{t:EqualityCase}
    Let $K_1,K_2\subset \R^n$ be convex sets with ${\rm bar}_\gamma(K_1)= {\rm bar}_\gamma(K_2)$ and $X$ be a standard Gaussian random vector in $\R^n$. 
    Then $\{ X\in K_1 \}$ and $\{ X\in K_2 \}$ are independent if and only if 
    \begin{equation}\label{e:NecCentering}
    {\rm bar}_\gamma(K_1) = {\rm bar}_\gamma(K_2)=0
    \end{equation}
    and \eqref{e:EqualGCISet} hold for  
    $$
    E := \big\{ x\in \R^n:\; {\rm Cov}\, (\mu) x = x \big\},\quad \text{where}\quad d\mu(x) := \frac{1}{\gamma(K_1)} \mathbf{1}_{K_1}d\gamma(x), 
    $$
    and some {convex sets} $\overline{K}_1 \subset E^\perp$ and $\overline{K}_2 \subset E$. 
\end{theorem}

A few remarks are in order. First, the subspace $E$ in the above statement is nothing but the eigenspace of ${\rm Cov}\, (\mu)$ corresponding to the eigenvalue 1. 
Second, $\overline{K_1}$ and $\overline{K_2}$ above necessarily satisfy ${\rm bar}_{\gamma_{E^\perp}}(\overline{K_1}) = {\rm bar}_{\gamma_{E}}(\overline{K_2}) =0$, where $\gamma_E$, $\gamma_{E^\perp}$ are standard Gaussians on $E$, $E^\perp$ respectively by \eqref{e:NecCentering}. 
As a final remark, one may also rephrase this characterization in terms of independence of random variables. 
For a convex set $K$ with $0\in {\rm int} K$, the Minkowski functional, taking its value in $[0,\infty)$, is defined as $\| x \|_{K}:= \inf \{ r>0 : x \in r K \}$ for $x \in \R^n$. 
For given two convex sets $K_1,K_2$ with $0\in {\rm int} K_1, {\rm int} K_2$, we obtain two random variables in $[0,\infty)$ by $X_i:= \| X \|_{K_i}$, $i=1,2$, where $X$ is a standard Gaussian random vector in $\R^n$. 
As usual, $X_1,X_2$ are said to be independent if $\mathbb{P}( X_1 \le r_1,\; X_2 \le r_2 ) = \mathbb{P}( X_1 \le r_1) \mathbb{P}( X_2 \le r_2)$ for all $r_1,r_2\ge0$. 
In our case, 
$$
\mathbb{P}( X_1 \le 1,\; X_2 \le 1 ) = \mathbb{P}( X \in K_1 \cap K_2 ), 
$$
and thus the following is an immediate consequence from Theorem \ref{t:EqualityCase}: 
\begin{corollary}
    Let $K_1,K_2\subset \R^n$ be convex sets with ${\rm bar}_\gamma(K_1)= {\rm bar}_\gamma(K_2)$ and $0\in {\rm int} K_1, {\rm int} K_2$. 
    Let also $X$ be a standard Gaussian random vector in $\R^n$. 
    Then two random variables $\|X\|_{K_1}, \|X\|_{K_2}$ are independent if and only if \eqref{e:EqualGCISet} and \eqref{e:NecCentering} hold. 
\end{corollary}

    %This is because 
    %$$
    %\int_{\overline{K}_1} x_{E^\perp} \,d\gamma_{E^\perp} 
    %=
    %\int_{\overline{K}_1\times E} x_{E^\perp} \,d\gamma(x)
    %=
    %\int_{\overline{K}_1\times E} x_{E^\perp} \,d\gamma(x)
    %+ 
    %\int_{\overline{K}_1\times E} x_{E} \,d\gamma(x)
    %=
    %{\rm bar}_\gamma(K_1)
    %= 
    %0
    %$$
    %since $\int_{\overline{K}_1\times E} x_{E} \,d\gamma(x) = \gamma_{E^\perp}(\overline{K_1}) \int_E x_E\, d\gamma_E = 0. 
%$

For the purpose of proving these results, the existing arguments of Royen and Milman do not readily extend. 
This is because what they actually proved in their proofs is \eqref{e:ActualGCI} rather than \eqref{e:OriginGCI}. As we explained, \eqref{e:ActualGCI} is a special case of \eqref{e:OriginGCI}, and one needs to invoke an approximation argument to recover \eqref{e:OriginGCI} from \eqref{e:ActualGCI}. 
This involves approximating a symmetric convex set by a countably infinite union of slabs and a degenerate Gaussian by a nondegenerate Gaussian.
Given these difficulties, we instead generalize our previous work on the inverse Brascamp--Lieb inequality \cite{NT3} and give a more direct proof of the Gaussian correlation inequality \eqref{e:OriginGCI}, avoiding the approximation arguments mentioned above.
In addition to this ingredient, in order to prove Theorem \ref{t:EqualityCase}, we are led to the following question. Let $E$ be the eigenspace of the covariance matrix of a given probability measure $d\mu$ on $\R^n$ corresponding to eigenvalue 1. Since the eigenspaces of other eigenvalues are orthogonal to $E$, the covariance matrix orthogonally splits into components on $E$ and $E^\perp$. 
The question is: under what conditions on $\mu$ will the pointwise behavior of $d\mu$ be inherited by this splitting property? More precisely, when does $d\mu(x)$ split into the product measure $d\gamma(x_E)d\overline{\mu}(x_{E^\perp})$, $x= x_E+x_{E^\perp} \in \R^n= E\oplus E^\perp$, for some $\overline{\mu}$? 
Through this question, we reveal a connection between the Gaussian correlation inequality and the study of rigidity problems for the spectral gap of the drifted Laplacian such as the works of Cheng and Zhou \cite{CZ} and Gigli, Ketterer, Kuwada, and Ohta \cite{GKKO}.

Our arguments for Theorem \ref{t:NonSymGCI} yield a further extension of \eqref{e:OriginGCI} under the centering assumption. 
For example, we establish the multilinear extension of \eqref{e:OriginGCI}. 
For a symmetric positive definite matrix $\Sigma$, the centered Gaussian with the covariance $\Sigma$ is denoted by 
$$
\gamma_{\Sigma}(x):= %({\rm det}\, (2\pi \Sigma))^{-\frac12} g_{\Sigma^{-1}}(x) = 
({\rm det}\, (2\pi \Sigma))^{-\frac12} e^{-\frac12 \langle x, \Sigma^{-1}x\rangle},\quad x\in \R^n. 
$$ 

\begin{theorem}\label{t:GenCor}
    Let $m \ge 2$ and $\Sigma_0, \Sigma_1, \dots, \Sigma_m \in \R^{n \times n}$ be  symmetric positive definite matrices with $\Sigma_0^{-1} \ge \Sigma_1^{-1}, \dots, \Sigma_m^{-1}$. Then for any centered convex sets $K_1, \dots, K_m \subset \R^n$ in the sense that $\int_{K_i} x_i d\gamma_{\Sigma_i}(x_i) = 0$, it holds that 
    \begin{equation}\label{e:GCI-General}
    \gamma_{\Sigma_0} \left( \bigcap_{i=1}^m K_i \right) 
    \ge
    \prod_{i=1}^m \gamma_{\Sigma_i}(K_i). 
    \end{equation}
\end{theorem}

Clearly, the case of $m=2$ and $\Sigma_0 = \Sigma_1 = \Sigma_2 = {\rm id}_n$ coincides with \eqref{e:OriginGCI}. 
There are further remarks. 
\begin{enumerate}
    \item 
    In the case of $\Sigma_0=\cdots=\Sigma_m = {\rm id}_n$, Theorem \ref{t:GenCor} states that 
    \begin{equation}\label{e:m-GCI}
    \gamma \left( \bigcap_{i=1}^m K_i \right) 
    \ge
    \prod_{i=1}^m \gamma(K_i) 
    \end{equation}
    holds whenever ${\rm bar}_\gamma(K_i) =0$,  $i=1,\ldots, m$. 
    If $K_1,\ldots,K_m$ all are symmetric convex sets, then \eqref{e:m-GCI} easily follows from the classical case of $m=2$. 
    However, in the framework of centered convex sets, the same reduction does not work since the centering assumption is not always preserved under the intersection operation. 
    \item 
    Even when $m=2$ and $K_1,K_2$ are symmetric, \eqref{e:GCI-General} with generic $\Sigma_0,\Sigma_1,\Sigma_2$ does not follow directly from \eqref{e:OriginGCI}.%; see Appendix. 
    \item 
    If one drops the assumption $\Sigma_0^{-1} \ge \Sigma_1^{-1}, \dots, \Sigma_m^{-1}$ in the above statement, then the inequality \eqref{e:GCI-General} %with constant $1$ (or even some universal constant) 
    does not always hold true. %; again see Appendix. 
    \item 
    {It is non-trivial to characterize the case of equality in \eqref{e:GCI-General} with this level of generality even when all $K_i$ are symmetric. We leave this problem to interested readers.}
\end{enumerate}

The structure of this note is as follows. 
In section 2, we will give a brief introduction of the symmetric inverse Brascamp--Lieb inequality, and exhibit our result on this inequality. The proof of this result is given in Section 3. 
%We then explain how this result implies Theorem \ref{t:GenCor} and hence the inequality part of Theorem \ref{t:CenterGCI+Equal}. 
%In section 3, we will prove the aforementioned result on the symmetric inverse Brascamp--Lieb inequality. 
In Section 4, we prove Theorems \ref{t:NonSymGCI} and \ref{t:GenCor}. 
In Section 5, we complete the proof of Theorem \ref{t:EqualityCase}. 
%{\color{blue}In this last section, we will also discuss another Gaussian correlation inequality investigated by Cordero-Erausquin.もしアブストラクにDarioを入れるなら．} 

\section{The inverse Brascamp--Lieb inequality}
It is very recent that Milman \cite{Mil} gave a simple proof of the symmetric Gaussian correlation inequality by realizing that it is an example of the \textit{symmetric inverse Brascamp--Lieb inequality} with a suitable regularization.  
The later is a family of multilinear functional inequalities that has been investigated by the authors \cite{NT3}. 
The word \textit{symmetric} comes from the crucial assumption in there that input functions are supposed to be even. 
The theory of the inverse Brascamp--Lieb inequality without any symmetry has been developed by Chen, Dafnis, and Paouris \cite{CDP} and Barthe and Wolff \cite{BW} before the authors. 
{Our approach to the Gaussian correlation inequality is also based on the inverse Brascamp--Lieb inequality. However, our way of using the latter inequality is slightly different from Milman's one, and this point is crucial for our purpose of proving Theorem \ref{t:NonSymGCI}.}
%We refer the reader to Section 2 for more comprehensive introduction of the theory of this inequality. 
%Given Milman's critical observation, it is reasonable to expect that the Gaussian correlation inequality for centered convex sets (Theorem \ref{t:GenCor}) would follow from the inverse Brascamp--Lieb inequality for centered (with respect to Lebesgue measure) input functions. 
%This centering assumption is clearly weaker than the evenness assumption, and so the main result in \cite{NT3} is not directly applicable. 
%However, as we alluded in \cite{NT3}, it was evident to us that most of arguments in \cite{NT3} formally work well even when input functions are just centered rather than even, apart from technical justifications. %This justification is for instance about the change of the order of limit and integration. 
%Consequently, the main body of this note will be devoted to providing rigorous of these. %justifications of these technical problems rather than revealing a novel idea. 
%In Section 3, we are going to confirm that this is the case.  We note that these are subtle objects and require us to develop new techniques handling centered functions which are less regular than even functions. 

Below, we first give a brief overview of a recent development regarding the theory of the inverse Brascamp--Lieb inequality, and then exhibit our result on the centered Gaussian saturation principle under the centering assumption. 
%We then explain how we can derive Theorem \ref{t:GenCor} from the result. 
%\subsection{An introduction to the theory of the Brascamp--Lieb inequality}
We will abbreviate the positive definiteness of a symmetric matrix $A$ by just writing $A>0$, and similarly the positive semi-definiteness by $A\ge 0$. 
We also denote $A_1 > A_2$ (and $A_1\ge A_2$) when $A_1 - A_2 >0$ (and $A_1 - A_2 \ge 0$). 
For later use, we introduce the notation 
$$g_A(x):= e^{-\frac12 \langle x, Ax\rangle},\quad x\in \R^n$$ for $A>0$. 
%{\color{red}We are going to use $+$ only for $\ge0$. We will not use any notation indicating $>0$. }
Let $m,n_1,\ldots, n_m, N\in \mathbb{N}$ and take a linear map $B_i:\R^N\to \R^{n_i}$ for each $i=1,\ldots,m$. 
%\footnote{
%\color{red}
%If one of $B_i$ is not surjective, it is immediate to see that the inverse Brascamp--Lieb constant becomes $0$. To be more precise, let us suppose $a_i \in \R^{n_i} \notin {\rm Im}\, B_i $ in which case we also have that $-a_i \notin {\rm Im}\, B_i$. 
%By taking $f_i := \frac12\big( \delta^{(n_i)}(\cdot - a_i) + \delta^{(n_i)}(\cdot + a_i) \big)$ which is even and hence centered, and any sufficiently regular $f_j$, $j\ne i$, it follows that 
%$$
%\int_{\R^{n_i}} f_i\, dx_i =1,\quad 
%\int_{\R^N} e^{\langle x,\mathcal{Q}x\rangle} 
%\prod_{i=1}^m f_i(B_ix)^{c_i}\, dx = 0
%$$
%whatever $\mathcal{Q}$, $B_j$, $j\neq i$, and $\mathbf{c}$ are. 
%もちろんこの$f_i$はlog-concaveではないのだが，genericなformulationとしては$B_i$: surjectiveの必要性はわかるかと．
%ただし，centered Gaussian constantが0になるかどうかはわからない．というかもしかしたら0にならない可能性すらある...のか？
%$->$ たぶんcentered Gaussian constも0になる．
%実際，上でとった$a_i$はもちろん$\neq 0$なので，subspace $\langle a_i\rangle$が${\rm Im}\, B_i$に入らないということになる（$B_i$の線形性）．正確には$\langle a_i\rangle \cap {\rm Im}\, B_i = \{0\}$. 
%するとこの時，subspace $\langle a_i\rangle$上のDirac deltaを近似するcentered Gaussianを考えることで，Gaussian const = 0が出てくるんじゃないか？
%} 
We  also take $c_1,\ldots, c_m >0$ and a symmetric matrix $\mathcal{Q}\in \R^{N\times N}$. 
By abbreviating $\mathbf{B}= (B_1,\ldots,B_m)$ and  $\mathbf{c} = (c_1,\ldots,c_m)$,  we refer $(\mathbf{B},\mathbf{c},\mathcal{Q})$ as the \textit{Brascamp--Lieb datum}. 
For each fixed Brascamp--Lieb datum $(\mathbf{B},\mathbf{c},\mathcal{Q})$, we are interested in the Brascamp--Lieb functional 
$$
{\rm BL}({\bf f})
=
{\rm BL}( {\bf B}, {\bf c}, \mathcal{Q} ; {\bf f}) 
\coloneqq
\frac{ \int_{\R^N} e^{\langle x, \mathcal{Q} x\rangle} \prod_{i=1}^m f_i(B_ix)^{c_i}\, dx}{ \prod_{i=1}^m \left( \int_{\R^{n_i}} f_i\, dx_i \right)^{c_i} }
$$
defined for ${\bf f} = (f_1, \dots, f_m) \in L_+^1(\R^{n_1}) \times \cdots \times L_+^1(\R^{n_m})$. 
Here we denote a class of all nonnegative and integrable functions (which is non-zero) by 
$L_+^1(\R^n) := \{ f \in L^1(\R^n) \, :\, f\ge 0, \int_{\R^n} f\, dx >0\}$.
When one concerns the upper bound of the Brascamp--Lieb functional, it is called as the forward Brascamp--Lieb inequality. 
Similarly, the lower bound of the functional is called the inverse Brascamp--Lieb inequality. 

%\textbf{Brief brush stroke on the history of forward, inverse, and symmetric inverse BL here.}

We give a brief history about this family of inequalities. 
The motivating examples are the sharp forms of forward and inverse Young's convolution inequality that give sharp upper and lower bounds of the functional 
\begin{align*}
\int_{\mathbb{R}^{2n}} f_1(x_1)^{\frac1{p_1}} f_2(x_2)^\frac1{p_2} f_3(x_1-x_2)^{\frac1{p_3}}\, dx_1dx_2 \bigg{/} \prod_{i=1}^3 \big( \int_{\mathbb{R}^n} f_i\, dx_i \big)^{\frac1{p_i}}. 
\end{align*}
Here the upper bound is meaningful only when $p_i\ge1$ and so is the lower bound only when $p_i<1$. Also the scaling condition $\sum_{i=1}^3 \frac1{p_i} = 2$ is necessary to have any nontrivial bound. 
The celebrated works of Beckner \cite{Beckner} as well as Brascamp and Lieb \cite{BraLi_Adv} identify the sharp upper and lower bound:  the best constant is achieved by an appropriate centered Gaussian.  
Another example with $\mathcal{Q}\neq 0$ is the dual form of Nelson's hypercontractivity for Ornstein--Uhlenbeck semigroup. 
The strength of the forward Brascamp--Lieb inequality has been repeatedly revealed since its born. 
For instance, Ball \cite{BallJLMS} penetrated it in the context of convex geometry. He applied the forward  Brascamp--Lieb inequality in order to derive his volume {ratio} estimate as well as the reverse isoperimetric inequality. 
Applications and perspectives of the theory {stems into} Harmonic analysis, combinatorics, analytic number theory,  convex / differential geometry, probability, stochastic process and statistics, statistical mechanics, information theory, and theoretical computer science; we refer interested readers to references in \cite{BBBCF,BCCT}.  
Whole theory of the forward Brascamp--Lieb inequality is underpinned by the fundamental principle so-called \textit{centered Gaussian saturation principle} that was established by Lieb \cite{Lieb}. 
\begin{theorem}[Lieb \cite{Lieb}]\label{t:Lieb}
Let $(\mathbf{B},\mathbf{c},\mathcal{Q})$ be the Brascamp--Lieb datum and $\mathcal{Q}\le0$. 
Then for any $f_i \in L^1_+(\R^{n_i})$, $i=1,\ldots, m$, 
$$
\int_{\R^N} e^{\langle x, \mathcal{Q} x\rangle} \prod_{i=1}^m f_i(B_ix)^{c_i}\, dx 
\le 
\big(
\sup_{A_i>0\; \text{on}\; \R^{n_i}} {\rm BL}(g_{A_1},\ldots, g_{A_m} ) 
\big)
\prod_{i=1}^m \left( \int_{\R^{n_i}} f_i\, dx_i \right)^{c_i}.  
$$
\end{theorem}
As the name suggests, this principle reduces the problem of identifying the best constant of the inequality to the investigation of just centered Gaussians. 
Remark that there is no nontrivial upper bound if $\mathcal{Q} \not\le 0$. 

Compared to the forward case, the inverse Brascamp--Lieb inequality is relatively new topic. 
Its archetypal example is the sharp lower bound of Young's convolution functional that we explained above, although, to be precise, one needs to consider some of $c_i$'s are negative in this case. We remark that Chen, Dafnis, and Paouris \cite{CDP} and Barthe and Wolff \cite{BW} have considered such a scenario too. Our analysis in this paper is also applicable to deal with negative $c_i$ to some extent, but we stick to positive ones in this section. 
%The basic idea of our analysis in this paper is also applicable even when some of $c_i<0$ although making the argument rigorous yields further technical problems. 
%For the purpose of clear exposition, we will not address this problem in this paper. 
%Another (and the simplest) example is the inverse H\"{o}lder inequality 
%$$
%p_1,p_2 <1,\; \frac1{p_1}+\frac1{p_2} =1\quad \Rightarrow \quad \int_{\mathbb{R}^n} f_1(x)^\frac1{p_1} f_2(x)^\frac1{p_2}\, dx \ge \prod_{i=1,2} \big( \int_{\R^n} f_i\, dx_i \big)^\frac1{p_i}. 
%$$
In \cite{CDP}, Chen, Dafnis, and Paouris initiated the study of the inverse Brascamp--Lieb inequality for some class of the Brascamp--Lieb datum. 
Shortly after, Barthe and Wolff \cite{BW} have developed a systematic study for more general class of Brascamp--Lieb data under their nondegenracy condition imposed on the datum $(\mathbf{B},\mathbf{c},\mathcal{Q})$.
Although their nondegeneracy condition is fairly general, it is vital to establsih the theory without it in order to reveal the link to the Blaschke--Santal\'{o}-type inequality in convex geometry; see the introduction of \cite{NT3} for more detailed discussion. 
Motivated by this link, the authors investigated the inverse Brascamp--Lieb inequality with a certain class of Brascamp--Lieb data that are relevant to the application to convex geometry. 
There, we proposed to study the inverse Brascamp--Lieb inequality by assuming some  symmetric assumption on $f_i$ rather than imposing the nondegeneracy condition on $(\mathbf{B},\mathbf{c},\mathcal{Q})$. More precisely, we established the following centered Gaussian saturation principle.  
\begin{theorem}[{\cite[Theorem 1.3]{NT3}}]\label{t:NT3}
Let $m,n_1,\ldots, n_m \in \mathbb{N}$, $c_1,\ldots,c_m>0$, and $\mathcal{Q}$ be a symmetric matrix on $\R^N:= \bigoplus_{i=1}^m \R^{n_i}$. 
Then for any even and log-concave $f_i \in L^1_+(\R^{n_i})$, $i=1,\ldots, m$, 
$$
\int_{\R^N} e^{\langle x, \mathcal{Q} x\rangle} \prod_{i=1}^m f_i(x_i)^{c_i}\, dx 
\ge 
\big(
\inf_{A_i>0\; \text{on}\; \R^{n_i}} {\rm BL}(g_{A_1},\ldots, g_{A_m} ) 
\big)
\prod_{i=1}^m \left( \int_{\R^{n_i}} f_i\, dx_i \right)^{c_i},
$$
where we choose $B_i:\R^N\to \R^{n_i}$ as the orthogonal projection onto $\R^{n_i}$ in the right-hand side. 
\end{theorem}
%Few remarks are in order. 
%Firstly, 
{We remark that the conclusion of Theorem \ref{t:NT3} may fail to hold if one does not impose any kind of symmetric assumption. To see this, we note that Theorem \ref{t:NT3} contains the functional form of the Blaschke--Santal\'{o} inequality due to Ball \cite{BallPhd} as a special case. See also \cite{NT3} for more detailed explanations.
}

In order to prove Theorem \ref{t:NonSymGCI}, we will need to generalize Theorem \ref{t:NT3}. 
Let us take two nonnegative definite matrices $G,H$ such that $G\le H$.  
Then we say that a function $f:\mathbb{R}^n\to [0,\infty)$ is $G$-uniformly log-concave if $\frac{f}{g_G}$ is log-concave on $\R^n$. 
Similarly, we say that $f$ is $H$-semi log-convex if $\frac{f}{g_H}$ is log-convex on $\R^n$. 
Also through this note, we say that an integrable function $f$ on $\R^n$ is centered if $\int_{\R^n} |x|f\, dx <+\infty$ and $\int_{\R^n} xf\, dx=0$. Remark that any integrable log-concave function has the finite first moment, namely $\int_{\R^n} |x| f\, dx < +\infty$. 
With this terminology, for $0 \le G\le H$, we define 
%$\mathcal{F}^{(o)}_{G,H}
% =
% \mathcal{F}^{(o)}_{G,H}(\R^n)
% $ as the set of all functions $f \in L^1_+(\mathbb{R}^n)$ such that $\int_{\mathbb{R}^n} xf(x)\, dx =0$ and $f$ is $G$-uniformly log-concave and $H$-semi log-convex. 
\begin{align*}
&\mathcal{F}^{(o)}_{G,H}
=
\mathcal{F}^{(o)}_{G,H}(\R^n)
\\
&:= 
\{f \in L^1_+(\mathbb{R}^n): 
\text{centered}, \; 
G\text{-uniformly log-concave},\; 
H\text{-semi log-convex}
\}.
\end{align*}
We will formally consider the case of $G =0$ or $H= \infty$ as well. 
That is, $f$ being $0$-uniformly log-concave means just log-concave. 
Also, $f$ being $\infty$-semi log-convex means there is no restriction. 
Thus, for $0\le G< \infty$, 
\begin{align*}
\mathcal{F}^{(o)}_{G,\infty}
&= 
\{f \in L^1_+(\mathbb{R}^n): 
%\int_{\mathbb{R}^n} xf(x)\, dx =0,\; 
\text{centered}, \;
G\text{-uniformly log-concave}
\}, \\
\mathcal{F}^{(o)}_{0,\infty}
&= 
\{f \in L^1_+(\mathbb{R}^n): 
\text{centered}, \;
\text{log-concave,} %\; 
%H\text{-semi log-convex}
\}. 
%\\
%\mathcal{F}^{(o)}_{LC}
%&:= 
%\mathcal{F}^{(o)}_{0,\infty}
%= 
%\{f \in L^1_+(\mathbb{R}^n): 
%%\int_{\mathbb{R}^n} xf(x)\, dx =0,\; 
%\text{centered}, \;
%\text{log-concave}
%\}. 
\end{align*}
%{\color{red}$\mathcal{F}^{(o)}_{0,H}$も使わないので消しました}
%Technical but important remark is that whenever $H<\infty$, the $H$-semi log-convex function $f$ is always positive everywhere
%\footnote{If $f$ is log-convex on $\R^n$ and there exists some $x_0 \in \R^n$ with $f(x_0)=0$, then $f \equiv 0$. In fact log convexity means that for any $x \in \R^n$, 
%$$
%\log f( \frac{x+x_0}{2}) \le \frac12 \log f(x) + \frac12 \log f(x_0) = -\infty
%$$
%since $\log f \in [- \infty, \infty)$.
%Since any point in $\R^n$ may be written as $(x+x_0)/2$ for some $x \in \R^n$, we conclude $f \equiv 0$. 
%},
%that is, 
%for $0\le G\le H<\infty$, 
%\begin{equation}\label{e:Positivity} 
%    f \in \mathcal{F}^{(o)}_{G,H}\quad 
%    \Rightarrow
%    \quad 
%    f>0. 
%\end{equation}
%For the proof, we refer to Lemma {\color{red}???} below.
%{\color{red}
%この性質を使っている場所を後でチェック！
%}
By arming this regularized class, we next introduce the regularized inequality. 
Let $(\mathbf{B},\mathbf{c},\mathcal{Q})$ be an arbitrary Brascamp--Lieb datum. 
We also take $0\le G_i\le H_i \le \infty$ on $\mathbb{R}^{n_i}$ for each $i=1,\ldots,m$,  and denote by ${\bf G} = (G_i)_{i=1}^m$ and ${\bf H} = (H_i)_{i=1}^m$.  

\begin{definition}
For $0\le G_i \le H_i \le\infty$, let ${\rm I}_{{\bf G}, {\bf H}}^{(o)} ({\bf B}, {\bf c}, \mathcal{Q}) \in [0, \infty]$ be the smallest constant for which the inequality 
$$
\int_{\R^N} e^{\langle x, \mathcal{Q} x\rangle} \prod_{i=1}^m f_i(B_ix)^{c_i} \, dx 
\ge 
{\rm I}_{{\bf G}, {\bf H}}^{(o)} ({\bf B}, {\bf c}, \mathcal{Q})
\prod_{i=1}^m \left( \int_{\R^{n_i}} f_i\, dx_i \right)^{c_i}
$$ 
holds for all $f_i \in \mathcal{F}^{(o)}_{G_i,H_i}$. In other words, 
$$
{\rm I}_{{\bf G}, {\bf H}}^{(o)} ({\bf B}, {\bf c}, \mathcal{Q})
:= 
\inf_{f_i \in \mathcal{F}^{(o)}_{G_i,H_i}} {\rm BL}(\mathbf{f}).
$$
Similarly, 
$$
{\rm I}_{{\bf G}, {\bf H}}^{(\mathcal{G})} ({\bf B}, {\bf c}, \mathcal{Q})
:= 
\inf_{A_i: G_i \le A_i \le H_i} {\rm BL}(g_{A_1},\ldots, g_{A_m}).
$$
\end{definition}

Despite of such a generic definition, we will be interested only in two cases below (i) all $H_i<+\infty$ and (ii) all $H_i = \infty$. 
In the latter case, we will simply denote 
$$
{\rm I}_{{\bf G}}^{(o)} ({\bf B}, {\bf c}, \mathcal{Q})
:= 
{\rm I}_{{\bf G}, \infty}^{(o)} ({\bf B}, {\bf c}, \mathcal{Q}),
\quad 
{\rm I}_{{\bf G}}^{(\mathcal{G})} ({\bf B}, {\bf c}, \mathcal{Q})
:= 
{\rm I}_{{\bf G}, \infty}^{(\mathcal{G})} ({\bf B}, {\bf c}, \mathcal{Q}). 
$$

This type of regularization may be found in the intermediate step of the proof of Theorem \ref{t:NT3}  although it had been frequently appeared in the literature.  
For instance, the forward Brascamp--Lieb inequality with $H$-semi log-convexity has been implicitly appeared in the work of Bennett, Carbery, Christ, and Tao \cite{BCCT}, see also  works \cite{BN,CP,Val} for other settings. 
The importance of this regularized formulation was recently discovered by Milman \cite{Mil} in his simple proof of the Gaussian correlation inequality. %We will capitalize his observation to prove Theorem \ref{t:GenCor} later. 
Our main result in this section is the following Gaussian saturation principle: 
\begin{theorem}\label{t:CenteredIBL}
    Let $(\mathbf{B},\mathbf{c},\mathcal{Q})$ be an arbitrary Brascamp--Lieb datum and $0\le G_i < \infty$, $i=1,\ldots,m$. 
    Then it holds that 
    $$
    {\rm I}_{\mathbf{G} }^{(o)}(\mathbf{B},\mathbf{c},\mathcal{Q})
    = 
    {\rm I}_{\mathbf{G}}^{(\mathcal{G})}(\mathbf{B},\mathbf{c},\mathcal{Q}).  
    $$ 
    \end{theorem}
This may be seen as a generalization of \cite[Theorem 1.1]{Mil} as the evenness assumption on $f_i$ is now weakened to the centering assumption. See also \cite[Theorem 2.5]{NT3} for scalar matrices as the regularized parameter ${\bf G}$. 
Similarly, if one chooses $G_i = 0$ for $i=1,\ldots,m$, then this recovers Theorem \ref{t:NT3}, as well as  \cite[Theorems 2.5]{NT3}, with the weaker centering assumption on $f_i$. 

Before proceeding, we give a comment on the assumption of the log-concavity. 
Under the convention $\infty\times 0=0$, there exists an example of the Brascamp--Lieb datum for which the Gaussian saturation fails to hold even if one restricts inputs to even functions. 
\begin{example}
Let 
$
N=m=2, n_1=n_2=1, c_1=c_2=1,
$
and 
\begin{equation}\label{e:ChoiceBi}
    \mathcal{Q}=0,\quad B_1=B_2:\R^2\ni x\mapsto x_2 \in \R. 
\end{equation}
For such a datum, it holds that 
$$
\inf_{f_i: even} {\rm BL}(f_1,f_2) =0 \neq \inf_{f_i \in \mathcal{F}^{(e)}_{0,\infty}} {\rm BL}(f_1,f_2) = \inf_{A_i>0} {\rm BL}(g_{A_1},g_{A_2}) = \infty.
$$

%    Let us consider 
%    $$
%    {\rm BL}(\mathbf{f}):= \int_{\R^2} f_1(B_1x)f_2(B_2x)\, dx/ \int_{\R} f_1\, dx_1 \int_{\R} f_2\, dx_2,\quad B_1=B_2:\R^2\ni x\mapsto x_2 \in \R. 
%    $$
%    Then for $f_i\in L^1(\R)$: log-concave,  
%    $$
%    {\rm BL}(\mathbf{f})
%    = 
%    \int_{\R}\, dx_1\int_{\R} f_1(x_2)f_2(x_2)\, dx_2/ \int_{\R} f_1\, dx_1 \int_{\R} f_2\, dx_2
%    =\infty
%    $$
%    since $f_i(x_i)\ge c$ on some small nbh of 0. 
%    On the other hand, 
%    $$
%    f_1=\mathbf{1}_{[-1,1]},\quad f_2= \mathbf{1}_{[-2,-1]} + \mathbf{1}_{[1,2]} \quad\leadsto \quad
% {\rm BL}(\mathbf{f})=0.    $$
%    So, 
%    $$
%    \inf_{f:even} {\rm BL}(f) = 0 \neq \inf_{f:eve,log-concave} {\rm BL}(f) = \inf_{g:centered\, Gaussian} {\rm BL}(g) = \infty.
%    $$
\end{example}
To see this, note that ${\rm BL}(f_1,f_2)=\infty$ for all even and log-concave $f_i\in L^1(\R)$ since such a function is strictly positive on a neighborhood of the origin. 
On the other hand, if we take 
$$
f_1=\mathbf{1}_{[-1,1]},\quad f_2= \mathbf{1}_{[-2,-1]} + \mathbf{1}_{[1,2]} 
$$
then we see that ${\rm BL}(f_1,f_2)=0$ since $\infty\times 0=0$. 
Because of this example, although it is very artificial, one cannot expect the Gaussian saturation to hold universally in terms of the Brascamp--Lieb datum even if the inputs are restricted to even functions. 
However, we emphasize that this failure comes from just our convention that $\infty \times 0 =0$. 
%This example explains that, for the purpose of establishing the Gaussian saturation regardless of datum (in this framework), the evenness is not enough and some extra condition such as log-concavtiy is necessary.
%On the other hand, if one works on more restricted class of BL data, such as $\R^N = \bigoplus_i \R^{n_i}$, then there might be a chance 

\section{Proof of Theorem \ref{t:CenteredIBL}}
\subsection{Preliminaries}
The fundamental strategy of the proof of Theorem \ref{t:CenteredIBL} is parallel to the one of Theorem 1.3 in our previous work  \cite{NT3}, where all $f_i$ was supposed to be even. 
%As we explained in the beginning of section 2, 
 However, there are several technical problems to be rigorously justified. 
These problems emerge because a centered log-concave function is ``less-regular" than an even log-concave function.
Consequently, we will need to give more involved and complicated arguments. %to ensure for instance the change of the order of limit and integral and so on. 
The main tool to deal with these difficulties is the maximal bound of a centered log-concave function due to Fradelizi: 
\begin{lemma}[\cite{FraArchMath}]\label{l:Fradelizi}
For a log-concave $f \in L^1_+(\R^n)$ with $\int_{\R^n}xf\, dx=0$, 
$$
f(0) \le \| f \|_\infty \le e^n f(0). 
$$
\end{lemma}

By arming this lemma, we may establish the  pointwise bound of a centered log-concave function.

\begin{lemma}\label{l:Convexity}
Let $G, H \in \R^{n \times n}$ be symmetric matrices with $0< G \le H$ and $f=e^{-\varphi} \in \mathcal{F}_{G, H}^{(o)}(\R^n)$ with $\int_{\R^n} f\, dx=1$. 
Let also $\lambda>0$ be the smallest eigenvalue of $G$ and $\Lambda>0$ be the largest eigenvalue of $H$. 
Then the followings hold true: 
\begin{itemize}
\item[(1)] 
$$
\frac n2 \log \frac{\pi}{\Lambda} - n 
%\le \min_{x \in \R^n} \varphi(x) 
\le \varphi(0) \le \frac n2 \log \frac{4\pi}{\lambda} + 2n. 
$$

\item[(2)] 
For $x\in \R^n$, 
$$
\frac{\lambda}4 |x|^2 + \varphi(0) - 2n \le \varphi(x) \le \Lambda|x|^2 + \varphi(0) + n. 
$$
In particular, 
$$
\frac {\lambda}{4} |x|^2 + \frac n2 \log \frac{\pi}{\Lambda} - 3n \le \varphi(x) \le \Lambda |x|^2 + \frac n2 \log \frac{4\pi}{\lambda} + 3n. 
$$

\item[(3)] Fix $r>0$. Then there exists some $C_{n,r,\lambda, \Lambda}>0$ such that for any $x, y \in [-r, r]^n$, it holds that 
$$
|\varphi(x) - \varphi(y)| \le C_{n,r,\lambda, \Lambda} |x-y|. 
$$

\item[(4)] Fix $r>0$. Then there exists some $C_{n,r,\lambda, \Lambda}>0$ such that it holds that 
$$
\sup_{x \in [-r, r]^n} | \varphi(x)| \le C_{n,r,\lambda, \Lambda}. 
$$
\end{itemize}
\end{lemma}

\begin{proof}
The proof is almost parallel to the one of \cite[Lemma 2.2]{NT3}, so we give a proof of the statement (1) only. 
Note that $f$ is $\lambda {\rm id}_n$-uniformly log-concave and $\Lambda {\rm id}_n$-semi log-convex. 
Since $f$ is $\Lambda {\rm id}_n$-semi log-convex, it follows that 
$$
\varphi(0) = \varphi( \frac x2 + \frac{-x}{2}) \ge \frac12 \varphi(x) + \frac12 \varphi(-x) - \frac{\Lambda}{8} |x - (-x)|^2
=
\frac12 \varphi(x) + \frac12 \varphi(-x) - \frac{\Lambda}{2} |x|^2
$$
for $x\in \R^n$. 
Lemma \ref{l:Fradelizi} implies that 
$$
\varphi(-x) \ge \min_{y \in \R^n} \varphi(y) \ge \varphi(0) -n, 
$$
from which we have 
\begin{equation}\label{e:LowerBoundPhi}
\varphi(x) \le \Lambda |x|^2 + \varphi(0) + n.  
\end{equation}

On the other hand, since $f$ is $\lambda {\rm id}_n$-uniformly log-concave,  it holds that for any $x \in \R^n$ and $t \in (0,1)$, 
$$
\min_{y \in \R^n} \varphi(y) \le \varphi(t x) \le (1-t) \varphi(0) + t \varphi(x) - \frac {\lambda}{2}t(1-t) |x|^2. 
$$ 
It follows from Lemma \ref{l:Fradelizi} again that we have 
$$
\varphi(0) - n \le (1-t) \varphi(0) + t \varphi(x) - \frac {\lambda}{2}t(1-t) |x|^2, 
$$
which yields that, putting $t= \frac12$, 
\begin{equation}\label{e:UpperBoundPhi}
\varphi(x) \ge \frac {\lambda}{4} |x|^2 + \varphi(0) - 2n. 
\end{equation}
Combining \eqref{e:LowerBoundPhi} and \eqref{e:UpperBoundPhi}, we obtain 
$$
e^{- \Lambda |x|^2 - \varphi(0) - n} \le f(x) \le e^{- \frac \lambda4 |x|^2 - \varphi(0) + 2n}. 
$$
Applying $\int_{\R^n}f\, dx=1$, we conclude that  
$$
\left( \frac{\pi}{\Lambda} \right)^{\frac n2} e^{ - \varphi(0) - n}
\le
1
\le
\left( \frac{4\pi}{\lambda} \right)^{\frac n2} e^{ - \varphi(0) + 2n}, 
$$
which yields the desired assertion. 
%(2) The desired assertion immediately follows from combining (1) with \eqref{e:LowerBoundPhi} and \eqref{e:UpperBoundPhi}. 
%(3) This assertion follows by the same argument as in \cite[Lemma 2.2(3)]{NT3}, and so we omit its proof here. 
% (3) Let us fix different points $x, y \in [-r, r]^n$, and put $\ell(t)=(1-t)x + ty \in \mathbb{R}^n$ for $t \in \mathbb{R}$. 
% We may take  $t_1>1$ with $\ell(t_1) \in [-2r, 2r]^n \setminus  [-r,r]^n$ and take $t_2>t_1$ with $\frac1{|x-y|}=t_2 - t_1$. Since $\varphi \circ \ell$ is convex, it holds that 
% \begin{align*}
% \varphi(y) - \varphi(x)
% &=
% \varphi(\ell(1)) - \varphi(\ell(0))
% \le
% \frac{\varphi(\ell(t_2)) - \varphi(\ell(t_1))}{t_2 - t_1}
% \\
% &\le
% \frac{\frac \Lambda2 |\ell(t_2)|^2 + \frac n2 \log \frac{2\pi}{\lambda} - \frac n2 \log \frac{2\pi}{\Lambda}}{t_2-t_1}
% \\
% &\le
% \frac{\Lambda |\ell(t_2)- \ell(t_1)|^2 + \Lambda | \ell(t_1)|^2 + \frac n2 \log \frac{2\pi}{\lambda} - \frac n2 \log \frac{2\pi}{\Lambda}}{t_2-t_1}
% \\
% &=
% \frac{\Lambda |t_2-t_1|^2 |y-x|^2 + \Lambda | \ell(t_1)|^2 + \frac n2 \log \frac{2\pi}{\lambda} - \frac n2 \log \frac{2\pi}{\Lambda}}{t_2-t_1}
% \\
% &\le
% C_{n,r, \lambda, \Lambda}|y-x|, 
% \end{align*}
% where the second inequality follows from (2), and the last inequality follows from $\ell(t_1) \in [-2r, 2r]^n$ and $t_2-t_1 = \frac1{|x-y|}$. 
% Similarly we obtain $\varphi(x) - \varphi(y) \le C_{n,r,\lambda,\Lambda} |y-x|$, and thus we conclude the desired assertion. 
%(4) This assertion immediately follows from (2). 
\end{proof}

Before proving Theorem \ref{t:CenteredIBL}, we first need to identify the Brascamp--Lieb datum $(\mathbf{B},\mathbf{c},\mathcal{Q})$ for which ${\rm I}_{0,\infty}^{(o)}(\mathbf{B},\mathbf{c},\mathcal{Q})$ becomes trivial; either $0$ or $\infty$. 
For this purpose, we decompose $\mathcal{Q}$ into its positive and negative parts. 
That is, we first diagonalize $\mathcal{Q}$ by $ U \mathcal{Q} U^* = {\rm diag}\, (q_1,\ldots, q_N)$ for some appropriate $U \in O(N)$, where we arrange so that 
$$
q_1,\ldots, q_{N_+} > 0,\quad 
q_{N_++1},\ldots, q_{N_--1} = 0,\quad 
0 > q_{N_-},\ldots, q_N 
$$
for some $0\le N_+ < N_- \le N$.  
By denoting $u_j:= U^*e_j \in \mathbb{S}^{N-1}$, we obtain a decomposition 
$
\mathcal{Q}= \mathcal{Q}_+ - \mathcal{Q}_-,
$
where 
$$
\mathcal{Q}_+ := \sum_{j=1}^{N_+} q_j u_j \otimes u_j,
\quad 
\mathcal{Q}_- := \sum_{j= N_-}^N |q_j| u_j \otimes u_j. 
$$
If we denote 
$$
E_+:= \langle u_1,\ldots, u_{N_+}\rangle,\; E_0:= \langle u_{N_++1},\ldots, u_{N_--1} \rangle,\; E_-:= \langle u_{N_-},\ldots, u_N\rangle 
$$ 
then the whole space may be orthogonally decomposed into 
\begin{equation}\label{e:QuadDecomp}
\R^N = E_+ \oplus E_0 \oplus E_-,\quad 
\mathcal{Q}_{\pm} >0\quad {\rm on}\quad E_{\pm}. 
\end{equation}
With this notation, the first condition that excludes the trivial case may be written as 
\begin{equation}\label{e:FiniteCond}
    \mathcal{Q}_{\bigcap_{i=1}^m{\rm Ker}\, B_i} <0. 
\end{equation}
Here, $\mathcal{Q}_V<0$ means that 
$$
\langle x, \mathcal{Q}x\rangle <0,\quad \forall x \in V\setminus\{0\}.
$$
In particular, we understand that $\mathcal{Q}|_{\{0\}}<0$ trivially holds. 
We note that this condition the same as one of nondegeneracy condition of Barthe and Wolff \cite{BW}.

\begin{lemma}\label{l:Const=Infty}
    If the Brascamp--Lieb datum $(\mathbf{B},\mathbf{c},\mathcal{Q})$ does not satisfy \eqref{e:FiniteCond}, then 
    $$
    {\rm I}^{(\mathcal{G})}_{0,\infty}(\mathbf{B},\mathbf{c},\mathcal{Q}) 
    = 
    {\rm I}^{(o)}_{0,\infty}(\mathbf{B},\mathbf{c},\mathcal{Q}) 
    =
    +\infty. 
    $$
\end{lemma}

\begin{proof}
    Clearly $ {\rm I}^{(\mathcal{G})}_{0,\infty}(\mathbf{B},\mathbf{c},\mathcal{Q}) 
    \ge 
    {\rm I}^{(o)}_{0,\infty}(\mathbf{B},\mathbf{c},\mathcal{Q})  $ from the definition. 
    It thus suffices to show that 
    $$
    {\rm BL}(\mathbf{f})=+\infty,\quad \forall f_i \in \mathcal{F}^{(o)}_{0,\infty}(\R^{n_i}). 
    $$
    We know from the assumption that $V:=\bigcap_{i=1}^m {\rm ker}\, B_i \subset \R^N$ is non-trivial i.e. ${\rm dim}\, V \ge 1$; otherwise $\mathcal{Q}_V = \mathcal{Q}_{\{0\}} <0$ holds in the trivial sense. 
    With this in mind, we decompose $x = x'+x''$ where $x'\in V$ and $x'' \in V^\perp$ for which we have $B_i x = B_i x''$, and thus 
\begin{align*}
    \int_{\R^N} 
    e^{\langle x, \mathcal{Q}x\rangle} \prod_{i=1}^m f_i(B_ix)^{c_i}\, dx 
    &= 
    \int_{V^\perp}  \big( \int_{V} e^{\langle x, \mathcal{Q}x\rangle}\, dx' \big) \prod_{i=1}^m f_i(B_i x'')^{c_i} \, dx'', 
\end{align*}
where we use the convention $\infty \times 0 =0$. 
    Since we are supposing that $\mathcal{Q}|_V \not< 0$, there exists a direction $\omega_0 \in V\setminus\{0\}$ such that $\langle \omega_0, \mathcal{Q}\omega_0\rangle\ge0$. Hence, it follows that 
    $
    \int_{V} e^{\langle x, \mathcal{Q}x\rangle}\, dx' = + \infty. 
    $
    This gives 
    $$
    {\rm BL}(\mathbf{f})
    = 
    \int_{\R^N} 
    e^{\langle x, \mathcal{Q}x\rangle} \prod_{i=1}^m f_i(B_ix)^{c_i}\, dx = + \infty
    $$ unless $\prod_{i=1}^m f_i(B_i x'')^{c_i} = 0$ a.e. $x''\in V^\perp$. 
    The log-concavity of $f_i$ ensures that the latter scenario indeed does not occur. 
    To see this, notice from $f_i \in \mathcal{F}^{(o)}_{0,\infty}(\R^{n_i})$ and  Lemma \ref{l:Fradelizi} that, for $x'' =0$, 
    $\prod_{i=1}^m f_i(0)^{c_i} > 0$. 
    Moreover, we may see from $f_i\in \mathcal{F}^{(o)}_{0,\infty}(\R^{n_i})$ that 
    there exists a small neighborhood around 0, denoted by $U_i = U_i(f_i) \subset \R^{n_i}$ such that 
    \begin{equation}\label{e:Continuity-f_i}
        \inf_{x_i \in U_i} f_i(x_i) \ge \frac12 f_i(0). 
    \end{equation} 
    Once we could see this, we would complete the proof. 
    To ensure the existence of such a set, we may use Lemma \ref{l:Fradelizi}, $0\in {\rm int}\big( {\rm supp}\, h \big)$, and $h$ is continuous on ${\rm int}\big( {\rm supp}\, h \big)$. 
    Among them the second claim that $0\in {\rm int}\big( {\rm supp}\, f_i \big)$ may not be trivial, although this could be well-known.  We thus give a short proof of it. 
    Suppose contradictory that $0\notin {\rm int}\big( {\rm supp}\, f_i \big)$. 
    From the log-concavity of $f_i$, ${\rm int}\big( {\rm supp}\, f_i \big)$ is open and convex. 
    Moreover, ${\rm int}\big( {\rm supp}\, f_i \big)\neq \emptyset $; otherwise $\int_{\R^{n_i}} f_i = 0$. 
    Thus Hahn--Banach's separation theorem ensures the existence of a hyperplane separating ${\rm int}\big( {\rm supp}\, f_i \big)$ and $\{0\}$. 
    More precisely, there exist $u \in \mathbb{S}^{n-1}$ and $t\in \R$ such that 
    $$
    \langle u,x \rangle < t\le 0,\quad \forall x\in {\rm int}\big( {\rm supp}\, f_i \big).
    $$
    In particular, $\langle u,x\rangle <0$ and hence 
    \begin{align*}
        \big\langle u, \int_{\R^n} x f_i(x)\, dx \big\rangle 
        = 
        \int_{\R^n} \langle u,x\rangle f_i(x)\, dx <0
    \end{align*}
    since $h$ is a nonzero function. However, this contradicts with that $h$ is centered. 
\end{proof}

\begin{remark}
    In the above argument, we used the log-concavity assumption in a crucial way. 
    One may wonder whether the log-concavity is essential here or not. Namely, is the condition \eqref{e:FiniteCond} still necessary for 
    $$
    {\rm inf}_{f_i \in L^1_+(\R^{n_i}):\; \int x_if_i\, dx_i =0} {\rm BL}(\mathbf{f}) 
    <+\infty? 
    $$
    For this question, the answer is in general negative as Example 2.5 shows. %\footnote{To be fair, we note that Example 2.5 is a consequence from the convention $\infty \times 0 =0$. }. 
\end{remark}

We next confirm the reverse implication. 
\begin{lemma}\label{l:Const-Finite}
    Let $(\mathbf{B},\mathbf{c},\mathcal{Q})$ be the Brascamp--Lieb datum satisfying  \eqref{e:FiniteCond}. 
    Then there exists $\Lambda_0 = \Lambda_0(\mathbf{B},\mathbf{c},\mathcal{Q})\gg1$ such that 
    \begin{equation}\label{e:FiniteGauss}
    \int_{\R^N} e^{\langle x, \mathcal{Q}x\rangle} \prod_{i=1}^m g_{\Lambda_0 {\rm id}_{n_i}} (B_ix)^{c_i}\, dx <+\infty.  
    \end{equation} 
    In particular, 
    $$
    {\rm I}^{(o)}_{0,\infty}(\mathbf{B},\mathbf{c},\mathcal{Q})
    \le 
    {\rm I}^{(\mathcal{G})}_{0,\infty}(\mathbf{B},\mathbf{c},\mathcal{Q})
    <+\infty. 
    $$
\end{lemma}

\begin{proof}
    Suppose contradictory that 
    $$
    \forall \Lambda\gg1,\quad 
    \int_{\R^N} e^{\langle x, \mathcal{Q}x\rangle} \prod_{i=1}^m g_{\Lambda {\rm id}_{n_i}} (B_ix)^{c_i}\, dx = +\infty.  
    $$
    Then $ \Lambda \sum_{i=1}^m c_i B_i^* B_i \not> 2\mathcal{Q} $, that is, 
    $$
    \exists \omega_\Lambda \in \mathbb{S}^{N-1}:\; 
    \Lambda \sum_{i=1}^m c_i |B_i \omega_\Lambda|^2 \le 2 \langle \omega_\Lambda, \mathcal{Q} \omega_{\Lambda} \rangle. 
    $$
    By $\mathcal{Q} = \mathcal{Q}_+-\mathcal{Q}_-$, this may be read as 
    \begin{equation}\label{e:Ineq3/9-1}
    \sum_{i=1}^m c_i |B_i \omega_\Lambda|^2 + \frac1{\Lambda} \langle \omega_\Lambda, \mathcal{Q}_- \omega_\Lambda\rangle 
    \le 
     \frac1{\Lambda} \langle \omega_\Lambda, \mathcal{Q}_+ \omega_\Lambda\rangle. 
    \end{equation}
    We will take a limit $\Lambda\to\infty$. For this purpose, we note that there exists $\omega_\infty\in\mathbb{S}^{N-1}$ such that 
    $$
    \lim_{\Lambda\to\infty} 
    |\omega_\infty - \omega_\Lambda|
    =0, 
    $$
    because of the compactness of $\mathbb{S}^{N-1}$ (after passing to the subsequence if necessary). 
    On the one hand, \eqref{e:Ineq3/9-1} yields that 
    $$
    \sum_{i=1}^m c_i |B_i\omega_\Lambda|^2 \le \frac1{\Lambda} \lambda_{\rm max}(\mathcal{Q}_+),
    $$
    and so after $\Lambda\to \infty$, we see that $\omega_\infty \in \bigcap_{i=1}^m {\rm ker}\, B_i $. 
    Since $\mathcal{Q}_->0$ on $\bigcap_{i=1}^m {\rm ker}\, (B_i)$ from \eqref{e:FiniteCond}, this yields that $\langle \omega_\infty, \mathcal{Q}_- \omega_\infty\rangle >0$. 
    On the other hand, \eqref{e:Ineq3/9-1} also gives that $\langle \omega_\Lambda, \mathcal{Q}_- \omega_\Lambda\rangle 
    \le 
    \langle \omega_\Lambda, \mathcal{Q}_+ \omega_\Lambda\rangle$, and so after $\Lambda\to\infty$, 
    $$
    \langle \omega_\infty, \mathcal{Q}_- \omega_\infty\rangle 
    \le 
    \langle \omega_\infty, \mathcal{Q}_+ \omega_\infty\rangle. 
    $$
    By combining this and $\langle \omega_\infty, \mathcal{Q}_- \omega_\infty\rangle >0$, it follows that $\langle \omega_\infty, \mathcal{Q}_+ \omega_\infty\rangle >0$ which means that $\omega_\infty \in E_- \cap E_+$. This is a contradiction since $\omega_\infty \neq0$.  
\end{proof}

The next condition to exclude the trivial case is the surjectivity of $B_i$. 
\begin{lemma}\label{l:Surjective}
    Let $(\mathbf{B},\mathbf{c},\mathcal{Q})$ be the Brascamp--Lieb datum satisfying \eqref{e:FiniteCond}. 
    If $B_{i_0}$ is not surjective onto $\R^{n_{i_0}}$ for some $i_0$ then 
    $$
    0 = 
    {\rm I}^{(\mathcal{G})}_{0,\infty} (\mathbf{B},\mathbf{c},\mathcal{Q}) = {\rm I}^{(o)}_{0,\infty}(\mathbf{B},\mathbf{c},\mathcal{Q}). 
    $$
\end{lemma}
\begin{proof}
    We may suppose $i_0 = 1$ without loss of generality. 
    Since $B_1:\R^N\to \R^{n_1}$ is not surjective, if we decompose 
    $$
    \R^{n_1}
    = 
    {\rm Im}\, B_1 \oplus \big( {\rm Im}\, B_1 \big)^\perp 
    $$ then the subspace $ \big( {\rm Im}\, B_1 \big)^\perp $ is non-trivial i.e. 
    $
    \exists \omega_1\in  \big( {\rm Im}\, B_1 \big)^\perp  \cap  \mathbb{S}^{n_1-1}. 
    $
    In particular, we have that 
    \begin{equation}\label{e:omega_1-3/9}
        \langle B_1x, \omega_1\rangle = 0,\; \forall x \in \R^N, \quad {\rm Im}\, B_1 \subset \langle \omega_1\rangle^\perp.
    \end{equation}
    We now take a Gaussian input. 
    For this, we first notice from the assumption \eqref{e:FiniteCond} and Lemma \ref{l:Const-Finite} that 
    $$
    \exists \Lambda_0\gg1:\; {\rm BL}(\gamma_{\Lambda_0^{-1}{\rm id}_{n_1}},\ldots,\gamma_{\Lambda_0^{-1}{\rm id}_{n_m}}) <+\infty. 
    $$
    With this in mind, we define $n_1 \times n_1$ symmetric matrix $\Sigma_1 = \Sigma_1(\varepsilon)$ by 
    $$
    \Sigma_1^{-1}:= 
    \Lambda_0 P_{ \omega_1^{\perp}}^*P_{ \omega_1^{\perp}} + \varepsilon \omega_1\otimes \omega_1, 
    $$
    where $\omega_1^\perp$ means the subspace of $\R^{n_1}$ that is orthogonal to $\omega_1$, and $P_{\omega_1^\perp}$ is a projection onto $\omega_1^\perp$. 
    Clearly $\Sigma_1>0$ on $\R^{n_1}$ and 
    \begin{align*}
        \gamma_{\Sigma_1}(B_1x)
        &=
        \big( {\rm det}\, ( 2\pi (\Lambda_0^{-1} P_{ \omega_1^{\perp}}^*P_{ \omega_1^{\perp}} + \varepsilon^{-1} \omega_1\otimes \omega_1 ) \big)^{-\frac12} e^{-\frac12 \langle B_1x, \big( \Lambda_0 P_{ \omega_1^{\perp}}^*P_{ \omega_1^{\perp}} + \varepsilon \omega_1\otimes \omega_1 \big) B_1x\rangle} \\
        &= 
        (2\pi \Lambda_0^{-1})^{-\frac{n_1-1}2}
        (2\pi \varepsilon^{-1})^{-\frac12}
        e^{- \frac{\Lambda_0}2 | P_{\omega_1^\perp} B_1 x |^2 } 
        e^{-\frac{\varepsilon}2 | \langle \omega_1,B_1x\rangle |^2 }. 
    \end{align*}
    But we know from \eqref{e:omega_1-3/9} that $\langle \omega_1,B_1x\rangle = 0$ and that $P_{\omega_1^\perp} B_1x = B_1x$ from which 
    $$
    \gamma_{\Sigma_1}(B_1x)
    = 
    (2\pi \varepsilon^{-1})^{-\frac12}
    (2\pi \Lambda_0^{-1})^{-\frac{n_1-1}2}
    e^{- \frac{\Lambda_0}2 | B_1 x |^2 }.  
    $$
    By definition, we also know that 
    $$
    \gamma_{\Lambda_0^{-1}{\rm id}_{n_1}}(B_1x)
    = 
    (2\pi \Lambda_0^{-1})^{-\frac12}
    (2\pi \Lambda_0^{-1})^{-\frac{n_1-1}2}
    e^{- \frac{\Lambda_0}2 | B_1 x |^2 }
    $$
    and so 
    $$
    \gamma_{\Sigma_1}(B_1x)
    = 
    \sqrt{\frac{\varepsilon}{\Lambda_0}} \gamma_{\Lambda_0^{-1}{\rm id}_{n_1}}(B_1x). 
    $$
    For other $i\ge2$, we simply let $\Sigma_i:= \Lambda_0^{-1}{\rm id}_{n_i}$. Then 
    \begin{align*}
        {\rm BL}( \gamma_{\Sigma_1}(\varepsilon), \gamma_{\Sigma_2} \ldots, \gamma_{\Sigma_m})
        %= 
        %{\rm BL}( \gamma_{\Sigma_1},\ldots, \gamma_{\Sigma_m} )
        = 
        \big( \frac{\varepsilon}{\Lambda_0} \big)^{\frac{c_1}2} 
        {\rm BL}(\gamma_{\Lambda_0^{-1}{\rm id}_{n_1}},\ldots,\gamma_{\Lambda_0^{-1}{\rm id}_{n_m}}),
    \end{align*}
    and thus 
    \begin{align*}
        {\rm I}^{(\mathcal{G})}_{0,\infty}
        &\le 
        \lim_{\varepsilon\to0} 
        {\rm BL}( \gamma_{\Sigma_1}(\varepsilon), \gamma_{\Sigma_2}, \ldots, \gamma_{\Sigma_m})\\
        &= 
        \lim_{\varepsilon\to0}
        \big( \frac{\varepsilon}{\Lambda_0} \big)^{\frac{c_1}2} 
        {\rm BL}(\gamma_{\Lambda_0^{-1}{\rm id}_{n_1}},\ldots,\gamma_{\Lambda_0^{-1}{\rm id}_{n_m}}) = 0
    \end{align*}
    as we know that ${\rm BL}(\gamma_{\Lambda_0^{-1}{\rm id}_{n_1}},\ldots,\gamma_{\Lambda_0^{-1}{\rm id}_{n_m}}) < +\infty$ from the choice of $\Lambda_0$.

\end{proof}

%{\color{red}
%\begin{remark}
%    If one considers the regularized constant ${\rm I}^{(\mathcal{G})}_{\mathbf{G}, \mathbf{H}}$ for some $0< G_i$ and $H_i<\infty$, then we cannot take the limit $\varepsilon\to0$ in the above argument, and so do not known whether the surjectivity is necessary or not.
%\end{remark}
%}

\subsection{The Gaussian saturation in the case of $0< G_i \le H_i <\infty$}
We begin with the case of $0< G_i \le H_i <\infty$. 
The goal of this subsection is to prove the following: 
\begin{proposition}\label{p:RegIBL}
Let $(\mathbf{B},\mathbf{c},\mathcal{Q})$ be an arbitrary Brascamp--Lieb datum, and $0 < G_i \le H_i<\infty$ for $i=1,\ldots,m$. Then 
$$
{\rm I}_{{\bf G}, {\bf H}}^{(o)} ({\bf B}, {\bf c}, \mathcal{Q})
=
{\rm I}_{{\bf G}, {\bf H}}^{(\mathcal{G})} ({\bf B}, {\bf c}, \mathcal{Q}). 
$$
%$$
%\lim_{\lambda \downarrow 0, \Lambda \uparrow + \infty} {\rm IBL}_{\mathcal{F}_{\lambda, \Lambda}}^{(o)} ({\bf n}, {\bf c}, {\bf Q})
%=
%{\rm IBL}_{{\bf g}} ({\bf c}). 
%$$
\end{proposition}

\begin{remark}
In fact, the proof of this proposition works well and the same conclusion is true even when  $c_i \in \mathbb{R}\setminus\{0\}$. 
\end{remark}

As in the proof of \cite[Theorem 2.3]{NT3}, the proof of Proposition \ref{p:RegIBL} is decomposed into three steps; (Step 1) the existence of the extremizer, (Step 2) the monotonicity-type result of the inequality under the self-convolution, (Step 3) the iteration of the monotonicity-type result. 
Once we have obtained Lemma \ref{l:Convexity}-\textit{(3), (4)},  we may establish the Step 1 by the exactly same argument as \cite[Theorem 2.1]{NT3}. 
So, we just give a statement here and refer \cite{NT3} for more details.  
\begin{lemma}\label{l:Extremiser}
Let $(\mathbf{B},\mathbf{c},\mathcal{Q})$ be an arbitrary Brascamp--Lieb datum, and $0< G_i \le H_i  <\infty$. 
Then ${\rm I}_{{\bf G}, {\bf H}}^{(o)}({\bf B}, {\bf c}, \mathcal{Q})$ is extremizable in the sense that 
$$
\exists \mathbf{f} \in \mathcal{F}^{(o)}_{G_1,H_1}(\R^{n_1}) \times \cdots \times \mathcal{F}^{(o)}_{G_m,H_m}(\R^{n_m}):\; 
{\rm I}_{{\bf G}, {\bf H}}^{(o)}({\bf B}, {\bf c}, \mathcal{Q})
= 
{\rm BL}(\mathbf{f}). 
$$
%{\color{red}
%${\rm I}_{\mathbf{G},\mathbf{H}}^{(o)} = \infty$の場合も含んでいるという理解で良い？
%}
\end{lemma}

\begin{remark}
In this lemma, ${\rm I}_{{\bf G}, {\bf H}}^{(o)}({\bf B}, {\bf c}, \mathcal{Q})=\infty$ might be allowed. In this case, any function in $\mathcal{F}^{(o)}_{G_1,H_1}(\R^{n_1}) \times \cdots \times \mathcal{F}^{(o)}_{G_m,H_m}(\R^{n_m})$ is an extremizer. 
\end{remark}

The Step 2 amounts to the following monotonicity-type result that we refer as Ball's inequality
%\footnote{The origin of this type of inequality may be traced to the work of Barthe \cite{BarGAFA, BarInvMath}, in which the author mentioned that it was communicated to him by K. Ball 
% {\color{red}
% \cite{BarInvMath}では maximizer $f_i$とGaussian maximizerのconvolutionがまたmaximizerになるという話をしているので若干違うのでは？今回の場合はGaussianとのconvolutionでは閉じない
% }
%} 
following \cite{BBBCF, BCCT}. 
\begin{lemma}\label{l:Monotone*}
Let $(\mathbf{B},\mathbf{c},\mathcal{Q})$ be arbitrary Brascamp--Lieb datum and $0< G_i\le H_i<\infty$.  
For any ${\bf f} \in \mathcal{F}_{G_1, H_1,}^{(o)}(\R^{n_1}) \times \cdots \times \mathcal{F}_{G_m, H_m}^{(o)}(\R^{n_m})$, 
$$
{\rm BL}({\bf f})^2 
\ge
{\rm I}_{{\bf G}, {\bf H}}^{(o)} ({\bf B}, {\bf c}, \mathcal{Q}) {\rm BL}( 2^{\frac{n_1}{2}}f_1 \ast f_1(\sqrt{2} \cdot), \dots, 2^{\frac{n_m}{2}}f_m \ast f_m(\sqrt{2} \cdot)). 
$$
%Here ${\bf f} \ast {\bf f} \coloneqq ( f_1 \ast f_1, \dots, f_m \ast f_m)$. 
\end{lemma}

\begin{proof}
We may suppose that $\int_{\R^{n_i}} f_i \, dx_i=1$. 
Put 
$$
F(x) \coloneqq e^{\langle x, \mathcal{Q} x\rangle} \prod_{i=1}^m f_i(B_ix)^{c_i}, \;\;\; x \in \R^N, 
$$
and observe that 
$$
\int_{\R^N} F \ast F\, dx = {\rm BL}({\bf f})^2. 
$$
Hence changing of variable reveals that 
\begin{align*}
&%\int_{\R^N} F \ast F\, dx
{\rm BL}({\bf f})^2 
%\\
%=&
%\int_{\R^N} \int_{\R^N} e^{\langle y, \mathcal{Q} y\rangle} \prod_{i=1}^m f_i(B_iy)^{c_i} e^{\langle x-y, \mathcal{Q} (x-y)\rangle} \prod_{i=1}^m f_i(B_ix - B_iy)^{c_i} \, dy\, dx
%\\
%=&
%2^{\frac{N}{2}} \int_{\R^N} \int_{\R^N} e^{\langle y, \mathcal{Q} y\rangle} \prod_{i=1}^m f_i(B_iy)^{c_i} e^{\langle \sqrt{2}x-y, \mathcal{Q} (\sqrt{2}x-y)\rangle} \prod_{i=1}^m f_i(\sqrt{2} B_ix - B_iy)^{c_i} \, dy\, dx
%\\
%=&
%\int_{\R^N} \int_{\R^N} e^{\langle \frac{x+y}{\sqrt{2}}, \mathcal{Q} \frac{x+y}{\sqrt{2}} \rangle} \prod_{i=1}^m f_i( B_i\frac{x + y}{\sqrt{2}} )^{c_i} e^{\langle \frac{x-y}{\sqrt{2}}, \mathcal{Q} \frac{x-y}{\sqrt{2}} \rangle} \prod_{i=1}^m f_i( B_i \frac{x - y}{\sqrt{2}} )^{c_i} \, dy\, dx
%\\
=&
\int_{\R^N} e^{\langle x, \mathcal{Q}x \rangle } \int_{\R^N} e^{\langle y, \mathcal{Q}y \rangle } 
\prod_{i=1}^m \left(f_i\left( B_i \frac{x + y}{\sqrt{2}} \right) f_i\left(B_i  \frac{x - y}{\sqrt{2}} \right) \right)^{c_i} \, dy\, dx;
\end{align*}
we refer \cite[Proposition 2.4]{NT3} for more detailed calculation. 
If we let 
$$
F_i(y_i)
=
F_i^{(x)}(y_i)
\coloneqq f_i\left( \frac{B_ix + y_i}{\sqrt{2}} \right) f_i\left( \frac{B_ix - y_i}{\sqrt{2}} \right), \;\;\; y_i \in \R^{n_i}
$$
for each fixed $x \in \R^N$, then it is obviously even and particularly centered. 
Moreover,  we may see that $F_i$ is $G_i$-uniformly log-concave and $H_i$-semi log-convex, %\footnote{
%To see this,  we denote $f_i=e^{-\varphi_i}$ in which case $\psi_i(x_i) := \varphi_i(x_i) - \frac12\langle G_i x_i, x_i \rangle$  is convex. 
%Put for a fixed $x \in \R^N$, 
%$$
%\Psi_i(y_i) := \varphi_i( \frac{B_ix + y_i}{\sqrt{2}})+ \varphi_i( \frac{B_ix - y_i}{\sqrt{2}}) - \frac12\langle G_i y_i, y_i \rangle, \;\;\;  y_i \in \R^{n_i}, 
%$$
%and let us show that $\Psi_i$ is convex. 
%Thanks to the identity 
%\begin{align*}
%\psi_i(\frac{B_ix + y_i}{\sqrt{2}}) + \psi_i(\frac{B_ix-y_i}{\sqrt{2}})
%\\
%&=
%\varphi_i(\frac{B_ix + y_i}{\sqrt{2}}) + \varphi_i(\frac{B_ix-y_i}{\sqrt{2}})
%- \frac12 \langle G_i \frac{B_ix + y_i}{\sqrt{2}}, \frac{B_ix + y_i}{\sqrt{2}} \rangle 
%- \frac12 \langle G_i \frac{B_ix - y_i}{\sqrt{2}}, \frac{B_ix - y_i}{\sqrt{2}} \rangle 
%\\
%=
%\varphi_i(\frac{B_ix + y_i}{\sqrt{2}}) + \varphi_i(\frac{B_ix-y_i}{\sqrt{2}})
%- \frac12 \langle G_i B_ix, B_ix \rangle 
%- \frac12 \langle G_i y_i, y_i\rangle, 
%\end{align*}
%the convexity of $\Psi_i$ is equivalent to the convexity of 
%$$
%\R^{n_i} \ni y_i \mapsto \psi_i(\frac{B_ix + y_i}{\sqrt{2}}) + \psi_i(\frac{B_ix-y_i}{\sqrt{2}}) \in \R. 
%$$
%However the convexity of this function derives from the convexity of $\psi_i$, and thus we obtain the desired claim. 
%The similar argument also endures that $F_i$ is $H_i$-semi log-convex. 
%}, 
and thus $F_i \in \mathcal{F}^{(o)}_{G_i,H_i}(\R^{n_i})$. 
Therefore, it follows that 
\begin{align*}
&
\int_{\R^N} e^{\langle x, \mathcal{Q}x \rangle } \int_{\R^N} e^{\langle y, \mathcal{Q}y \rangle } 
\prod_{i=1}^m \left(f_i( \frac{B_ix + B_iy}{\sqrt{2}} ) f_i( \frac{B_ix - B_iy}{\sqrt{2}} ) \right)^{c_i} \, dy\, dx
\\
%=&
%\int_{\R^N} e^{\langle x, \mathcal{Q}x \rangle } \int_{\R^N} e^{\langle y, \mathcal{Q}y \rangle } 
%\prod_{i=1}^m F_i(B_iy)^{c_i} \, dy\, dx
%\\
\ge&
{\rm I}_{{\bf G}, {\bf H}}^{(o)} ({\bf B}, {\bf c}, \mathcal{Q})
\int_{\R^N} e^{\langle x, \mathcal{Q}x \rangle } \prod_{i=1}^m \left( \int_{\R^{n_i}} F_i \, dy_i \right)^{c_i}\, dx
\\
%=&
%{\rm I}_{{\bf Q}, {\bf P}, +}^{(o)} ({\bf B}, {\bf c}, \mathcal{Q})
%\int_{\R^N} e^{\langle x, \mathcal{Q}x \rangle } \prod_{i=1}^m \left( 2^{\frac {n_i}{2}} f_i \ast f_i(\sqrt{2} B_ix) \right)^{c_i}\, dx
%\\
=&
{\rm I}_{{\bf G}, {\bf H}}^{(o)} ({\bf B}, {\bf c}, \mathcal{Q})
{\rm BL}( 2^{\frac{n_1}{2}}f_1 \ast f_1(\sqrt{2} \cdot), \dots, 2^{\frac{n_m}{2}}f_m \ast f_m(\sqrt{2} \cdot)), 
\end{align*}
as we wished. 
\end{proof}

As the Step 3, let us conclude the proof of Proposition \ref{p:RegIBL} by combining above lemmas.
\begin{proof}[Proof of Proposition \ref{p:RegIBL}]
{Firstly we may suppose that ${\rm I}_{{\bf G}, {\bf H}}^{(o)} ({\bf B}, {\bf c}, \mathcal{Q}) < +\infty$; otherwise the assertion is evident.} 
Thanks to Lemma \ref{l:Extremiser}, there exists an extremizer ${\bf f} \in \mathcal{F}_{G_1, H_1}^{(o)}(\R^{n_1}) \times \cdots \times\mathcal{F}_{G_m, H_m}^{(o)}(\R^{n_m})$ such that 
$$
{\rm I}_{{\bf G}, {\bf H}}^{(o)} ({\bf B}, {\bf c}, \mathcal{Q})
=
{\rm BL} ({\bf B}, {\bf c}, \mathcal{Q} ; {\bf f}). 
$$
Without loss of generality, we may suppose that $\int_{\R^{n_i}} f_i \, dx_i=1$. 
We then apply Lemma \ref{l:Monotone*} to see that 
$$
{\rm BL}({\bf f})^2 
\ge
{\rm I}_{{\bf G}, {\bf H}}^{(o)} ({\bf B}, {\bf c}, \mathcal{Q}) {\rm BL}( 2^{\frac{n_1}{2}}f_1 \ast f_1(\sqrt{2} \cdot), \dots, 2^{\frac{n_m}{2}}f_m \ast f_m(\sqrt{2} \cdot)). 
$$
The next observation is that $2^{\frac{n_i}{2}}f_i \ast f_i(\sqrt{2} \cdot) \in \mathcal{F}_{G_i, H_i}^{(o)}(\R^{n_i})$ for $i=1, \dots, m$. 
For this,  it is evident from the definition that the self-convolution preserves the centering condition. 
The preservation of uniform log-concavity and semi log-convexity under the self-convolution is no longer trivial, but this may be confirmed by virtue of the Pr\'{e}kopa--Leindler inequality together with the observation due to Brascamp--Lieb \cite{BLJFA}; see \cite[Lemma 6.2]{NT3} for more details. 
Thus we may again apply Lemma \ref{l:Monotone*} for $2^{\frac{n_i}{2}}f_i \ast f_i(\sqrt{2} \cdot)$.  
By iterating this procedure, we obtain that 
\begin{align*}
{\rm I}_{{\bf G}, {\bf H}}^{(o)} ({\bf B}, {\bf c}, \mathcal{Q})^{2^k}
=&
{\rm BL} ({\bf f})^{2^k}
\\
\ge&
{\rm I}_{{\bf G}, {\bf H}}^{(o)} ({\bf B}, {\bf c}, \mathcal{Q})^{2^k-1} 
{\rm BL}( (2^k)^{\frac{n_1}{2}} f_1^{(2^k)}(2^{\frac{k}2} \cdot), \dots, (2^k)^{\frac{n_m}{2}}f_m^{(2^k)}(2^{\frac{k}2} \cdot)), 
\end{align*}
where 
\begin{equation}\label{e:IterateConv}
f_i^{(2^k)}
\coloneqq
\overbrace{f_i \ast \cdots \ast f_i}^{\text{$2^k$-times}}, \;\;\; i=1, \dots, m.
\end{equation}
Since $0<{\rm I}_{{\bf G}, {\bf H}}^{(o)} ({\bf B}, {\bf c}, \mathcal{Q}) <+\infty$ by Lemma \ref{l:Convexity},  it follows that 
$$
{\rm I}_{{\bf G}, {\bf H}}^{(o)} ({\bf B}, {\bf c}, \mathcal{Q})
\ge
{\rm BL}( (2^k)^{\frac{n_1}{2}} f_1^{(2^k)}(2^{\frac{k}2} \cdot), \dots, (2^k)^{\frac{n_m}{2}}f_m^{(2^k)}(2^{\frac{k}2} \cdot)). 
$$
Now the central limit theorem yields that $(2^k)^{\frac{n_i}{2}} f_i^{(2^k)}(2^{\frac{k}2} \cdot)$ converges to $\gamma_{\Sigma_i}$ as $k \to \infty$ in $L^1$ topology, and thus especially pointwisely $dx_i$-a.e. on $\R^{n_i}$, where $\Sigma_i$ is the covariance matrix of $f_i$. 
In view of $f_i \in \mathcal{F}_{G_i, H_i}^{(o)}(\R^{n_i})$,  so is  $\gamma_{\Sigma_i}$ i.e. $G_i \le \Sigma_i^{-1}\le H_i$.  
Hence, from Fatou's lemma,  we derive that 
\begin{align}\label{e:Iteration}
{\rm I}_{{\bf G}, {\bf H}}^{(o)} ({\bf B}, {\bf c}, \mathcal{Q})
\ge
&\liminf_{k \to \infty} {\rm BL}( (2^k)^{\frac{n_1}{2}} f_1^{(2^k)}(2^{\frac{k}2} \cdot), \dots, (2^k)^{\frac{n_m}{2}}f_m^{(2^k)}(2^{\frac{k}2} \cdot))
\\
\ge&
{\rm BL}( \gamma_{\Sigma_1}, \dots,  \gamma_{\Sigma_m}) 
\ge
{\rm I}_{{\bf G}, {\bf H}}^\mathcal{(G)}({\bf B}, {\bf c}, \mathcal{Q}). \nonumber 
\end{align}
Since the converse inequality is evident we conclude the proof. 
\end{proof}
{\begin{remark}
    In the above proof, we used the regularity assumption $0<G_i\le H_i <\infty$ to ensure the existence of the extremizer and positivity of the inverse Brascamp--Lieb constant only. 
    In fact, if one apriori knows these properties the above proof would work even without the regularity  assumption. 
    This point will be relevant when we investigate the case of equality in \eqref{e:OriginGCI}. 
\end{remark}
}
\subsection{The case of $G_i >0 $ and $ H_i = \infty$: a limiting argument}
By taking a limit $G_i\to0$ and $H_i \to \infty$ in Proposition \ref{p:RegIBL}, we may formally derive Theorem \ref{t:CenteredIBL}. 
However, making the argument rigorous requires further works with careful and complicated analysis.
This is because we only impose the centering assumption rather than evenness.  
We will thus take three steps towards completing the proof of Theorem \ref{t:CenteredIBL}: 
\begin{enumerate}
\item 
The case of $G_i>0$ and inputs $f_i \in \mathcal{F}^{(o)}_{G_i,\infty}$ are supposed to have a compact support. 
\item 
The case of $G_i>0$ and inputs $f_i \in \mathcal{F}^{(o)}_{G_i,\infty}$ are arbitrary.  
\item 
The case of $G_i \ge 0 $ and inputs  $f_i \in \mathcal{F}^{(o)}_{G_i,\infty}$ are arbitrary.  
\end{enumerate}

\begin{proposition}\label{p:IBL-CompactSupp}
    Let $(\mathbf{B},\mathbf{c},\mathcal{Q})$ be arbitrary Brascamp--Lieb datum and $0< G_i$ for $i=1, \dots, m$. 
    Then for any $f_i \in \mathcal{F}^{(o)}_{G_i, \infty}(\R^{n_i})$ which is compactly supported, we have that 
    $$
    \int_{\mathbb{R}^N} e^{\langle x,\mathcal{Q}x\rangle} 
    \prod_{i=1}^m f_i(B_ix)^{c_i}\, dx 
    \ge 
    {\rm I}^{(\mathcal{G})}_{\mathbf{G}}(\mathbf{B},\mathbf{c}, \mathcal{Q}) 
    \prod_{i=1}^m \left( \int_{\mathbb{R}^{n_i}} f_i\, dx_i \right)^{c_i}. 
    $$
\end{proposition}

\begin{proof}
%{\color{red}NOT YET COMPLETED!!}

First note that we may assume \eqref{e:FiniteCond}; otherwise there is nothing to prove from Lemma \ref{l:Const=Infty}. 
For $i=1, \dots, m$ and $t>0$, we evolve $f_i$ by the $\beta$-Fokker--Planck flow: 
$$
f_i^{(t)}(x_i)
=
\gamma_{\beta(1-e^{-2t}) {\rm id}_{n_i}} \ast (e^{n_it}f_i)(e^t x_i)
=
\int_{\mathbb{R}^{n_i}} e^{ - \frac{|x_i- e^{-t}y_i|^2}{2\beta(1-e^{-2t})}} \, \frac{f_i(y_i) \,dy_i}{(2\pi\beta(1-e^{-2t}))^{n_i/2}}, 
$$
where $\beta>0$ is a fixed constant such that $\beta^{-1} {\rm id}_{n_i} \ge G_i$ for all $i=1, \dots, m$. 
Let us first note that $\int_{\R^{n_i}} f_i^{(t)}\, dx_i= \int_{\R^{n_i}} f_i\, dx_i$ and $\int_{\R^{n_i}} x_if_i^{(t)}\, dx_i=0$ for any $t>0$ by definition. 
As is well-known, the $\beta$-Fokker--Planck flow preserves the $G_i$-uniformly log-concavity as long as $\beta^{-1} {\rm id}_n \ge G_i$; see \cite[Theorem 4.3]{BLJFA}.
Thus, $f_i^{(t)}$ is $G_i$-uniformly log-concave. 
Also the Li--Yau inequality provides a gain of the log-convexity, that is, $f_i^{(t)}$ becomes $H_i(t)$-semi log-convex, and thus $f_i^{(t)} \in \mathcal{F}_{G_i, H_i(t)}^{(o)}(\R^{n_i})$, where $H_i(t) := (\beta (1-e^{-2t}))^{-1}\, {\rm id}_n$. 
Therefore we may apply Proposition \ref{p:RegIBL} to see that 
$$
\int_{\mathbb{R}^N} e^{\langle x, \mathcal{Q}x\rangle } \prod_{i=1}^m f_i^{(t)} (B_ix)^{c_i} \, dx \ge {\rm I}_{\mathbf{G},\mathbf{H}(t)}^{(\mathcal{G})}(\mathbf{B}, \mathbf{c},\mathcal{Q}) \prod_{i=1}^m \left( \int_{\mathbb{R}^{n_i}} f_i\, dx_i \right)^{c_i}. 
$$
Here we also used the mass-conservation $\int_{\mathbb{R}^{n_i}} f_i^{(t)}\, dx_i = \int_{\mathbb{R}^{n_i}} f_i\, dx_i$ on the right-hand side. 

It thus suffices to ensure the change of the order of $\lim_{t\to 0}$ and the integration: 
\begin{equation}\label{e:ChangeLimInt'}
    \lim_{t \to 0} 
    \int_{\mathbb{R}^N} e^{\langle x, \mathcal{Q}x\rangle } \prod_{i=1}^m f_i^{(t)}(B_ix)^{c_i}  \, dx
    = 
    \int_{\mathbb{R}^N} e^{\langle x, \mathcal{Q}x\rangle } \prod_{i=1}^m f_i(B_ix)^{c_i}\, dx,
\end{equation}
and 
\begin{equation}\label{e:ChangeLimGaussconst}
    \lim_{t \to 0} 
    {\rm I}_{\mathbf{G},\mathbf{H}(t)}^{(\mathcal{G})}(\mathbf{B}, \mathbf{c},\mathcal{Q})
    \ge 
    {\rm I}_{\mathbf{G},\infty}^{(\mathcal{G})}(\mathbf{B}, \mathbf{c},\mathcal{Q}). 
\end{equation}
Note that since $H_i(t) \le \infty$, \eqref{e:ChangeLimGaussconst} is evident, and thus let us show \eqref{e:ChangeLimInt'}. 
To see this, let $\Lambda (t) := e^{-2t}(\beta (1-e^{-2t}))^{-1}$. 
Note that $t\to 0$ is equivalent to $\Lambda(t) \to \infty$. 
Moreover changing of variables yields that 
$$
\int_{\mathbb{R}^N} e^{\langle x, \mathcal{Q}x\rangle } \prod_{i=1}^m f_i^{(t)}(B_ix)^{c_i}  \, dx
    =
    e^{-Nt + t\sum_{i=1}^m c_i n_i}
    \int_{\mathbb{R}^N} e^{e^{-2t}\langle x, \mathcal{Q}x\rangle } \prod_{i=1}^m (f_i)_{\Lambda(t)} (B_ix)^{c_i}  \, dx, 
$$
where 
$$
(f_i)_{\Lambda }(x_i)
=
\frac{1}{(2\pi/\Lambda)^{n_i/2}} \int_{\mathbb{R}^{n_i}} e^{ - \frac{\Lambda}2|x_i-y_i|^2} f_i(y_i)\, dy_i,\quad \Lambda>0. 
$$
Therefore it suffices to justify that 
\begin{equation}\label{e:ChangeLimInt}
    \lim_{t \to 0} 
    \int_{\mathbb{R}^N} e^{e^{-2t}\langle x, \mathcal{Q}x\rangle } \prod_{i=1}^m (f_i)_{\Lambda(t)}(B_ix)^{c_i}  \, dx
    = 
    \int_{\mathbb{R}^N} e^{\langle x, \mathcal{Q}x\rangle } \prod_{i=1}^m f_i(B_ix)^{c_i}\, dx. 
\end{equation}

Since $f_i$ is compactly supported, we may assume that ${\rm supp}\, f_i \subset \mathbf{B}^{n_i}_2(R_{f_i})$ for some $R_{f_i}>0$. 
We then claim that 
\begin{equation}\label{e:PWBound-f_i}
    (f_i)_\Lambda(x_i)
    \le 
    C_{{f_i},n_i} \big( \mathbf{1}_{\mathbf{B}^{n_i}_2(10R_{f_i})}(x_i) + e^{-c\Lambda_0 |x_i|^2} \big)
    =: 
    D_{f_i,\Lambda_0}(x_i), \;\;\; i=1, \dots, m
\end{equation}
for some numerical constant $c>0$ and any $\Lambda,\Lambda_0>0$ such that $\Lambda_0 \le \Lambda$. 
Once we could confirm this claim, we may conclude the proof as follows. 
In view of \eqref{e:FiniteCond}, we may apply Lemma \ref{l:Const-Finite} to ensure the existence of some small $t_0$ (in which case $\Lambda(t_0)$ becomes large enough) such that 
$$
\int_{\mathbb{R}^N} e^{ \langle x, \mathcal{Q}_+ x\rangle -e^{-2t_0} \langle x, \mathcal{Q}_- x\rangle } \prod_{i=1}^m D_{f_i,\Lambda(t_0)}(B_ix)^{c_i}\, dx <\infty.  
$$
If we choose $\Lambda_0 := \Lambda(t_0)$ then $\Lambda_0 \le \Lambda(t)$ for any $t\le t_0$. So, together with \eqref{e:PWBound-f_i}, we see that for any $t\le t_0$, 
$$
e^{e^{-2t}\langle x,\mathcal{Q}x\rangle} \prod_{i=1}^m  (f_i)_{\Lambda(t)}(B_ix)^{c_i}
\le 
e^{ \langle x, \mathcal{Q}_+ x\rangle -e^{-2t_0} \langle x, \mathcal{Q}_- x\rangle } \prod_{i=1}^m D_{f_i,\Lambda(t_0)}(B_ix)^{c_i} \in L^1(\R^N). 
$$
This justifies the application of the Lebesgue convergence theorem to conclude \eqref{e:ChangeLimInt}. 

To prove \eqref{e:PWBound-f_i}, we notice from Lemma \ref{l:Fradelizi} that $\| f_i\|_\infty \le e^n f_i(0)$ and thus 
\begin{align*}
(f_i)_{\Lambda}(x_i)
&\le 
e^n f_i(0) 
\frac{1}{(2\pi/\Lambda)^{n_i/2}} \int_{\mathbb{R}^{n_i}} e^{ - \frac{\Lambda}2|x_i-y_i|^2} \mathbf{1}_{ \mathbf{B}^{n_i}_2(R_{f_i}) }(y_i)\, dy_i \\ 
& =
e^n f_i(0) 
\frac{1}{(2\pi)^{n_i/2}} \int_{\mathbb{R}^{n_i}} e^{ - \frac{1}2|y_i|^2} \mathbf{1}_{ \mathbf{B}^{n_i}_2(R_{f_i}) }(x_i-\frac1{\sqrt{\Lambda}}y_i)\, dy_i \\
&= 
\frac{e^n f_i(0) }{(2\pi)^{n_i/2}}
\bigg(
\int_{|y_i|\le \frac{\sqrt{\Lambda}}{10}|x_i|} + \int_{|y_i|\ge \frac{\sqrt{\Lambda}}{10} |x_i|} 
e^{ - \frac{1}2|y_i|^2} \mathbf{1}_{ \mathbf{B}^{n_i}_2(R_{f_i}) }(x_i-\frac1{\sqrt{\Lambda}}y_i)\, dy_i
\bigg). 
\end{align*}
For the first term, notice that 
$$
|y_i|\le \frac{\sqrt{\Lambda}}{10}|x_i|
\quad \Rightarrow 
\quad 
\big|x_i-\frac{1}{\sqrt{\Lambda}} y_i\big|
\ge 
\big||x_i|-\frac{1}{\sqrt{\Lambda}} |y_i|\big|
\ge \frac{9}{10} |x_i|, 
$$
from which we derive the desired bound for the first term:
$$
\int_{|y_i|\le \frac{\sqrt{\Lambda}}{10}|x_i|} 
e^{ - \frac{1}2|y_i|^2} \mathbf{1}_{ \mathbf{B}^{n_i}_2(R_{f_i}) }(x_i-\frac1{\sqrt{\Lambda}}y_i)\, dy_i
\le 
\mathbf{1}_{\mathbf{B}^{n_i}_2( R_{f_i})}(\frac9{10}x_i) \int_{\mathbb{R}^{n_i}} e^{-\frac12|y_i|^2}\, dy_i. 
$$
For the second term, in view of the asymptotic estimate $\int_{K}^\infty e^{-\frac12 t^2}\, dt \sim c \frac{1}{K}e^{- \frac12 K^2}$ as $K\to \infty$, 
\begin{align*}
 	\int_{|y_i|\ge \frac{\sqrt{\Lambda}}{10}|x_i|}
e^{-\frac12|y_i|^2} \mathbf{1}_{\mathbf{B}^{n_i}_2(R_{f_i})}(x_i - \frac1{\sqrt{\Lambda}}y_i)\, dy_i  
%&\le C_{n_i}
%\int_{\frac{\sqrt{\Lambda}}{10}|x_i|}^\infty e^{-\frac12t^2} t^{n_i-1}\, dt \\
%&\le C_{n_i}
%\int_{\frac{\sqrt{\Lambda}}{10}|x_i|}^\infty e^{-\frac14t^2}\, dt \\
&\le C
e^{-c \Lambda |x_i|^2}. 
\end{align*}
These two  estimates yield \eqref{e:PWBound-f_i} and thus we obtain \eqref{e:ChangeLimInt}. 
Our proof is complete. 
\end{proof}

The next task is to get rid of the compact support condition. 
For this purpose, we need some approximating argument. 
More precisely,  we expect that the following property would be true: given any $h \in \mathcal{F}^{(o)}_{G,\infty}(\R^n)$ for some $G\ge0$, there exists $(h_k)_k \subset \mathcal{F}^{(o)}_{G,\infty}(\R^n)$ such that (i) each $h_k$ is compactly supported (ii) $\lim_{k\to\infty} h_k(x) = h(x)$ $dx$-a.e.  $x\in \R^n$, and (iii) there exists a dominating function $D \in L^1_+(\R^n)$ of $h_k$. 
Although this is a fairly reasonable statement, we could not find any literature nor prove this. 
As an alternative statement, we may obtain the following property. 
\begin{proposition}\label{l:CenterApprox}
    Let $G\ge 0$ on $\R^n$ and $h \in \mathcal{F}^{(o)}_{G,\infty}(\R^n)$. 
    We also take $\varepsilon>0$. 
    Then there exist a sequence $ (h_k^{(\varepsilon)})_{k=1}^\infty \subset \mathcal{F}^{(o)}_{G,\infty}(\R^n)$  such that (i) each $h_k^{(\varepsilon)}$ has a compact support, (ii) 
    \begin{equation}\label{e:Approx1}
        \lim_{k \to \infty} h_k^{(\varepsilon)}(x) = h((1 + \varepsilon)x), \;\;\; \text{$dx$-a.e. $x \in \R^n$}, 
    \end{equation}
    and (iii) $\exists R_0 = R_0(\varepsilon,h)\gg1$ s.t. 
    \begin{equation}\label{e:Approx2}
        k\ge R_0\; \Rightarrow\; h_k(x) \le \big( 2 e^n\big)^\varepsilon  h(x), \;\;\;  x \in \R^n, \;\;\;   k \in \mathbb{N}. 
    \end{equation}
\end{proposition}

\begin{proof}
%    {\color{red}NOT YET}
    %Without loss of generality, we may suppose that $\|f \|_\infty \le 1$. 
    %In fact it is enough to consider $\|f\|_\infty^{-1}f$ instead of the original $f$. 
    For $R>0$, put 
    $$
    h_R(x) := h(x) \mathbf{1}_{R\mathbf{B}_2^n}(x), \;\;\; x \in \R^n
    $$
    and  
    $$
    \xi_R := \frac{\int_{\R^n} x h_R(x)\, dx}{\int_{\R^n} h_R(x)\, dx}.  
    $$
    Note that $\int_{\R^n} h_R\, dx>0$ since $h$ is positive near the origin, and so $\xi_R$ is well-defined. 
    Then the function 
    $$
    h_R^{(\varepsilon)}:= h_R((1+\varepsilon) x + \xi_R) 
    $$
    is centered and compactly supported function. 
    Moreover, it is $(1+\varepsilon)^2G$-uniformly log-concave, and so $G$-uniformly log-concave.  
    We also notice that $\lim_{R \to \infty} \xi_R =0$ from the Lebesgue convergence theorem with the dominating function $|x|h \in L^1(\R^n)$. 
    Now let $U=U_h$ be a small bounded neighborhood around $0$ with $\inf_{x \in U} h(x) \ge \frac12 h(0)$; see the proof of \eqref{e:Continuity-f_i} for this property. 
    %To ensure the existence of such a set, we use\footnote{
    %Although this may be well-known, we give a short proof of the second claim that $0\in {\rm int}\big( {\rm supp}\, h \big)$. 
    %This is a consequence of $\int_{\R^n} x h\, dx =0$ as follows. 
    %Suppose contradictory that $0\notin {\rm int}\big( {\rm supp}\, h \big)$. 
    %Since ${\rm int}\big( {\rm supp}\, h \big)$ is open and convex, Hahn--Banach's separation theorem ensures the existence of a hyperplane separating ${\rm int}\big( {\rm supp}\, h \big)$ and $\{0\}$. 
    %More precisely, there exist $u \in \mathbb{S}^{n-1}$ and $t\in \R$ such that 
    %$$
    %\langle u,x \rangle < t\le 0,\quad \forall x\in {\rm int}\big( {\rm supp}\, h \big).
    %$$
    %In particular, $\langle u,x\rangle <0$ and hence 
    %\begin{align*}
    %    \big\langle u, \int_{\R^n} x h(x)\, dx \big\rangle 
    %    = 
    %    \int_{\R^n} \langle u,x\rangle h(x)\, dx <0
    %\end{align*}
    %since $h$ is a nonzero function. However, this contradicts with that $h$ is centered. 
    %} Lemma \ref{l:Fradelizi}, $0\in {\rm int}\big( {\rm supp}\, h \big)$, and $h$ is continuous on ${\rm int}\big( {\rm supp}\, h \big)$. 
    This $U$ depends only on $h$, and so, in view of $\lim_{R\to \infty} \xi_R = 0$, we may find $R_0 = R_0(\varepsilon,h) > 0$ such that $- \varepsilon^{-1} \xi_R \in U \subset R \mathbf{B}_2^n$ for any $R \ge R_0$. 
    In particular, we have that
    $$
    h(-\varepsilon^{-1}\xi_R) \ge \frac12h(0),\quad \forall R\ge R_0, 
    $$
    and thus the log-concavity of $h$ yields that 
    $$
    h(x)^{1 + \varepsilon} \ge h( (1+\varepsilon) x + \xi_R) h(-\varepsilon^{-1} \xi_R)^\varepsilon \ge \big(\frac12 h(0) \big)^\varepsilon h((1+\varepsilon)x + \xi_R)
    $$
    for $x \in \R^n$ and $R \ge R_0$. 
    On the other hand, Lemma \ref{l:Fradelizi} yields that $h(x)^{1+\varepsilon} \le \big(e^n h(0)\big)^\varepsilon h(x) $. 
    Therefore, in view of $h(0)>0$ from Lemma \ref{l:Fradelizi}, we may see that  $h((1+\varepsilon)x+\xi_R) \le \big( 2 e^n\big)^\varepsilon h(x)$ for all $x \in \R^n$ and $R \ge R_0$.  
    This means that 
    $$
    h_R^{(\varepsilon)}(x) = h((1+\varepsilon)x + \xi_R) \mathbf{1}_{R \mathbf{B}_2^n}((1+\varepsilon)x + \xi_R) \le \big( 2 e^n\big)^\varepsilon h(x), \;\;\; \forall x \in \R^n, \;\;\; \forall R \ge R_0. 
    $$
    Finally we claim that 
    $$
    \lim_{R\to\infty}h_R^{(\varepsilon)}(x)
    = 
    h((1+\varepsilon) x),\quad{\rm whenever}\quad (1+\varepsilon)x\in \R^n\setminus \partial \big(  {\rm supp}\, h \big),
    $$ 
    that concludes the proof. 
    To see this claim, suppose $(1+\varepsilon)x\in {\rm int} \big( {\rm supp}\, h \big)$ in which case we may find a larger $R_0' \ge R_0$ so that $R\ge R_0' \Rightarrow $ $(1+\varepsilon)x + \xi_R \in {\rm int} \big( {\rm supp}\, h \big)$. 
    We then make use of the fact that $h$ is continuous on the interior of its support to conclude the desired property. 
    If $(1+\varepsilon)x\notin  {\rm supp}\, h$ then, as above, we may find a larger $R_0'' \ge R_0$ so that $R\ge R_0'' \Rightarrow $ $(1+\varepsilon)x + \xi_R \notin  {\rm supp}\, h $. 
    Thus, 
    $$
    \lim_{R\to\infty}h_R^{(\varepsilon)}(x)
    =
    0
    = 
    h((1+\varepsilon) x).  
    $$
    Since $| \partial \big( {\rm supp}\, h \big) |=0$, this conclude the proof. 
\end{proof}

As a corollary, we may remove the condition that each $f_i$ is compactly supported in Proposition \ref{p:IBL-CompactSupp}.
\begin{corollary}\label{Cor:IBL-G_i>0}
Let $(\mathbf{B},\mathbf{c},\mathcal{Q})$ be arbitrary Brascamp--Lieb datum and $0< G_i < \infty$ for each $i=1,\ldots,m$. 
    Then for any $f_i \in \mathcal{F}^{(o)}_{G_i,\infty}(\R^{n_i})$, we have that 
    $$
    \int_{\mathbb{R}^N} e^{\langle x,\mathcal{Q}x\rangle} 
    \prod_{i=1}^m f_i(B_ix)^{c_i}\, dx 
    \ge 
    {\rm I}^{(\mathcal{G})}_{\mathbf{G}}(\mathbf{B},\mathbf{c}, \mathcal{Q}) 
    \prod_{i=1}^m \big( \int_{\mathbb{R}^{n_i}} f_i\, dx_i \big)^{c_i}. 
    $$
\end{corollary}

\begin{proof}
%{\color{red}NOT YET}
    Without loss of generality, we may suppose that 
    \begin{equation}\label{e:BLFunc-Finite}
        \int_{\R^N} e^{\langle x, \mathcal{Q}x\rangle} \prod_{i=1}^m f_{i}(B_ix)^{c_i}\, dx < +\infty.
    \end{equation}
    With this in mind, we take arbitrary small $\varepsilon$ that tends to 0 in the end, and apply Proposition \ref{l:CenterApprox} to find $(f_{i,k}^{(\varepsilon)})_{k=1}^\infty$ satisfying \eqref{e:Approx1} and \eqref{e:Approx2} for $f_i$. 
    For each $k$, $f_{i,k}^{(\varepsilon)}$ satisfies assumptions of Proposition \ref{p:IBL-CompactSupp}, and hence we obtain that 
    $$
    \int_{\R^N} e^{\langle x, \mathcal{Q}x\rangle} \prod_{i=1}^m f_{i,k}^{(\varepsilon)}(B_ix)^{c_i}\, dx 
    \ge 
    {\rm I}_{\mathbf{G},\infty}^{(\mathcal{G})} (\mathbf{B},\mathbf{c},\mathcal{Q}) \prod_{i=1}^m \big( \int_{\mathbb{R}^{n_i}} f_{i,k}^{(\varepsilon)}\, dx_i \big)^{c_i}. 
    $$
    We will then take the limit $k\to\infty$. 
    For the right-hand side, we may simply apply Fatou's lemma together with \eqref{e:Approx1} to see that  
    \begin{align*}
        \int_{\R^{n_i}} 
        f_i((1+\varepsilon)x_i)\, dx_i 
        &\le 
        \liminf_{k\to\infty}
        \int_{\R^{n_i}} 
        f_{i,k}^{(\varepsilon)}(x_i)\, dx_i.  
    \end{align*}
    For the left-hand side, we know from \eqref{e:Approx2} and \eqref{e:BLFunc-Finite} that 
    $$
    e^{\langle x,\mathcal{Q}x\rangle} 
    \prod_{i=1}^m f_{i,k}^{(\varepsilon)}(B_ix)^{c_i} 
    \le 
    C_{\varepsilon,\mathbf{n}} 
    e^{\langle x,\mathcal{Q}x\rangle} 
    \prod_{i=1}^m f_{i}(B_ix)^{c_i} \in L^1(\R^N)
    $$
    for some constant $C_{\varepsilon, {\bf n}}>0$, where ${\bf n}:=(n_1, \dots, n_m)$. 
    Therefore, we may apply the Lebesgue convergence theorem to confirm that 
    $$
    \lim_{k\to\infty}
    \int_{\R^N} e^{\langle x, \mathcal{Q}x\rangle} \prod_{i=1}^m f_{i,k}^{(\varepsilon)}(B_ix)^{c_i}\, dx 
    =
    \int_{\R^n} e^{\langle x, \mathcal{Q}x\rangle} \prod_{i=1}^m f_{i}((1+\varepsilon)B_ix)^{c_i}\, dx. 
    $$
    Overall, after taking the limit $k\to\infty$, we obtain 
    $$
    \int_{\R^N} e^{\langle x, \mathcal{Q}x\rangle} \prod_{i=1}^m f_{i}((1+\varepsilon)B_ix)^{c_i}\, dx
    \ge 
    {\rm I}_{\mathbf{G},\infty}^{(\mathcal{G})} (\mathbf{B},\mathbf{c},\mathcal{Q}) \prod_{i=1}^m \left( \int_{\mathbb{R}^{n_i}} f_{i}((1+\varepsilon)x_i)\, dx_i \right)^{c_i}. 
    $$
    We then finally take another limit $\varepsilon\to0$. 
    For the right-hand side, it is easy to see that 
    $$
    \int_{\mathbb{R}^{n_i}} f_{i}((1+\varepsilon)x_i)\, dx_i 
    = 
    (1+\varepsilon)^{-n_i}
    \int_{\mathbb{R}^{n_i}} f_{i}(x_i)\, dx_i 
    \to 
    \int_{\mathbb{R}^{n_i}} f_{i}(x_i)\, dx_i.  
    $$
    For the left-hand side, we may run the same argument as in the proof of Proposition \ref{l:CenterApprox}. 
    That is, the log-concavity of $f_i$ implies that 
    $$
    f_i(x_i)^{1 + \varepsilon} %= f_i(\frac1{1+\varepsilon} \big(1+\varepsilon)x_i + \frac{\varepsilon}{1+\varepsilon}0\big)^{1 + \varepsilon} 
    \ge f_i( (1+\varepsilon) x_i ) f_i(0)^\varepsilon, 
    $$
    while Lemma \ref{l:Fradelizi} yields that $f_i(x_i)^{1+\varepsilon} \le \big(e^{n_i} f_i(0)\big)^\varepsilon f_i(x_i) $. 
    In view of $f_i(0)>0$ from Lemma \ref{l:Fradelizi}, we may see that  $f_i((1+\varepsilon)x_i) \le e^{\varepsilon n_i} f_i(x_i)$ from which 
    \begin{align*}
    \int_{\R^N} e^{\langle x, \mathcal{Q}x\rangle} \prod_{i=1}^m f_{i}((1+\varepsilon)B_ix)^{c_i}\, dx
    &\le 
     e^{\varepsilon \sum_{i} c_i n_i}
    \int_{\R^N} e^{\langle x, \mathcal{Q}x\rangle} \prod_{i=1}^m f_{i}(B_ix)^{c_i}\, dx \\
    &\to 
    \int_{\R^N} e^{\langle x, \mathcal{Q}x\rangle} \prod_{i=1}^m f_{i}(B_ix)^{c_i}\, dx
    \end{align*}
    as $\varepsilon\to0$. 
    Putting all together, we conclude the proof. 
\end{proof}

We finally complete the proof of Theorem \ref{t:CenteredIBL} by taking the limit $G_i\to 0$. 
For this purpose, we need to approximate a log-concave and centered function by some uniformly log-concave and centered functions. 
For this purpose we may employ the following which is similar to Proposition \ref{l:CenterApprox}. 
\begin{proposition}\label{p:GaussianApprox}
    Let $G\ge0$ and $h \in \mathcal{F}^{(o)}_{G,\infty}(\R^n)$. 
    We also take $\varepsilon>0$. 
    Then there exists a sequence $(h_k^{(\varepsilon)})_{k=1}^\infty \subset \mathcal{F}^{(o)}_{G,\infty}(\R^n)$ satisfying \eqref{e:Approx1} and \eqref{e:Approx2} such that $h_k^{(\varepsilon)}$ is $G+ \frac1k{\rm id}_n$-uniformly log-concave. 
\end{proposition}

\begin{proof}
	%{\color{red}NOT YET}
    %Without loss of generality, we may suppose that $\|f \|_\infty \le 1$. 
    %In fact it is enough to consider $\|f\|_\infty^{-1}f$ instead of the original $f$. 
    First of all, we remark that the proof here is almost same with Proposition \ref{p:IBL-CompactSupp}. 
    For $R>0$, put 
    $$
    h_R(x) := h(x) e^{-\frac1{2R}|x|^2}, \;\;\; x \in \R^n
    $$
    and  
    $$
    \xi_R := \frac{\int_{\R^n} x h_R(x)\, dx}{\int_{\R^n} h_R(x)\, dx}.  
    $$
    Note that $\int_{\R^n} h_R\, dx>0$ since $h$ is positive near the origin, and so $\xi_R$ is well-defined. 
    Then the function 
    $$
    h_R^{(\varepsilon)}:= h_R((1+\varepsilon) x + \xi_R) 
    $$
    is centered. 
    We also notice that $\lim_{R \to \infty} \xi_R =0$ from the Lebesgue convergence theorem with the dominating function $|x|h \in L^1(\R^n)$.
    % So by choosing $R$ sufficiently large, we have that $|\xi_R|\le \frac{1}{100}$. 
    % We then see an elementary inequality\footnote{
    % Here is more details: 
    % $$
    % |x|\le 
    % | (1+\varepsilon)x  |
    % = 
    % |(1+\varepsilon)x+\xi_R - \xi_R|
    % \le 
    % |(1+\varepsilon)x+\xi_R| + |\xi_R|.
    % $$
    % This means that $|(1+\varepsilon)x+\xi_R|\ge |x| - |\xi_R|$. 
    % In view of $|\xi_R|\le \frac1{100}$, for $|x|\ge 1$, we have $|x| - |\xi_R|>0$ and so 
    % $$
    % |(1+\varepsilon)x+\xi_R|^2
    % \ge 
    % |x|^2 + |\xi_R|^2 - 2 |\xi_R||x| 
    % \ge 
    % |x|^2 + |\xi_R|^2 - \frac1{2}|x|  
    % =
    % \frac12 |x|^2 + |\xi_R|^2 + \frac12 (|x|-\frac1{2})^2 - \frac1{8} 
    % \ge \frac14|x|^2 + \frac1{4} - \frac1{8} \ge \frac14|x|^2.   
    % $$
    % It thus yields that $e^{-\frac{|(1+\varepsilon)x+\xi_R |^2}2} \le e^{-\frac{1}{8}|x|^2}$ for $|x|\ge 1$. 
    % On the other hand, for $|x|\le 1$, it is obvious to see 
    % $$e^{-\frac{|(1+\varepsilon)x+\xi_R |^2}2} \le 1 \le e^\frac{1}8 e^{-\frac{1}{8}|x|^2} .$$
    % } 
    % $e^{-\frac{|(1+\varepsilon)x+\xi_R |^2}2}  \le e^\frac{1}8 e^{-\frac{1}{8}|x|^2} $ and thus 
    % \begin{equation}\label{e:RapidDecay}
    % F_R^{(\varepsilon)}(x)
    % \le 
    % \|f\|_\infty e^{-\frac{|(1+\varepsilon)x+\xi_R |^2}{2R}}
    % \le 
    % \|f\|_\infty e^\frac18 e^{-\frac{1}{8R}|x|^2}
    % \end{equation}
    % for any $x\in \R^n$ uniformly in $\varepsilon$. In one word, this says that $F_R^{(\varepsilon)}$ rapidly decays apart from $R^\frac12\mathbf{B}^n_2$. 
    Moreover, $h_R^{(\varepsilon)}$ is $(1+\varepsilon)^2G + (1+\varepsilon)^2\frac{1}{R}$-uniformly log-concave, and so $G+\frac{1}{R}$-uniformly log-concave. 
    As in the proof of Proposition \ref{p:IBL-CompactSupp}, let $U=U_h$ be a small bounded neighborhood around $0$ with $\inf_{x \in U} h(x) \ge \frac12 h(0)$. 
    This $U$ depends only on $h$, and so, in view of $\lim_{R\to \infty} \xi_R = 0$, we may find $R_0 = R_0(\varepsilon,h) > 0$ such that $- \varepsilon^{-1} \xi_R \in U \subset R \mathbf{B}_2^n$ for any $R \ge R_0$. 
    Then the same argument as in Proposition \ref{p:IBL-CompactSupp} ensures that $h((1+\varepsilon)x+\xi_R) \le \big( 2 e^n\big)^\varepsilon h(x)$ for all $x \in \R^n$ and $R \ge R_0$. 
    % In particular, we have that
    % $$
    % h(-\varepsilon^{-1}\xi_R) \ge \frac12h(0),\quad \forall R\ge R_0, 
    % $$
    % and thus the log-concavity of $f$ yields that 
    % $$
    % f(x)^{1 + \varepsilon} \ge f( (1+\varepsilon) x + \xi_R) f(-\varepsilon^{-1} \xi_R)^\varepsilon \ge \big(\frac12 f(0) \big)^\varepsilon f((1+\varepsilon)x + \xi_R), \;\;\; \forall x \in \R^n, \;\;\;\forall R\ge R_0.  
    % $$
    % On the other hand, Lemma \ref{l:Fradelizi} yields that $f(x)^{1+\varepsilon} \le \big(e^n f(0)\big)^\varepsilon f(x) $. 
    % Therefore, in view of $f(0)>0$ from Lemma \ref{l:Fradelizi}, we may see that  $f((1+\varepsilon)x+\xi_R) \le \big( 2 e^n\big)^\varepsilon f(x)$ for all $x \in \R^n$ and $R \ge R_0$.  
    This means that 
    $$
    h_R^{(\varepsilon)}(x) = h((1+\varepsilon)x + \xi_R) g_{1/R}((1+\varepsilon)x + \xi_R) \le \big( 2 e^n\big)^\varepsilon h(x), \;\;\; \forall x \in \R^n, \;\;\; \forall R \ge R_0. 
    $$
    Finally, by the same argument as in Proposition \ref{p:IBL-CompactSupp}, we may conclude that 
    $$
    \lim_{R\to\infty}h_R^{(\varepsilon)}(x)
    = 
    h((1+\varepsilon) x),\quad{\rm whenever}\quad (1+\varepsilon)x\in \R^n\setminus \partial \big(  {\rm supp}\, h \big).
    $$ 
\end{proof}

\begin{proof}[Proof of Theorem \ref{t:CenteredIBL}]
%{\color{red}NOT YET}
   Let us prove the case that some of $G_i$ is not positive definite. 
   Let us take arbitrary $f_i \in \mathcal{F}_{{\bf G}, \infty}^{(\mathcal{G})}(\R^{n_i})$ for $i=1, \dots, m$. 
    This time, we appeal to Proposition \ref{p:GaussianApprox} to find $(f_{i,k}^{(\varepsilon)})_k$ satisfying \eqref{e:Approx1} and \eqref{e:Approx2} for $f_i$. 
    Note that each $f_{i,k}^{(\varepsilon)}$ is $G_i(k)$-uniformly log-concave where $G_i(k):=G_i + \frac1k {\rm id}_{n_i}>0$, and thus we may employ Corollary \ref{Cor:IBL-G_i>0} to see that 
    %what we have proven above in the case of $G_i>0$ to see that 
    \begin{align*}
    \int_{\R^N} e^{\langle x, \mathcal{Q}x\rangle} \prod_{i=1}^m f_{i,k}^{(\varepsilon)}(B_ix)^{c_i}\, dx 
    &\ge 
    {\rm I}_{\mathbf{G}(k),\infty}^{(\mathcal{G})} (\mathbf{B},\mathbf{c},\mathcal{Q}) \prod_{i=1}^m \big( \int_{\mathbb{R}^{n_i}} f_{i,k}^{(\varepsilon)}\, dx_i \big)^{c_i}
    \\
    &\ge
    {\rm I}_{\mathbf{G},\infty}^{(\mathcal{G})} (\mathbf{B},\mathbf{c},\mathcal{Q}) \prod_{i=1}^m \big( \int_{\mathbb{R}^{n_i}} f_{i,k}^{(\varepsilon)}\, dx_i \big)^{c_i}, 
    \end{align*}
    where the second inequality follows from ${\rm I}_{\mathbf{G}(k),\infty}^{(\mathcal{G})} (\mathbf{B},\mathbf{c},\mathcal{Q}) \ge {\rm I}_{\mathbf{G},\infty}^{(\mathcal{G})} (\mathbf{B},\mathbf{c},\mathcal{Q})$ by $G_i(k) \ge G_i$. 
    By the same argument as in Corollary \ref{Cor:IBL-G_i>0}, we may take the limit $k\to\infty$ and then $\varepsilon\to0$ to observe that 
    $$
    \int_{\R^n} e^{\langle x, \mathcal{Q}x\rangle} \prod_{i=1}^m f_{i}(B_ix)^{c_i}\, dx 
    \ge 
    {\rm I}_{\mathbf{G},\infty}^{(\mathcal{G})} (\mathbf{B},\mathbf{c},\mathcal{Q}) \prod_{i=1}^m \big( \int_{\mathbb{R}^{n_i}} f\, dx_i \big)^{c_i}. 
    $$
    % Finally we may apply the argument proving \eqref{e:ChangeLimGaussconst} to conclude that $\lim_{k\to\infty} {\rm I}_{\mathbf{Q}(k),+}^{(\mathcal{G})} (\mathbf{B},\mathbf{c},\mathcal{Q}) \ge {\rm I}_{\mathbf{Q},\infty}^{(\mathcal{G})} (\mathbf{B},\mathbf{c},\mathcal{Q})$. 
    Since $f_i \in \mathcal{F}_{{\bf G}, \infty}^{(\mathcal{G})}(\R^{n_i})$ is arbitrary, we obtain the desired assertion. 
\end{proof}

%{\color{red}
%Remark at the end: Although we could manage to justify the limiting argument of $G_i \to0$ and $H_i \to \infty$ above, as we saw, its argument is longer and more complicated.  
%The situation will get more serious when one wishes to extend Proposition \ref{p:RegIBL} so that it allows $G_i = 0$ and $H_i =  \infty$ when $c_i \in \R\setminus \{0\}$. 
%In such a case, justifying the limiting argument $G_i\to0$ and $H_j\to\infty$ would require more substantial work. 
%}

\section{Proof of Theorems \ref{t:NonSymGCI} and \ref{t:GenCor}}\label{s:Section4}

% Szarek--Werner \cite{SW} proposed the following Gaussian correlation inequality for non-symmetric convex sets. 

% \begin{problem}[\cite{SW}]\label{ProblemSW}
%     For any convex sets $K_1, K_2 \subset \R^n$ with $\int_{K_1} x\, \frac{d\gamma}{\gamma(K_1)} = \int_{K_2} x\, \frac{d\gamma}{\gamma(K_2)}$, it holds that 
%     $$
%     \gamma(K_1 \cap K_2) \ge \gamma(K_1) \gamma(K_2). 
%     $$
% \end{problem}
% They showed the case that $K_1$ is a convex body and $K_2$ is a strip. 
% Let us show that this problem may be solved affirmatively. 

% To see this,
\subsection{Proof of Theorem \ref{t:GenCor}}\label{SSec-DeriGCI}
We give a proof of Theorem \ref{t:GenCor} by using Theorem \ref{t:CenteredIBL}. 
To this end, choose the Brascamp--Lieb datum as $N=n$, $n_1=\cdots=n_m=n$, and 
 $$ 
 B_1=\cdots = B_m ={\rm id}_n,\quad 
 c_1=\cdots = c_m=1, \quad 
 \mathcal{Q} = \frac12 \sum_{i=1}^m \Sigma_i^{-1} - \frac12 \Sigma_0^{-1}. 
 $$
We then choose the regularized parameters as $G_i = \Sigma_i^{-1}$ for $i=1,\ldots, m$. 
 %We remark that the assumption \eqref{e:LargeQ_i} is valid since 
 %$$
 %\sum_{i=1}^mc_i B_i^* Q_iB_i =
 %\sum_{i=1}^m \Sigma_i^{-1} > \sum_{i=1}^m \Sigma_i^{-1} - \Sigma_0^{-1} = 2\mathcal{Q}
 %$$
 %as we are assuming $\Sigma_0>0$. 
Because of the convexity and the centering assumption of $K_i$, we know that $f_i(x_i) := \mathbf{1}_{K_i}(x_i) e^{-\frac12 \langle \Sigma_i^{-1} x_i, x_i \rangle}\in \mathcal{F}^{(o)}_{G_i,\infty}$. 
With this choice, we have that 
$$
\gamma_{\Sigma_0} \left( \bigcap_{i=1}^m K_i \right) 
= 
\big( {\rm det} (2\pi \Sigma_0) \big)^{-\frac12}
\int_{\R^n}
e^{\langle x, \mathcal{Q}x\rangle} 
\prod_{i=1}^m 
f_i(B_ix)^{c_i} \, dx. 
$$
Therefore, Theorem \ref{t:CenteredIBL} yields that 
$$
\gamma_{\Sigma_0} \left( \bigcap_{i=1}^m K_i \right) 
    \ge
    \frac{ \prod_{i=1}^m ({\rm det} (2\pi \Sigma_i))^\frac12}{ ({\rm det} (2\pi \Sigma_0))^\frac12}
    {\rm I}_{{\bf G}, \infty}^{(\mathcal{G})} ({\bf B}, {\bf c}, \mathcal{Q})
    \prod_{i=1}^m \gamma_{\Sigma_i}(K_i), 
 $$
where the constant ${\rm I}_{{\bf G}, \infty}^{(\mathcal{G})} ({\bf B}, {\bf c}, \mathcal{Q})$ is explicitly given by 
$$
{\rm I}_{{\bf G}, \infty}^{(\mathcal{G})} ({\bf B}, {\bf c}, \mathcal{Q})^2
=
(2\pi)^{- n(m-1)} \inf_{A_i \ge \Sigma_i^{-1}} \frac{ \prod_{i=1}^m {\rm det}(A_i) }{ {\rm det}(\sum_{i=1}^m (A_i - \Sigma_i^{-1}) + \Sigma_0^{-1} )}.
$$
Thus the proof would be completed by showing the following: 
\begin{lemma}\label{l:GenGaussConst}
 Let $\Sigma_0, \Sigma_1, \dots, \Sigma_m \in \mathbb{R}^{n\times n}$ be positive symmetric matrices with $\Sigma_0^{-1} \ge \Sigma_i^{-1}$ for $i=1, \dots, m$. Then for any $A_1, \dots, A_m \in \mathbb{R}^{n\times n}$ with $A_i \ge \Sigma_i^{-1}$ for $i=1, \dots, m$, it holds that 
$$
 \frac{ \prod_{i=1}^m {\rm det}(A_i) }{ {\rm det}(\sum_{i=1}^m (A_i - \Sigma_i^{-1}) + \Sigma_0^{-1} )}
 \ge
 \frac{ \prod_{i=1}^m {\rm det}(\Sigma_i^{-1}) }{ {\rm det}( \Sigma_0^{-1} )}.
$$
    
\end{lemma}

\begin{proof}
Put $\widetilde{A_i} := \Sigma_1^{\frac12} A_i \Sigma_1^{\frac12}$ for $i=1, \dots, m$ and $\widetilde{\Sigma_i}^{-1} := \Sigma_1^{\frac12} \Sigma_i^{-1} \Sigma_1^{\frac12}$ for $i=0, 2, \dots, m$.
    Then 
    $$
    \frac{ \prod_{i=1}^m {\rm det}(A_i) }{ {\rm det}(\sum_{i=1}^m (A_i - \Sigma_i^{-1}) + \Sigma_0^{-1} )} 
    =
    \frac{ {\rm det} (\widetilde{A_1}) \prod_{i=2}^m {\rm det}(A_i) }{ {\rm det} ((\widetilde{A_1} - {\rm id}_n) + \sum_{i=2}^m (\widetilde{A_i} - \widetilde{\Sigma_i}^{-1}) + \widetilde{\Sigma_0}^{-1} )}. 
    $$
    Let us show that, by fixing $A_2,\ldots,A_m$, the right-hand side is minimized when $\widetilde{A}_1={\rm id}_{n}$ among all $\widetilde{A}_1 \ge {\rm id}_n$: 
    \begin{align}\label{e:GenDetIn}
    \frac{ {\rm det} (\widetilde{A_1}) \prod_{i=2}^m {\rm det}(A_i) }{ {\rm det}((\widetilde{A_1} - {\rm id}_n) + \sum_{i=2}^m (\widetilde{A_i} - \widetilde{\Sigma_i}^{-1}) + \widetilde{\Sigma_0}^{-1} )}
    &\ge
    \frac{ \prod_{i=2}^m {\rm det}(A_i) }{ {\rm det}(  \sum_{i=2}^m (\widetilde{A_i} - \widetilde{\Sigma_i}^{-1}) + \widetilde{\Sigma_0}^{-1} )}
    \\
    &=
    \frac{ {\rm det}(\Sigma_1^{-1}) \prod_{i=2}^m {\rm det}(A_i) }{ {\rm det}(\sum_{i=2}^m (A_i - \Sigma_i^{-1}) + \Sigma_0^{-1} )}. \notag
    \end{align}
    Once we could prove this, we may repeat the same argument for each $i=2,\ldots,m$ to conclude the desired result. 

    If we denote $M:= \sum_{i=2}^m (\widetilde{A_i} - \widetilde{\Sigma_i}^{-1}) + \widetilde{\Sigma_0}^{-1} $, then it follows from assumptions $A_i\ge \Sigma_i^{-1}$ and $\Sigma_0^{-1} \ge \Sigma_1^{-1}$ that $M\ge {\rm id}_n$. 
    Thus for the purpose of proving \eqref{e:GenDetIn}, it suffices to show that 
    $$
    \widetilde{A_1},M\ge {\rm id}_n \quad \Rightarrow \quad 
    \frac{ {\rm det} (\widetilde{A_1})  }{ {\rm det}(\widetilde{A_1} - {\rm id}_n + M)}
    \ge 
    \frac1{{\rm det}\, M}. 
    $$
    To show this, without loss of generality, we may suppose that $\widetilde{A_1} = {\rm diag} (a_1, \dots, a_n)$ with $a_1, \dots, a_n \ge 1$.  
    Put $X := e_1 \otimes e_1$ and define for $t \ge -(a_1-1)$, 
    $$
    \Phi(t)
    :=
    \log \frac{ {\rm det} (\widetilde{A_1} + tX) }{ {\rm det}(\widetilde{A_1} + tX - {\rm id}_n + M )}.
    $$
    We then have that 
    \begin{align*}
        \Phi'(t)
        =
        {\rm Tr} \left[  \left( (\widetilde{A}_1 + tX )^{-1} - (\widetilde{A_1} + tX - {\rm id}_n +M )^{-1} \right)X \right]. 
    \end{align*}
    It is readily checked from $M\ge {\rm id}_n$ that $(\widetilde{A}_1 + tX )^{-1} - (\widetilde{A_1} + tX - {\rm id}_n +M )^{-1} \ge0$. 
%    Note that $\widetilde{A_i} \ge \widetilde{\Sigma_i}^{-1}$ for $i=2, \dots, m$ and that $\widetilde{\Sigma_0}^{-1} \ge {\rm id}_n$ since $A_i\ge \Sigma_i^{-1}$ and $\Sigma_0^{-1}\ge \Sigma_1^{-1}$. Thus we see that 
%    \begin{align*}
%        (\widetilde{A_1} + tX - {\rm id}_n) + \sum_{i=2}^m (\widetilde{A_i} - \widetilde{\Sigma_i}^{-1}) + \widetilde{\Sigma_0}^{-1} 
%        \ge \widetilde{A_1} +tX. 
%    \end{align*}
%    Hence, 
%    $$
%     \left(\widetilde{A_1} + tX \right)^{-1} - \left((\widetilde{A_1} + tX - {\rm id}_n) + \sum_{i=2}^m (\widetilde{A_i} - \widetilde{\Sigma_i}^{-1}) + \widetilde{\Sigma_0}^{-1} \right)^{-1} \ge 0. 
%    $$
    In view of $X \ge 0$, this implies that $\Phi'(t) \ge 0$. 
    Hence $\Phi(-(a_1-1)) \le \Phi(0)$, which means that the quantity we are focusing on is indeed minimized when $a_1=1$. 
    By repeating this argument for each $a_2,\ldots,a_m$, we may reduce $a_2,\ldots, a_m$ to 1, and thus \eqref{e:GenDetIn} follows.  
\end{proof}

We give a remark about Milman's way of deriving the Gaussian correlation inequality for symmetric convex sets.
To this end, let us recall the equivalent
%\footnote{To be precise \eqref{e:ProbGCI} is equivalent to a slightly general form 
%$\gamma_{\Sigma}(K_1\cap K_2) \ge \gamma_{\Sigma}(K_1) \gamma_{\Sigma}(K_2)$, where %$\Sigma$ is an arbitrary symmetric positive definite matrix on $\R^n$. However, this is %clearly equivalent to \eqref{e:OriginGCI}.
%} 
form of the Gaussian correlation inequality. 
Let $d_1,d_2\in \mathbb{N}$ and $(X_1,X_2)$ be the random vector in $\R^{d_1}\times \R^{d_2}$ normally distributed with mean zero and covariance matrix $\Sigma = \begin{pmatrix} \Sigma_{11} & \Sigma_{12} \\ \Sigma_{12}^* &\Sigma_{22} \end{pmatrix}$ which is nonnegative definite on $\R^{d_1}\times \R^{d_2}$. 
Then the Gaussian correlation inequality is equivalent to that for symmetric and convex sets $L_i \subset \R^{d_i}$, $i=1,2$, 
%\begin{equation}\label{e:ProbGCI}
\begin{equation}\label{e:GCIProb}
\mathbb{P} ( X_1\in L_1,\; X_2\in L_2 ) \ge \mathbb{P}(X_1\in L_1) \mathbb{P} (X_2 \in L_2). 
\end{equation}
%\end{equation}
In the case that $\Sigma$ is non-degenerate,  this may be read as 
\begin{equation}\label{e:ProbGCI-2}
\int_{\R^{d_1}\times \R^{d_2}} \mathbf{1}_{L_1}(x_1)\mathbf{1}_{L_2}(x_2) \, d\gamma_{\Sigma}(x_1,x_2) 
\ge 
\int_{\R^{d_1}} \mathbf{1}_{L_1}\, d\gamma_{\Sigma_{11}}
\int_{\R^{d_2}} \mathbf{1}_{L_2}\, d\gamma_{\Sigma_{22}}. 
\end{equation}
This would clearly follow from the inverse Brascamp--Lieb inequality 
$$
\int_{\R^{d_1+d_2}} e^{\langle x, \mathcal{Q}_\Sigma x\rangle } h_1(x_1)h_2(x_2) \, dx 
\ge 
\bigg(
\frac{ {\rm det}\, 2\pi \Sigma }{{\rm det}\, 2\pi \Sigma_{11}{\rm det}\, 2\pi \Sigma_{22}}
\bigg)^\frac12 
\prod_{i=1,2} \int_{\R^{d_i}} h_i\, dx_i,
$$
where $\mathcal{Q}_{\Sigma}:= {\frac12} \begin{pmatrix} \Sigma_{11} & 0\\ 0 & \Sigma_{22} \end{pmatrix}^{-1} - {\frac12}\Sigma^{-1}$. 
For this inequality, the linear maps are just orthogonal projections $x = (x_1,x_2) \mapsto x_i$, and thus one may apply Theorem \ref{t:NT3} rather than Theorem \ref{t:CenteredIBL} involving general linear maps. Instead, one has to consider fairly general $\mathcal{Q} = \mathcal{Q}_\Sigma$. 
Note that appealing to \eqref{e:ProbGCI-2} is a way of how Milman \cite{Mil} derived the Gaussian correlation inequality from Theorem \ref{t:NT3}. 
Thus there is a slight difference between the choices of Brascamp--Lieb data of Milman and ours. 
%{We emphasize that \eqref{e:OriginGCI} for symmetric convex bodies follows from \eqref{e:GCIProb} with an appropriate degenerate $\Sigma$, and the latter follows from \eqref{e:ProbGCI-2} by approximating the degenerate $\Sigma$ by nondegenerate one. However, this approximation step becomes problematic if one studies \eqref{e:OriginGCI} for non-symmetric convex bodies or the case of equality. }

%To see the equivalence between \eqref{e:OriginGCI} and \eqref{e:ProbGCI}, we note that \eqref{e:OriginGCI} may be easily generalized into the following form: for any nonnegative definite symmetric matrix $\Sigma$ on $\R^n$, $\gamma_{\Sigma}(K_1\cap K_2) \ge \gamma_{\Sigma}(K_1) \gamma_{\Sigma}(K_2)$. 
%{\color{red}$\gamma_{\Sigma}$の定義がまだ．}
%We then choose $n=2d$ and $K_1 = L_1 \times \R^n$ and $K_2 = \R^n\times K_1$ to derive %\eqref{e:ProbGCI}: 
%$$
%\mathbb{P} ( X\in L_1,\; Y\in L_2 )
%=
%\gamma_{\Sigma}(K_1\cap K_2)
%\ge 
%\gamma_{\Sigma}(L_1 \times \R^n)
%\gamma_{\Sigma}(\R^n \times L_2)
%= 
%\mathbb{P}(X\in L_1) \mathbb{P} (Y \in L_2). 
%$$
%On the other hand, by choosing $d =n$, the degenerate $\Sigma = \begin{pmatrix}
%	{\rm id}_n & {\rm id}_n \\ {\rm id}_n & {\rm id}_n
%\end{pmatrix}$, and $L_i = K_i$, \eqref{e:ProbGCI} yields that 
%$$
%\gamma(K_1\cap K_2)
%=
%\int_{\R^{n}\times \R^n} \mathbf{1}_{L_1}(x_1) \mathbf{1}_{L_2}(x_2) \delta(x_1-x_2) d\gamma(x_1)dx_2
%=
%\mathbb{P} ( X\in L_1,\; Y\in L_2 )
%\ge
%\mathbb{P}(X\in L_1) \mathbb{P} (Y \in L_2)
%= 
%\gamma(K_1)\gamma(K_2). 
%$$

\subsection{Proof of Theorem \ref{t:NonSymGCI}}
Through this section, we consider the special Brascamp--Lieb datum, namely $m=2$, $n_1=n_2=n$, $c_1=c_2=1$, $B_1=B_2 = {\rm id}_n$, and $\mathcal{Q} = \frac12 {\rm id}_n$. 
Furthermore we denote 
    $$ 
    \mathcal{F}^{(a)}_1
    =
\{f \in L^1_+(\mathbb{R}^n): 
\int_{\mathbb{R}^n} x \frac{f(x)\, dx}{\int_{\R^n} f\, dy } =a,\; 
%\text{centered}, \;
{\rm id}_n\text{-uniformly log-concave}
\}, 
$$
for $a \in \R^n$.  The corresponding Brascamp--Lieb constant is defined by 
$$
{\rm I}^{(a)}:= \inf_{f_1, f_2 \in \mathcal{F}^{(a)}_1}{\rm BL}(f_1, f_2). 
$$
With this terminology, the proof of Theorem \ref{t:NonSymGCI} may be reduced to the following claim: 
\begin{theorem}\label{t:NonCenterIBL}
    Let $m=2$, $n_1=n_2=n$, $c_1=c_2=1$ and $B_1=B_2 = {\rm id}_n$, and $\mathcal{Q} = \frac12 {\rm id}_n$. 
    For any $a \in \R^n$, it holds that 
    $$
    {\rm I}^{(a)} \ge {\rm I}_{{\rm id}_n, \infty}^{(o)}({\bf B}, {\bf c}, \mathcal{Q}) = (2\pi)^{-\frac n2}. 
    $$
    Furthermore if there exist $f_1, f_2 \in \mathcal{F}_1^{(a)}$ satisfying 
    $$
    {\rm BL}(f_1, f_2)  = (2\pi)^{-\frac n2}, 
    $$
    then it must hold $a=0$. 
\end{theorem}
Remark that we will use the second part of Theorem \ref{t:NonCenterIBL} to prove Theorem \ref{t:EqualityCase} in the next section. 
By assuming Theorem \ref{t:NonCenterIBL} for a while, let us conclude the proof of Theorem \ref{t:NonSymGCI}. 
\begin{proof}[Proof of Theorem \ref{t:NonSymGCI}]
If we choose 
$
f_i := \mathbf{1}_{K_i} e^{-\frac12 |\cdot|^2},   
$
$i=1,2$, 
then, by the assumption of $K_i$, we have $f_1, f_2 \in \mathcal{F}_1^{(a)}$, where 
$$
a := \int_{K_1} x\, \frac{d\gamma}{\gamma(K_1)} = \int_{K_2} x\, \frac{d\gamma}{\gamma(K_2)}. 
$$
Thus we may apply Theorem \ref{t:NonCenterIBL} to see that 
$$
\int_{\R^n} e^{\frac12 |x|^2} f_1(x) f_2(x)\, dx \ge (2\pi)^{-\frac n2} \int_{\R^n} f_1\, dx \int_{\R^n} f_2\, dx, 
$$
which yields the desired assertion. 
\end{proof}

\begin{proof}[Proof of Theorem \ref{t:NonCenterIBL}]
    First, remark that ${\rm I}_{{\rm id}_n, \infty}^{(o)}({\bf B}, {\bf c}, \mathcal{Q}) = (2\pi)^{-\frac n2}$ is the consequence of Theorem \ref{t:GenCor}. 
    {Thus, to show the first part of Theorem \ref{t:NonCenterIBL},} it suffices to show the inequality. 
    For this purpose, we claim the following Ball's inequality: 
    for any ${\rm id}_n$-uniformly log-concave $h_1, h_2 \in L^1_+(\R^n)$ with $\int_{\R^n} h_1\, dx = \int_{\R^n} h_2\, dx=1$ and $a \in \R^n$, it holds that 
    \begin{equation}\label{e:BallTypeIn}
        \left( \int_{\R^n} e^{\frac12 |x + a|^2} h_1(x) h_2(x)\, dx \right)^2
        \ge
        (2\pi)^{-\frac n2} \int_{\R^n} e^{\frac12 |x + \sqrt{2}a|^2} \widetilde{h}_1^{(2)}(x)  \widetilde{h}_2^{(2)}(x)\, dx, 
    \end{equation}
    where $\widetilde{h}_i^{(2^k)}:= (2^\frac{k}2)^{n} h_i^{(2^k)}(2^{\frac{k}2}\cdot)$ and $h_i^{(2^k)}$ is given in \eqref{e:IterateConv}  for $i=1,2$ and  $k \in \mathbb{N}$. 
    The idea of the proof of \eqref{e:BallTypeIn} is almost the same as Lemma \ref{l:Monotone*}. 
    In fact, the direct calculation combining with the change of variables yields that 
    \begin{align*}
        &\left( \int_{\R^n} e^{\frac12 |x + a|^2} h_1(x) h_2(x) dx \right)^2
        \\
        &=
        \int_{\R^n} e^{\frac12 |x + \sqrt{2}a|^2} \int_{\R^n} e^{\frac12 |y|^2} h_1\left( \frac{x+y}{\sqrt{2}} \right) h_1\left( \frac{x-y}{\sqrt{2}} \right) h_2\left( \frac{x+y}{\sqrt{2}} \right) h_2\left( \frac{x-y}{\sqrt{2}} \right)\, dydx. 
    \end{align*}
    As already checked in Lemma \ref{l:Monotone*}, we know that the functions  
    $$
    H_i^{(x)} : \R^n \ni y \mapsto h_i\left( \frac{x+y}{\sqrt{2}} \right) h_i\left( \frac{x-y}{\sqrt{2}} \right) \in [0, \infty), \;\;\; i=1,2
    $$
    are in $\mathcal{F}_{{\rm id}_n,\infty}^{(o)}$. 
    Thus by applying Theorem \ref{t:CenteredIBL} and Lemma \ref{l:GenGaussConst} with $m=2$ and $\Sigma_0=\Sigma_1=\Sigma_2 = {\rm id}_n$, we obtain that 
     \begin{align*}
        &\left( \int_{\R^n} e^{\frac12 |x + a|^2} h_1(x) h_2(x) dx \right)^2
         \\
         &\ge
        (2\pi)^{-\frac n2}
        \int_{\R^n} e^{\frac12 |x + \sqrt{2}a|^2} \int_{\R^n} H_1^{(x)}(y)\, dy 
         \int_{\R^n} H_2^{(x)}(y) \, dy\, dx. 
    \end{align*}
    This means \eqref{e:BallTypeIn}. 

    To obtain the desired conclusion, we contradictory suppose that there exists some $a \in \R^n \setminus \{0\}$ such that ${\rm I}^{(a)} < (2\pi)^{-\frac n2}$. 
    Then we may take some functions $f_1, f_2 \in \mathcal{F}_{{\rm id}_n, \infty}^{(a)}$ such that 
    \begin{equation}\label{e:ContraAss}
    {\rm BL}(f_1, f_2) < (2\pi)^{-\frac n2}. 
    \end{equation}
    Without loss of generality, we may suppose that $\int_{\R^n} f_1\,d x = \int_{\R^n} f_2\, dx =1$. 
    By letting $h_i(x) := f_i(x + a)$, the iterative applications of \eqref{e:BallTypeIn} imply that 
    \begin{align*}
        &\left( \int_{\R^n} e^{\frac12 |x+a|^2} h_1(x) h_2(x) \,dx \right)^{2^k}
        \\
        &\ge
        (2\pi)^{-\frac n2 (2^{k}-1)} \int_{\R^n} e^{\frac12 |x + 2^{k/2}a|^2} \widetilde{h}_1^{(2^k)}(x) \widetilde{h}_2^{(2^k)}(x) \,dx, \;\;\; \forall k \in \mathbb{N}.  
    \end{align*}
    On the other hand, by definition, we see that 
    $$
    \int_{\R^n} e^{\frac12 |x|^2} f_1(x) f_2(x) \,dx 
    =
    \int_{\R^n} e^{\frac12 |x+a|^2} h_1(x) h_2(x)\, dx.  
    $$
    Combining with \eqref{e:ContraAss}, we conclude 
    $$
    (2\pi)^{-\frac n2} 
    > \int_{\R^n} e^{\frac12 |x + 2^{k/2}a|^2} \widetilde{h}_1^{(2^k)}(x) \widetilde{h}_2^{(2^k)}(x) \, dx, \;\;\; \forall k \in \mathbb{N}. 
    $$
    However this is a contradiction. In fact, if one notices $h_i \in \mathcal{F}_{{\rm id}_n, \infty}^{(o)}$ by definition, the central limit theorem enables us to take centered Gaussians $g_1, g_2$ such that $\lim_{k \to \infty} \widetilde{h}_i^{(2^k)} = g_i$ pointwisely for $i=1,2$. 
    Thus since $a \neq 0$, Fatou's lemma implies that 
    $$
    \liminf_{k \to \infty} \int_{\R^n} e^{\frac12 |x + 2^{k/2}a|^2} \widetilde{h}_1^{(2^k)}(x) \widetilde{h}_2^{(2^k)}(x)\, dx
    = \infty. 
    $$
    %Our proof is complete. 

    {Next, let us show the second part in Theorem \ref{t:NonCenterIBL}. 
    Let $f_1, f_2 \in \mathcal{F}_1^{(a)}$ satisfy 
    $$
    {\rm BL}(f_1, f_2)  = (2\pi)^{-\frac n2}.  
    $$
    Then by the same argument above, we have 
    $$
    (2\pi)^{-\frac n2} 
    \ge \int_{\R^n} e^{\frac12 |x + 2^{k/2}a|^2} \widetilde{h}_1^{(2^k)}(x) \widetilde{h}_2^{(2^k)}(x) \, dx, \;\;\; \forall k \in \mathbb{N}, 
    $$
    where $\widetilde{h}_i^{(2^k)}$ for $i=1,2$ are the same ones above. 
    Thus we may conclude $a=0$ by considering $k \to \infty$ as we have already seen above. 
    }
\end{proof}

%\begin{remark}
%Though we only described the Gaussian correlation inequality for two convex sets for simplicity, 
%the argument above may be easily generalized for multiple functions for which all barycenters coincide. 
%\end{remark}

%{\color{blue}
%\begin{remark}
    One may wonder another formulation of the non-symmetric version of the Gaussian correlation inequality rather than Theorem \ref{t:NonSymGCI}.
    One possibility is as follows: for any log-concave functions $h_1, h_2 \in L^1(\gamma)$, 
    \begin{equation}\label{e:BaryGCI}
    \int_{\R^n} h_1 h_2\, d\gamma \ge (1 + \langle {\rm bar}(h_1), {\rm bar}(h_2) \rangle) \int_{\R^n} h_1\, d\gamma \int_{\R^n} h_2\, d\gamma, 
    \end{equation}
    where 
    $$
    {\rm bar}(h_i) := \int_{\R^n} x \frac{h_i d\gamma}{\int_{\R^n} h_i\, d\gamma}, \;\;\; i=1,2. 
    $$
    Actually Harg\'{e} \cite[Theorem 2]{Harge2} have shown \eqref{e:BaryGCI} for any \textit{convex} functions $h_1, h_2$. 
    We here observe that one cannot hope to have \eqref{e:BaryGCI} for all log-concave functions. 
    To see this, let us consider $n=1$, $A_1 :=(0, \infty)$ and $A_2 =A_2(R_2):= (R_2, \infty)$ for $R_2 \gg 1$, and put $h_i := \mathbf{1}_{A_i}$ for $i=1,2$. 
    Note that $A_2 \subset A_1$. Hence, if \eqref{e:BaryGCI} would hold true for any log-concave functions, we could derive 
    $$
    1 \ge (1 + \langle {\rm bar}(h_1), {\rm bar}(h_2) \rangle) \gamma(A_1). 
    $$
    On the other hand, by definition, we know that ${\rm bar}(h_1), {\rm bar}(h_2)>0$ and moreover $\lim_{R_2 \to \infty} {\rm bar}(h_2)=\infty$. 
    This means that 
    $$
    \lim_{R_2 \to \infty} (1 + \langle {\rm bar}(h_1), {\rm bar}(h_2) \rangle) \gamma(A_1)=\infty
    $$
    which is a contradiction.

We conclude this section by giving another variant of the Gaussian correlation inequality introduced by Cordero-Erausquin \cite{C-E}. 
%Let us conclude this section by giving a variant of the non-symmetric Gaussian correlation inequality. 
To describe this, we need to introduce some notations. 
For a set $A \subset \R^n$, let ${\rm Iso}(A)$ be a set of all isometries $r$ on $\R^n$ with $r(A)=A$, and put 
$$
{\rm Fix} (A) :=\{x \in \R^n\, :\, r(x)=x, \;\;\; \forall r \in {\rm Iso}(A)\}. 
$$
Remark that ${\rm Iso}(A)$ consists of finite compositions of translations, rotations and reflections. Especially ${\rm Fix}(A)=\{0\}$ if $A$ is symmetric. 
It is worth to noting that, in general, ${\rm Fix}(A)$ is determined by the geometric shape of $A$ only, rather than for instance the distribution of mass. 
Then Cordero-Erausquin \cite{C-E} showed \eqref{e:OriginGCI} for a convex body $K_1 \subset \R^n$ with ${\rm Fix}(K_1)=\{0\}$ and $K_2=\mathbf{B}_2^n := \{ x \in \R^n\, :\, |x| \le1\}$. 
We give a generalization of it for multiple convex sets, which seems new even when $m=2$. 

\begin{corollary}\label{cor:Cordero}
    Let $m \in \mathbb{N}$ be $m \ge 2$. For any convex sets $K_i \subset \R^n$ with ${\rm Fix} (K_i)=0$, $i=1, \dots, m$, we have \eqref{e:m-GCI}. 
\end{corollary}

\begin{proof}
    For any $r \in {\rm Iso}(K_i)$, 
    $$
    \int_{K_i} x \, d\gamma = \int_{r(K_i)} x \, d\gamma 
    =
    \int_{K_i} r(x) \, d\gamma
    =
    r \left( \int_{K_i} x \, d\gamma \right), 
    $$
    which means that $\int_{K_i} x \, d\gamma \in {\rm Fix}(K_i)$ for $i =1, \dots, m$. 
    Since ${\rm Fix}(K_i)=\{0\}$, we have $\int_{K_i} x \, d\gamma=0$, and thus we may apply Theorem \ref{t:GenCor} to see the desired assertion. 
\end{proof}

\section{Proof of Theorem \ref{t:EqualityCase}}
We begin with an investigation of the case of equality in the inequality in Lemma \ref{l:GenGaussConst} by specializing $m=2$ and $\Sigma_0=\Sigma_1 = \Sigma_2={\rm id}_n$: \begin{equation}\label{e:MatrixGCI}
\frac{ \det (A_1) \det(A_2) }{ \det \big( A_1+A_2 - {\rm id}_n \big) }
\ge 
\frac{ \det ({\rm id}_n) \det({\rm id}_n) }{ \det \big( {\rm id}_n+{\rm id}_n - {\rm id}_n \big) },\quad A_1,A_2\ge {\rm id}_n.
\end{equation}
In other words, we first study the case of equality in the Gaussian correlation inequality for Gaussian inputs. 
For this purpose, we introduce a useful notation. 
For a given $A\ge {\rm id}_n$, let $E_{\rm id}(A) \subset \R^n$ be the eigenspace of $A$ corresponding to the eigenvalue 1: 
$$
E_{\rm id}(A):= 
{\rm Span}\, \big\{ u : Au =u \big\}. 
$$
This $E_{\rm id}(A)$ is uniquely determined by $A$. 
Evidently, $E_{\rm id}(A) = \{0\}$ means that $A > {\rm id}_n$ and $E_{\rm id}(A) = \R^n$ means that $A = {\rm id}_n$. 
The spectral decomposition of $A$ gives that 
\begin{equation}\label{e:SpecDecomp}
A =  
P_{E_{\rm id}(A)^\perp}^* \overline{A}P_{E_{\rm id}(A)^\perp}
+
P_{E_{\rm id}(A)}^* P_{E_{\rm id}(A)}
\end{equation}
for some $\overline{A} > {\rm id}_{E_{\rm id}(A)^\perp}$, where $P_E$ denotes the orthogonal projection onto a subspace $E$. 
In fact, if we denote eigenvalues of $A$ by $a_1,\ldots, a_{n_0} \gneq 1 = a_{n_0+1}=\cdots = a_n$ and eigenvectors by $u_1,\ldots, u_n$, where $n_0 := {\rm dim}\, (E_{\rm id}(A)^\perp)\ge0$, then 
\begin{equation}\label{e:DetailedSpecDecomp}
A =\sum_{i=1}^{n_0} a_i u_i\otimes u_i + \sum_{j=n_0+1}^n u_j\otimes u_j
=:  
P_{E_{\rm id}(A)^\perp}^* \overline{A}P_{E_{\rm id}(A)^\perp}
+
P_{E_{\rm id}(A)}^* P_{E_{\rm id}(A)}. 
\end{equation}

\begin{lemma}\label{l:EqualMatrixGCI}
    Equality in \eqref{e:MatrixGCI} holds if and only if it holds that 
    $$
    A_2 = P^*_{E_{\rm id}(A_1)^\perp}P_{E_{\rm id}(A_1)^\perp} + P^*_{E_{\rm id}(A_1)} \overline{A_2} P_{E_{\rm id}(A_1)}
    $$
    for some $\overline{A}_2 \ge {\rm id}_{ E_{\rm id}(A_1) }$, or in other word, 
    \begin{equation}\label{e:EqualMatrixGCI}
         E_{\rm id}(A_2) \supset E_{\rm id}(A_1)^\perp . 
    \end{equation}
\end{lemma}

\begin{proof}
    Suppose $A_1,A_2$ satisfy \eqref{e:EqualMatrixGCI} and let us see that equality holds in \eqref{e:MatrixGCI}. 
    As in the proof of Lemma \ref{l:GenGaussConst}, we may assume that $A_1$ is diagonal. Moreover, by denoting $n_0 := {\rm dim}\, E_{\rm id}(A_1)^\perp \ge 0$, we may suppose that\footnote{In the case of $n_0 =0$, we understand $A_1 = {\rm id}_n$.} 
    \begin{equation}\label{e:DiagA1}
    A_1 = {\rm diag}\, (a^{(1)}_1,\ldots, a_{n_0}^{(1)},1,\ldots,1)
    \end{equation}
    for some $a^{(1)}_1,\ldots, a_{n_0}^{(1)}\gneq 1$. 
    In other word, $E_{\rm id}(A_1)^\perp = \langle e_1,\ldots,e_{n_0}\rangle$ and $E_{\rm id}(A_1) = \langle e_{n_0+1},\ldots, e_n\rangle$. 
    It follows from \eqref{e:EqualMatrixGCI} that $E_{\rm id}(A_2) \supset \langle e_1,\ldots, e_{n_0}\rangle$. Hence, in view of the spectral decomposition \eqref{e:SpecDecomp},  $A_2$ must have the form of 
    \begin{equation}\label{e:GoalA2}
    A_2 = \begin{pmatrix}
        {\rm id}_{n_0} & O_{ n-n_0, n_0 } \\
        O_{n_0, n-n_0} & \overline{A_2}
    \end{pmatrix}
    \quad \text{for some}\quad \overline{A_2} \ge {\rm id}_{\langle e_{n_0+1},\ldots, e_n\rangle}.  
    \end{equation}
    Given such an explicit form of $A_1,A_2$, it is readily to see equality in \eqref{e:MatrixGCI}. 

    We next take any $A_1,A_2$ establishing equality in \eqref{e:MatrixGCI} and show that these must satisfy \eqref{e:EqualMatrixGCI}. 
    For this purpose, we may again assume the form of \eqref{e:DiagA1} without loss of generality. This is because, if we take any $U \in O(n)$ and let $\widetilde{A}_i:= U^* A_i U$, then $\widetilde{A}_1,\widetilde{A}_2$ also satisfy equality in \eqref{e:MatrixGCI}. Also, it suffices to notice that $E_{\rm id}(U^* AU) = U^*( E_{\rm id}(A) )$ holds in general from \eqref{e:DetailedSpecDecomp}. 
    In the case of \eqref{e:DiagA1}, our goal is to show that $A_2$ has the form of \eqref{e:GoalA2}. 
    By recalling the proof of Lemma \ref{l:GenGaussConst}, we consider  
    \begin{align*}\Phi_i(t)&:= 
    \log\, \frac{ \det A_1^i(t) }{\left( \det A_1^i(t) +C  \right)}, \quad t \in  [ -( a^{(1)}_i -1 ), 0 ], 
    \end{align*}
    for each $i=1,\ldots, n_0$, where 
    $$
    A_1^i(t):= A_1 + t e_i \otimes e_i,\quad C:= A_2 - {\rm id}_n. 
    $$
    Note that $C\ge0$. 
    As we saw in the proof of Lemma \ref{l:GenGaussConst}, 
    \begin{align*}
        \Phi_i'(t) 
        &= 
        {\rm Tr}\, \bigg[ \bigg( A_1^i(t)^{-1} - \big( A_1^i(t) + C \big)^{-1} \bigg) e_i \otimes e_i \bigg] \\
        & = 
        \bigg\langle e_i,\bigg( A_1^i(t)^{-1} - \big( A_1^i(t) + C \big)^{-1} \bigg)  e_i \bigg\rangle, 
    \end{align*}
    and $\Phi'_i(t) \ge0$ holds for all $t \in [-( a^{(1)}_i -1 ), 0]$, from which the inequality \eqref{e:MatrixGCI} follows. 
    Since $(A_1,A_2)$ establishes equality in \eqref{e:MatrixGCI}, we must have that $\Phi_i'(t) =0$, that is 
    $$
    \bigg\langle e_i,\bigg( A_1^i(t)^{-1} - \big( A_1^i(t) + C \big)^{-1} \bigg)  e_i \bigg\rangle = 0,\quad \forall t\in [ -(a^{(1)}_i -1), 0 ]. 
    $$
    In particular, at $t=0$, this gives that 
    $$
    (a^{(1)}_i)^{-1} = \bigg\langle e_i, \big( A_1 + C \big)^{-1}  e_i \bigg\rangle.  
    $$
    We have for $i=1$, for instance, that 
    \begin{align*}
    \bigg\langle e_1, \big( A_1 + C \big)^{-1} e_1\bigg\rangle 
    %&= 
    %\bigg\langle e_1, 
    %\begin{pmatrix}
    %    a_1^{(1)} + C_{11} & C_{12} & \cdots & C_{1n} \\
    %    C_{12} & &&\\ 
    %    \vdots &  & \overline{D_1} + \overline{C} & \\
    %    C_{1n} 
    %\end{pmatrix}^{-1} e_1
    %\bigg\rangle \\
    &= 
    \bigg( (a_1^{(1)}+C_{11}) - ( C_{12} \cdots C_{1n} ) \big( {A_1}' + {C}' \big)^{-1} \begin{pmatrix} C_{12}\\ \vdots \\ C_{1n} \end{pmatrix} \bigg)^{-1}, 
    \end{align*}
    where we use notations that $A = (A_{ij})_{i,j=1,\ldots,n}$ and ${A}':= (A_{ij})_{i,j=2,\ldots,n}$. 
    It thus follows that 
    \begin{equation}\label{e:C11}
        C_{11}
        = 
         ( C_{12} \cdots C_{1n} ) \big( {A_1}' + {C}' \big)^{-1} \begin{pmatrix} C_{12}\\ \vdots \\ C_{1n} \end{pmatrix}. 
    \end{equation}
    We claim that it follows from \eqref{e:C11} that 
    \begin{equation}\label{e:Goal6/10-1}
        C_{11} = C_{12}=\cdots = C_{1n} = 0. 
    \end{equation}
    %IF we could have that $\overline{C} >0$ then it is easy to conclude \eqref{e:Goal6/10-1} as follows. 
    %Note that $\overline{C}^{-1}$ exists in such a case. Thus, in view of $
    %\overline{D_1}>0$, 
    %\begin{align*}
    %    C_{11}
    %    &= 
    %    ( C_{12} \cdots C_{1n} ) \big( \overline{D_1} + \overline{C} \big)^{-1} \begin{pmatrix} C_{12}\\ \vdots \\ C_{1n} \end{pmatrix} \\ 
    %    &\le 
    %    ( C_{12} \cdots C_{1n} )  \overline{C}^{-1} \begin{pmatrix} C_{12}\\ \vdots \\ C_{1n} \end{pmatrix}, 
    %\end{align*}
    %and the inequality is strict unless $\begin{pmatrix} C_{12}\\ \vdots \\ C_{1n} \end{pmatrix} = \boldsymbol{0}$. 
    %On the other hand, as in the Fact 4.1 in Barthe--Wolff, it follows from $C\ge0$ and $\overline{C}>0$ that 
    %$$
    %C_{11} - ( C_{12} \cdots C_{1n} )  \overline{C}^{-1} \begin{pmatrix} C_{12}\\ \vdots \\ C_{1n} \end{pmatrix} \ge0. 
    %$$
    %Thus, the above inequality must be equality and thus we must have that $\begin{pmatrix} C_{12}\\ \vdots \\ C_{1n} \end{pmatrix} = \boldsymbol{0}$. 
    %Moreover, this shows \eqref{e:Goal6/10-1} since 
    %$$
    %C_{11}
    %    = 
    %    ( C_{12} \cdots C_{1n} ) \big( \overline{D_1} + \overline{C} \big)^{-1} \begin{pmatrix} C_{12}\\ \vdots \\ C_{1n} \end{pmatrix} = 0. 
    %$$
    %However, in general, $\overline{C}$ is not necessarily positive definite, and so we need an extra argument. 
    We postpone showing this claim later and let us conclude the proof by admitting the validity of \eqref{e:Goal6/10-1}. 
    By adapting the same argument for all $i =1,\ldots, n_0$, we derive that $C_{i1}= \cdots = C_{in}=0$ for all $i =1,\ldots, n_0$, and thus 
    $$
    C = \begin{pmatrix}
        O_{n_0} & O_{ n-n_0,n_0 } 
        \\ 
        O_{n_0 , n-n_0} & \overline{C}
    \end{pmatrix} \quad \text{for some} \quad \overline{C} \ge 0 \; {\rm on}\; \langle e_{n_0+1},\ldots,e_n\rangle. 
    $$
    In view of the definition $C:= A_2-{\rm id}_n$, this proves \eqref{e:GoalA2}.

    To complete the proof, we need to show \eqref{e:Goal6/10-1}. 
    For this purpose, we make use of ${A_1}'>0$. 
    Since $C' \ge 0$, for any $\varepsilon >0$, we have that ${C}'+\varepsilon A_1'>0$. Thus it follows from $C+\varepsilon A_1>0$ ans Schur's complement that %\footnote{In general, given a block matrix representation $A = \begin{pmatrix} X & Y \\ Y^* &Z \end{pmatrix}$, if $Z>0$ then $A> 0$ is equivalent to $X- YZ^{-1}Y^*>0$. } 
    $$
    C_{11}+ \varepsilon a^{(1)}_1 >
    (C_{12} \cdots C_{1n}) \big( {C}' + \varepsilon {A_1}' \big)^{-1} \begin{pmatrix} 
    C_{12}\\ \vdots \\ C_{1n} 
    \end{pmatrix}
    =: \big\langle \boldsymbol{C}, \big( {C}' + \varepsilon {A_1}' \big)^{-1} \boldsymbol{C}\big\rangle, 
    $$
    where we use a notation $\boldsymbol{C}:= (C_{12},\ldots,C_{1n})^*$. 
    By combining this with \eqref{e:C11}, 
    $$
    \big\langle \boldsymbol{C}, \big( {C}' +  {A_1}' \big)^{-1} \boldsymbol{C}\big\rangle
    + \varepsilon a^{(1)}_1 > 
    \big\langle \boldsymbol{C}, \big( {C}' + \varepsilon {A_1}' \big)^{-1} \boldsymbol{C}\big\rangle. 
    $$
    In view of $A_1 >0$, there exists some $\delta>0$ such that $A_1 \ge 2 \delta\, {\rm id}_n$. 
    Since $C' + A_1' \ge C' + 2\delta\, {\rm id}_{n-1}$ and $C' + \varepsilon A_1' \le C' + \varepsilon A_1' + \delta {\rm id}_{n-1}$, we obtain that 
    $$
    \langle \boldsymbol{C}, (C' + 2\delta\, {\rm id}_{n-1})^{-1} \boldsymbol{C} \rangle + \varepsilon a_1^{(1)} > \langle \boldsymbol{C}, ( C' + \varepsilon A_1' + \delta\, {\rm id}_{n-1} )^{-1} \boldsymbol{C} \rangle. 
    $$
    Letting $\varepsilon \to 0$, it follows that 
    $$
    \langle \boldsymbol{C}, (C' + 2\delta\, {\rm id}_{n-1})^{-1} \boldsymbol{C} \rangle 
    \ge 
    \langle \boldsymbol{C}, ( C' + \delta\, {\rm id}_{n-1} )^{-1} \boldsymbol{C} \rangle. 
    $$
    However, this would be a contradiction unless $\boldsymbol{C} =0$ since the reversed inequality $(C' + \delta\, {\rm id}_{n-1})^{-1} > (C' + 2\delta\, {\rm id}_{n-1})^{-1}$ always holds.
    %, it holds that  
    %$$
    %\langle \boldsymbol{C}, (C' + \delta\, {\rm id}_{n-1})^{-1} \boldsymbol{C} \rangle
    %>
    %\langle \boldsymbol{C}, ( C' + 2\delta {\rm id}_{n-1} )^{-1} \boldsymbol{C} \rangle, 
    %$$
    %which is a contradiction. 
    %Thus $\boldsymbol{C}=0$. 
    That $C_{11}=0$ follows from \eqref{e:C11} combining with $\boldsymbol{C}=0$. 
\end{proof}

By using Lemma \ref{l:EqualMatrixGCI} we derive a constraint on the extremizer for the functional form of the Gaussian correlation inequality. 
That is, we investigate the case of equality in the inequality 
\begin{equation}\label{e:FuncGCI6/10}
    \int_{\R^n} h_1(x) h_2(x)\, d\gamma(x) \ge 1
\end{equation}
for all log-concave $h_i$ such that 
\begin{equation}\label{e:Assump/10}
\int_{\R^n} x h_i\, d\gamma = 0,\quad  
\int_{\R^n} h_i\, d\gamma=1.
\end{equation}
As in Subsection \ref{SSec-DeriGCI}, if we let $f_i:= h_i\gamma$ then the Gaussian correlation inequality is translated to the inverse Brascamp--Lieb inequality 
\begin{equation}\label{e:GCIIBL6/10}
    \int_{\R^n} e^{\frac12|x|^2} f_1(x) f_2(x)\, dx \ge (2\pi)^{-\frac{n}2}
\end{equation}
for all ${\rm id}_n$-uniformly log-concave $f_i$ such that 
\begin{equation}\label{e:Class6/10}
\int_{\R^n} x f_i\, dx= 0,\quad  
\int_{\R^n} f_i\, dx=1.
\end{equation}
We denote the covariance matrix of a centered probability density $\rho$ by ${\rm Cov}\, (\rho) := \int_{\R^n} x\otimes x \rho\, dx$. 

\begin{proposition}\label{p:EqualGCI-Step1}
    Suppose that $f_1,f_2$ are 1-uniformly log-concave functions satisfying \eqref{e:Class6/10} and establish equality in \eqref{e:GCIIBL6/10}. 
    Then 
    \begin{equation}\label{e:EqualGCI-Step1}
    {\rm Cov}\, (f_2) = P_{E_{\rm id}(f_1)^\perp}^*P_{E_{\rm id}(f_1)^\perp} + P_{E_{\rm id}(f_1)}^* \overline{\Sigma_2} P_{E_{\rm id}(f_1)}
    \end{equation}
    for some $\overline{\Sigma_2}\le {\rm id}_{E_{\rm id}(f_1)}$, and $E_{\rm id}(f_1):= E_{\rm id}( {\rm Cov}\, (f_1) )$. 
    In other word, 
    \begin{equation}\label{e:EqualGCI-Step2}
        E_{\rm id}(f_1)^\perp \subset E_{\rm id}(f_2).
    \end{equation}
\end{proposition}

\begin{proof}
    We revisit the proof of Proposition \ref{p:RegIBL}. In particular, from \eqref{e:Iteration} as well as the Remark in the end of the subsection 3.2, that $f_1,f_2$ are the extremizer for \eqref{e:GCIIBL6/10} yields that $\gamma_{\Sigma_1},\gamma_{\Sigma_2}$ are also the extremizer for \eqref{e:GCIIBL6/10} where $\Sigma_i:= {\rm Cov}\, (f_i)$. 
    This means that $A_i:= \Sigma_i^{-1}$ establish equality in \eqref{e:MatrixGCI}. 
    Thus, applying Lemma \ref{l:EqualMatrixGCI}, we see that
    $$
    \Sigma_2^{-1} = 
    P^*_{ E_{\rm id}(\Sigma_1^{-1})^\perp }P_{ E_{\rm id}(\Sigma_1^{-1})^\perp }
    + 
    P^*_{ E_{\rm id}(\Sigma_1^{-1}) } (\overline{\Sigma_2})^{-1} P_{ E_{\rm id}(\Sigma_1^{-1}) }
    $$
    for some $(\overline{\Sigma_2})^{-1} \ge {\rm id}_{E_{\rm id}(\Sigma_1^{-1})}$. 
    Since $E_{\rm id}(\Sigma_1^{-1}) = E_{\rm id}(\Sigma_1)$, this concludes that 
    $$
    \Sigma_2 = 
    P^*_{ E_{\rm id}(\Sigma_1)^\perp }P_{ E_{\rm id}(\Sigma_1)^\perp }
    + 
    P^*_{ E_{\rm id}(\Sigma_1) } \overline{\Sigma_2} P_{ E_{\rm id}(\Sigma_1) }. 
    $$
\end{proof}

At this stage, we are led to the following question. 
For a centered probability measure $\mu$, let $E:= E_{\rm id}({\rm Cov}\, (\mu))$ that may be $\{0\}$ or $\R^n$. In view of the spectral decomposition \eqref{e:DetailedSpecDecomp}, 
\begin{equation}\label{e:AssumpCov}
    {\rm Cov}\, (\mu) = P^*_{E}P_{E} + P_{E^\perp}^* \overline{M} P_{E^\perp}
\end{equation}
holds true for some $\overline{M}$ that has no eigenvalue 1. 
Then is this splitting property inhered to the pointwise behavior of $\mu$? 
More precisely, is it true that there exists some probability measure $\overline{\mu}$ on $E^\perp$ such that 
\begin{equation}\label{e:MeasureSplit}
    d\mu(x) = d\gamma(x_E) d\overline{\mu}(x_{E^\perp}),\quad \text{a.e.} \quad x = x_E+x_{E^\perp} \in \R^n = E\oplus E^\perp? 
\end{equation}
It is not difficult to see that the answer is in general negative. 
In fact, by considering an example $d\mu(x) = \frac{1}{| r_n\mathbf{B}^n_2 |} \mathbf{1}_{ r_n\mathbf{B}^n_2}$ with an appropriate $r_n>0$, we have that ${\rm Cov}\,(\mu) = {\rm id}_n$ while all variables correlates with other variables and so \eqref{e:MeasureSplit} does not occur. 
However, it is still curious to ask whether it is possible to derive the splitting property \eqref{e:MeasureSplit} by assuming some extra assumption on $\mu$. 
For instance, for our purpose, $\mu$ has a strong property that it is 1-uniformly log-concave. 
We indeed give a positive answer to the above question under the 1-uniform log-concavity.
To this end, we employ a strong measure-splitting result due to Gigli, Ketterer, Kuwada, and Ohta \cite{GKKO}. 
For a probability measure $\mu$ with finite second moment, define the relative entropy functional by 
$$
{\rm Ent}_{\mu}(\phi):= \int_{\R^n} \phi\log\, \phi\, d\mu,  
$$
where the test function $\phi$ is such that $\int_{\R^n} \phi\, d\mu =1$ and $\int_{\R^n} |x|^2 \phi\, d\mu <\infty$, %and $\int_{\{\phi>1\}} \phi\log\,\phi \, d\mu <\infty$, 
otherwise ${\rm Ent}_{\mu}(\phi):= \infty$. 
For $\kappa \in \R$, we say that $(\R^n,|\cdot|,  \mu)$ satisfies the curvature-dimension condition ${\rm CD}(\kappa,\infty)$ if ${\rm Ent}_\mu$ is $\kappa$-convex with respect to the $L^2$-Wasserstein metric; see \cite{BGL,GKKO, Vill} for more details. It might be worth to mentioning that if $\mu$ is regular enough so that $d\mu = e^{-V}dx$ for some $V \in C^2(\R^n)$, then the ${\rm CD}(\kappa,\infty)$-condition is simply equivalent to  $\nabla^2 V \ge \kappa{\rm id}_n$; see \cite{BGL, Vill}. 
However, that the definition of ${\rm CD}(\kappa,\infty)$ does not require the differentiability of $\mu$ is crucial for our purpose. 
\begin{theorem}[Gigli, Ketterer, Kuwada, and Ohta \cite{GKKO}]\label{t:GKKO}
    Let $\kappa>0$ and assume that $(\R^n,|\cdot|, \mu)$ satisfies the ${\rm CD}(\kappa,\infty)$-condition. 
    If there exist $\phi_1,\ldots, \phi_k$ such that 
    \begin{equation}\label{e:EqualPoi}
        \langle \phi_i,\phi_j\rangle_{L^2(\mu)}=\delta_{ij},\quad 
        \int_{\R^n} |\phi_i|^2\, d\mu - \left( \int_{\R^n} \phi_i \, d\mu\right)^2
        = 
        \frac1{\kappa} \int_{\R^n} |\nabla \phi_i|^2\, d\mu
    \end{equation}
    for some $k \in \{1,\ldots,n\}$ and all $i=1,\ldots,k$, 
    then there exists a subspace $V$ of $\R^n$ with ${\rm dim}\, V=k$ such that $d\mu(x) = d\gamma_{\kappa^{-1}{\rm id}_n}(x_V) d\overline{\mu}(x_{V^\perp})$ for some $\overline{\mu}$. 
\end{theorem}

We remark that the authors in \cite{GKKO} established this result on  metric measure spaces so-called RCD$(\kappa,\infty)$, and this is the point on their paper. 
In fact, if one is ready to impose a smooth structure, the same result has been established by Cheng and Zhou \cite{CZ} before their result. 
For our purpose, we would like to apply this type of results with $\mu = \mathbf{1}_{K}\gamma$, which is not a smooth function on $\R^n$ precisely speaking, for a convex set $K$. We thus invoke the former result. 

\begin{corollary}\label{cor:GKKO}
    Suppose that $\mu$ is a centered and 1-uniform log-concave probability measure and let $E:= E_{\rm id}({\rm Cov}\, (\mu))$. 
    Then the $\mu$ splits off the standard Gaussian on $E$: \eqref{e:MeasureSplit}. 
\end{corollary}

\begin{proof}
    Let $k:= {\rm dim}\, E$. If $k=0$ then there is nothing to prove, so we assume $k\ge1$. 
    %Since we know from the Poincar\'{e} inequality that ${\rm Cov}\, (\mu) \le {\rm id}_n$, we may also assume that $\overline{M}<{\rm id}_{E^\perp}$ without loss of generality\footnote{Otherwise there must exist at least one eigen value and vector $(\lambda,v)$ of $\overline{M}$ with $\lambda=1$ and $v \in E^\perp$. 
    %We then collect all such eigen vectors $v_1,\ldots, v_l$ and enlarge $E$ by $\widetilde{E}:=E \oplus \langle v_1,\ldots, v_l\rangle$. 
    %Since other eigen vectors of $\overline{M}$ are $\lneq 1$, $\widetilde{M}:= {\rm Cov}\, (\mu)|_{\widetilde{E}^\perp}$ is $< {\rm id}_{\widetilde{E}^\perp}$. 
    %}.
    It is well-known that if $\mu$ is 1-uniform log-concave then the CD$(1,\infty)$-condition is satisfied; see \cite{BGL, Vill} for instance. 
    From \eqref{e:AssumpCov}, we have that 
    $$
    \int_{\R^n} \langle u_i,x\rangle \langle u_j,x\rangle\, d\mu = \delta_{ij},
    $$
    for all $i,j=1,\ldots,k$, where $u_1,\ldots, u_{k}$ are orthonormal basis of $E$. 
    Since $\mu$ is centered, by letting $\phi_i(x):= \langle u_i,x\rangle$, this means that the assumptions in \eqref{e:EqualPoi} are satisfied with $\kappa=1$. 
    Therefore the $\mu$ must split and have the form of $d\mu(x) = d\gamma(x_V) d\overline{\mu}(x_{V^\perp})$ for a.e. $x = x_V+x_{V^\perp} \in \R^n$ and for some $k$-dim subspace $V$ and $\overline{\mu}$. 
    It remains to show that $V$ is indeed $E$. 
    However, this is a consequence of $E= E_{\rm id}({\rm Cov}\, (\mu))$. 
    %$\overline{M}<{\rm id}_{E^\perp}$ together with the dimensional consideration. {\color{red}TRUE?}%\footnote{
    %Suppose contradictory that there exits $u_*\in V\setminus E$. We may normalize $|u_*|=1$. 
    %If we decompose $u_* = P_E^*P_E u_* + P_{E^\perp}^* P_{E^\perp} u_*$ then $P_{E^\perp}^* P_{E^\perp} u_*$ must be non-zero; otherwise $u_* \in E$.  
    %Since $u_*\in V\cap \mathbb{S}^{n-1}$, 
    %$$
    %1 = \langle u_* ,{\rm Cov}\, ( \gamma_V \otimes \overline{\mu}) u_*\rangle= \langle u_* ,{\rm Cov}\, ( \mu) u_*\rangle
    %= | P_E u_* |^2 + \langle P_{E^\perp}u_*, \overline{M} P_{E^\perp}u_*\rangle.
    %$$
    %On the other hand, in view of $\overline{M} > {\rm id}_{E^\perp}$ and $P_{E^\perp}u_* \neq 0$, we have that $\langle P_{E^\perp}u_*, \overline{M} P_{E^\perp}u_*\rangle\gneq | P_{E^\perp} u_* |^2$.
    %By putting together, we derive the contradiction that
    %$$
    %1\gneq | P_Eu_* |^2 + |P_{E^\perp}u_*|^2 =|u_*|^2=1. 
    %$$
    %Similarly, if there exists $u_* \in E\setminus V$ then
    %} 
\end{proof}

\begin{remark}
    One may wonder if there is a simpler way to prove Corollary \ref{cor:GKKO} without appealing to the too strong result Theorem \ref{t:GKKO}. 
    This is indeed the case if the measure $\mu$ is regular enough so that $\mu = e^{-V}dx$ with $V\in C^2(\R^n)$. 
    One way of seeing this is to exploit the case of equality for the variance Brascamp--Lieb inequality 
    $$
    \int_{\R^n} | \langle u,x\rangle |^2\, d\mu \le \int_{\R^n} \langle u, (\nabla^2V)^{-1} u\rangle \, d\mu. 
    $$
    Another way is to investigate the case of equality of the Bochner inequality. The letter was kindly pointed out by Shin-ich Ohta and we appreciate to him. We will give more details about this in the Appendix. 
\end{remark}

Let us conclude the proof of Theorem \ref{t:EqualityCase}. 
We indeed characterize the case of equality in the functional form \eqref{e:FuncGCI6/10}. 
\begin{theorem}
    Among log-concave $h_1,h_2$ satisfying 
    {$$
    \int_{\R^n} x h_1\, d\gamma = \int_{\R^n} x h_2\, d\gamma,\quad  
\int_{\R^n} h_1\, d\gamma = \int_{\R^n} h_2\, d\gamma=1,
    $$}
    equality in \eqref{e:FuncGCI6/10} is achieved 
    if and only if 
    {\eqref{e:Assump/10} holds}
    and 
    \begin{equation}\label{e:EqualityFuncGCI}
    h_1(x) = h_1(x_{E^\perp}),\quad h_2(x) = h_2(x_{E}),\quad \textrm{a.e.}\quad x = x_E+x_{E^\perp} \in \R^n = E\oplus E^\perp.
    \end{equation}
    holds for $E = E_{\rm id}(h_1\gamma)$.
\end{theorem}

\begin{proof}
    It is evident that equality in \eqref{e:FuncGCI6/10} is achieved if $h_1,h_2$ satisfy \eqref{e:EqualityFuncGCI}, and thus we focus on showing the necessity part. 
    %For $h_1$, define $E_0(h_1)$ to be a unique maximal subspace of $\R^n$ having the property that 
    %$$
    %h_1(x) = h_1(x_{E_0(h_1)^\perp}),\quad \text{a.e.}\quad x = x_{E_0(h_1)} + x_{E_0(h_1)^\perp} \in \R^n = E_0(h_1) \oplus E_0(h_1)^\perp.
    %$$
    %{\color{red} Existence??}
    %We here remark that $E_0(h_1)$ may be trivial $\{0\}$ or $\R^n$. If $E_0(h_1) = \R^n$ then $h_1$ is constant on whole $\R^n$ and so there is nothing to prove since any $h_2$ gives equality in \eqref{e:FuncGCI6/10}. 
    %If $E_0(h_1) =\{0\}$ then this means that there is no direction along which $h_1$ is constant. Thus, in such a case we need to show that $h_2$ must be constant on $\R^n$.
    {To this end, we first note that the second part of Theorem \ref{t:NonCenterIBL} yields \eqref{e:Assump/10}.} Thus, we have only to consider the case of centered $h_1,h_2$. 
    If we let $f_i = h_i\gamma$ then $f_1,f_2$ satisfy the assumption in Proposition \ref{p:EqualGCI-Step1}.  
    Hence, we have the inclusion \eqref{e:EqualGCI-Step2}. 
    On the other hand, by applying Corollary \ref{cor:GKKO} with $d\mu =f_idx$, we see that 
    \begin{equation}\label{e:ApplySplit}
    f_i(x) = \gamma(x_{E_{\rm id}(f_i)})d\overline{\mu}_i( x_{E_{\rm id}(f_i)^\perp} ),\quad \text{a.e.}\quad x = x_{E_{\rm id}(f_i)}+x_{E_{\rm id}(f_i)^\perp} \in \R^n %= E_{\rm id}(f_i)\oplus E_{\rm id}(f_i)^\perp,
    \end{equation}
    for $i=1,2$. Since $E_{\rm id}(f_1)^\perp \subset E_{\rm id}(f_2)$, we may decompose 
    $$
    E_{\rm id}(f_2) = E_{\rm id}(f_1)^\perp \oplus W,\quad 
    \R^n = E_{\rm id}(f_1)^\perp \oplus W \oplus E_{\rm id}(f_2)^\perp, 
    $$
    where\footnote{For subspaces $E\subset V \subset \R^n$, we denote the orthogonal space of $E$ in $V$ by $E^{\perp_V}$. } $W:= \big( E_{\rm id}(f_1)^\perp \big)^{\perp_{ E_{\rm id}(f_2)}} = \big(E_{\rm id}(f_1)^\perp \oplus  E_{\rm id}(f_2)^\perp\big)^\perp$. 
    Hence, \eqref{e:ApplySplit} with $i=2$ may be read as 
    \begin{align*}
        &f_2(x)dx 
        = 
        d\gamma(x_{E_{\rm id}(f_1)^\perp}) d\gamma(x_W) d\overline{\mu}_2(x_{E_{\rm id}(f_2)^\perp}),\\
        &\qquad \text{a.e.} \quad x= x_{E_{\rm id}(f_1)^\perp} + x_W + x_{E_{\rm id}(f_2)^\perp} \in \R^n.
    \end{align*}
    Since $W\oplus E_{\rm id}(f_2)^\perp = E_{\rm id}(f_1)$, we may write $x_W + x_{E_{\rm id}(f_2)^\perp}=x_{E_{\rm id}(f_1)}$ and thus 
    $$
    f_2(x)dx
    = 
    d\gamma(x_{E_{\rm id}(f_1)^\perp}) d\overline{\mu}_{12}(x_{E_{\rm id}(f_1)}),\quad 
    \text{a.e.} \quad x = x_{E_{\rm id}(f_1)}+x_{E_{\rm id}(f_1)^\perp} \in \R^n, 
    $$
    where $d\overline{\mu}_{12}(x_{E_{\rm id}(f_1)}):=d\gamma(x_W) d\overline{\mu}_2(x_{E_{\rm id}(f_2)^\perp}) $. 
    By comparing this with \eqref{e:ApplySplit} with $i=1$, we may conclude \eqref{e:EqualityFuncGCI} with $E = E_{\rm id}(f_1)$.
    %the decomposition \eqref{e:EqualGCI-Step1} follows. 
    %This means that, for $d\mu := f_2 dx$, \eqref{e:AssumpCov} is satisfied with $E = E_{\rm id}(f_1)^\perp$ and thus Corollary \ref{cor:GKKO} confirms that 
    %$$
    %f_2(x)dx = d\gamma(x_{E_{\rm id}(f_1)^\perp}) d\overline{\mu}_2( x_{E_{\rm id}(f_1)} ),\quad \text{a.e.}\quad x = x_{E_{\rm id}(f_1)^\perp} + x_{E_{\rm id}(f_1)} \in \R^n 
    %$$
    %for some $\overline{\mu}_2$. 
    %Also, from the definition of $E_{\rm id}(f_1) := E_{\rm id}( {\rm Cov}\, (f_1) )$, we have the decomposition 
    %$$
    %{\rm Cov}\, (f_1) = P_{E_{\rm id}(f_1)}^* P_{E_{\rm id}(f_1)} +  P_{E_{\rm id}(f_1)^\perp}^* \overline{\Sigma_1} P_{E_{\rm id}(f_1)^\perp}
    %$$
    %for some $\overline{\Sigma_1} < {\rm id}_{E_{\rm id}(f_1)^\perp}$.
    %Thus, for $d\mu := f_1 dx$ and $E:= E_{\rm id}(f_1)$,  \eqref{e:AssumpCov} is satisfied, and so Corollary \ref{cor:GKKO} confirms that 
    %$$
    %f_1(x)dx = d\gamma(x_{E_{\rm id}(f_1)}) d\overline{\mu}_1( x_{E_{\rm id}(f_1)^\perp} ),\quad \text{a.e.}\quad x = x_{E_{\rm id}(f_1)^\perp} + x_{E_{\rm id}(f_1)} \in \R^n
    %$$
    %for some $\overline{\mu}_1$. 
    %By replacing $f_i$ by $h_i \gamma$ we conclude the proof. 
\end{proof}

By choosing $h_i:= \frac{1}{\gamma(K_i)}\mathbf{1}_{K_i}$, \eqref{e:EqualityFuncGCI} may be read as \eqref{e:EqualGCISet} and thus the proof of Theorem \ref{t:EqualityCase} is completed.

\section*{Appendix}

\subsection{Proof of Corollary \ref{cor:Translate}}

First, we observe that any convex body $K \subset \R^n$ has some point $b_0 \in K$ such that ${\rm bar}_\gamma(K-b_0)=0$. 
To see this, let 
$$
\Psi_K(b) := \int_{K} x  \frac{e^{\langle x, b \rangle}\, d\gamma(x)}{\int_K e^{\langle\cdot, b \rangle}\, d\gamma}, \;\;\; b \in K. 
$$
It suffices to show that $\Psi_K(b) \in K$ for any $b \in K$. In fact, since $K$ is a convex body and $\Phi_K$ is continuous, thes together with Brouwer's fixed-point theorem yield the existence of $b_0 \in K$ satisfying $\Psi_K(b_0) = b_0$, which means ${\rm bar}_\gamma(K-b_0)=0$. 
To see $\Psi_K(b) \in K$, we may suppose that $b \in {\rm int}(K)$ due to continuity of $\Psi_K$. 
Then the Minkowski functional $\| \cdot \|_{K-b}$ is well-defined since $0 \in {\rm int}(K-b)$. Thus 
\begin{align*}
    \| \Psi_K(b) - b \|_{K-b}
    &=
    \left\| \int_{K} (x-b)  \frac{e^{\langle x, b \rangle}\, d\gamma(x)}{\int_K e^{\langle\cdot, b \rangle}\, d\gamma} \right\|_{K-b}
    =
    \left\| \int_{K-b} x  \,\frac{d\gamma(x)}{\gamma(K-b)} \right\|_{K-b}
    \\
    &\le
    \int_{K-b} \| x\|_{K-b}  \,\frac{d\gamma(x)}{\gamma(K-b)}
    \le 1, 
\end{align*}
which implies $\Psi_K(b) -b \in K-b$, that is, $\Psi_K(b) \in K$. 

With this in mind, let us show Corollary \ref{cor:Translate}. 
On the one hand, the above argument allows us to take $b_1 \in K_1$ and $b_2 \in K_2$ satisfying ${\rm bar}_\gamma(K_1-b_1)={\rm bar}_\gamma(K_2-b_2)=0$, and thus apply Theorem \ref{t:NonSymGCI} to see that 
$$
\gamma\big( ( K_1 - b_1) \cap ( K_2 - b_2 ) \big) \ge 
\gamma( K_1 - b_1 )\gamma ( K_2 - b_2 ). 
$$ 
If we have equality here, then $\{ X \in K_i - b_i \}$, $i=1,2$, are independent, which leads to our desired assertion. 
Hence, we suppose that the inequality is strict: $\Phi_{K_1,K_2}(-b_1,-b_2)>1$, where 
$$
\Phi_{K_1,K_2}(a_1,a_2):= \frac{\gamma\big( ( K_1 + a_1) \cap ( K_2 + a_2 ) \big)}{ 
\gamma( K_1 + a_1 )\gamma ( K_2 + a_2 )}. 
$$
On the other hand, by taking $a_1,a_2$ with $|a_1 - a_2| \to \infty$, we obtain $\Phi_{K_1,K_2}(a_1,a_2)\to 0$ since $K_1,K_2$ are convex bodies which allows $K_1+a_1$ not to intersect with $K_2+a_2$. 
Thus, by continuity of $\Phi_{K_1,K_2}$, there must exist $a_1,a_2$ for which $\Phi_{K_1,K_2}(a_1,a_2) = 1$, and hence it means independence of the events $\{X\in K_1 + a_1\}$ and $\{X\in K_2 + a_2\}$.

\subsection{Formal proof of Theorem \ref{t:GKKO} for linear $\phi_i$ and regular $\mu$}
In this appendix we give a formal proof of Theorem \ref{t:GKKO} in the special case that $\phi_1,\ldots,\phi_k$ are linear functions and $\mu$ is regular in the sense that $d\mu = \rho dx$ with some $\rho$ such that $\log\, \rho \in C^2(\R^n)$.  
The formality here means that we are going to use the integration by parts freely and do not also explicitly identify a class of test functions; we refer to \cite{GKKO} for the rigorous argument. 
Since it suffices to consider the case $k=1$ and $\kappa=1$, we simply assume that 
\begin{equation}\label{e:Assump6/12}
    \int_{\R^n} |\phi|^2\, d\mu = \int_{\R^n} |\nabla  \phi|^2\, d\mu
\end{equation}
for $\phi(x) = \langle u,x\rangle $ for some $u\in \mathbb{S}^{n-1}$. 
Let $\mathcal{L}_\mu$ be the generator of the Dirichlet form $\int_{\R^n} \langle \nabla \phi_1, \nabla \phi_2\rangle \, d\mu$. 
%Since we are assuming $d\mu = e^{-V}dx$ with $V\in C^2(\R^n)$, it is explicitly given by %$\mathcal{L}_\mu = \Delta - \langle \nabla V, \nabla \cdot\rangle$. 
The Poincar\'{e} inequality  is a consequence from the Bochner inequality stating that 
\begin{equation}\label{e:Bochner}
    \frac12 \int_{\R^n} |\nabla \phi|^2 \mathcal{L}_\mu \psi\, d\mu 
    + 
    \int_{\R^n} (\mathcal{L}_\mu \phi)^2 \psi\, d\mu + \int_{\R^n} \langle \nabla \psi, \nabla \phi\rangle \mathcal{L}_\mu \phi\, d\mu 
    \ge \int_{\R^n} | \nabla \phi |^2 \psi \, d\mu
\end{equation}
for all test functions $\phi,\psi$.
The left-hand side is the well-known $\Gamma_2$ operator in a weak formulation. 
Due to equality \eqref{e:Assump6/12} in the Poincar\'{e} inequality, the Bochner inequality must be achieved its equality by the same $\phi(x) = \langle u,x\rangle$. 
Since the extremizer of the Poincar\'{e} inequality is the eigenfunction of $\mathcal{L}_\mu$, we also know that $\mathcal{L}_\mu \phi = -\phi$. Thus, equality version of \eqref{e:Bochner} may be simplified to 
\begin{align*}
    \int_{\R^n} | \langle u,x\rangle |^2\psi\, d\mu
    -
    \int_{\R^n} \langle u, \nabla \psi\rangle \langle u,x\rangle\, d\mu 
    = 
    \int_{\R^n} \psi\, d\mu
\end{align*}
for all test functions $\psi$, where we also used $\int_{\R^n} \mathcal{L}_\mu \psi \, d\mu =0$ for the first term. 
For the second term of the left-hand side, by $d\mu = \rho dx$, 
\begin{align*}
    -
    \int_{\R^n} \langle u, \nabla \psi\rangle \langle u,x\rangle\, d\mu 
    &= 
    -\int_{\R^n}
    \big\langle \langle u,x\rangle \rho(x) u, \nabla \psi \big\rangle\, dx\\
    &= 
    \int_{\R^n} {\rm div}\, \big( \langle u,x\rangle \rho(x) u \big) \psi\, dx \\
    &= 
    \int_{\R^n}  \big( \langle u, \nabla\rho(x)\rangle \langle u,x\rangle + \rho(x) \big) \psi\, dx.  
\end{align*}
Hence, we derive that 
\begin{align*}
    \int_{\R^n} | \langle u,x\rangle |^2 \rho(x) \psi\, dx
    + 
    \int_{\R^n} \langle u, \nabla \rho(x)\rangle \langle u,x\rangle \psi\, dx = 0
\end{align*}
and thus we have that $ \langle u, \nabla \rho(x)\rangle = - \langle u,x\rangle \rho(x)$ in an appropriate weak sense. 
We now make use of the regularity of $\rho$ to conclude that $\rho(x)$ is the standard Gaussian in the direction $u$.

\section*{Acknowledgments}
This work was supported by JSPS Kakenhi grant numbers 21K13806, 23K03156, and 23H01080 (Nakamura), and JSPS Kakenhi grant numbers 24KJ0030 (Tsuji). 
The authors would be grateful to Emanuel Milman for impressive discussions and comments. 
The authors also appreciate Shin-ichi Ohta for his comment on Theorem  \ref{t:GKKO}.


\begin{thebibliography}{99}
\bibitem{ACS} R. Assouline, A. Chor, and S. Sadovsky, {\it A refinement of the \v{S}id\'{a}k-Khatri inequality and a strong Gaussian correlation conjecture}, arXiv:2407.15684, 2024.
%\bibitem{AKM} S. Artstein-Avidan, B. Klartag, V. Milman, \textit{The Santal\'{o} point of a function, and a functional form of the Santal\'{o} inequality}, Mathematika \textbf{51} (2004), 33--48. 

\bibitem{BGL} D. Bakry, I. Gentil, M. Ledoux, \textit{Analysis and geometry of Markov diffusion operators}. Grundlehren der mathematischen Wissenschaften 348, Springer (2014). 


\bibitem{BallPhd} K. Ball,  \textit{Isometric problems in $\ell_p$ and sections of convex sets},  Doctoral thesis, University of Cambridge, 1986.
\bibitem{Ball91} K. Ball, \textit{Volumes of sections of cubes and related problems}, Geometric Aspects of Functional Analysis, ed. by J. Lindenstrauss and V. D. Milman, Lecture Notes in Math. 1376, Springer, Heidelberg, 1989, 251--260.
\bibitem{BallJLMS} K. Ball, \textit{Volume ratio and a reverse isoperimetric inequality}, J. Lond. Math. Soc.   \textbf{44} (1991), 351--359. 
%\bibitem{BarGAFA}F. Barthe, {\it Optimal Young's inequality and its converse: a simple proof}, Geom. funct. anal. {\bf 8} (1998), 234--242.
%\bibitem{BarInvMath}F. Barthe, {\it On a reverse form of the Brascamp--Lieb inequality}, Invent. math. {\bf 134} (1998), 335--361.
\bibitem{BW} F. Barthe, P. Wolff, \textit{Positive Gaussian kernels also have Gaussian minimizers}, Mem. Am. Math. Soc., {\bf 276} (2022), no. 1359, 90pp.
\bibitem{Beckner} W. Beckner, \textit{Inequalities in Fourier analysis}, Ann. of Math. \textbf{102} (1975), 159--182.
\bibitem{BBBCF} J. Bennett, N. Bez, S. Buschenhenke, M. G. Cowling, T. C. Flock, \textit{On the nonlinear  Brascamp--Lieb inequality},  Duke Math. J.  \textbf{169} (2020), 3291--3338. 
\bibitem{BCCT} J. Bennett, A. Carbery, M. Christ, T. Tao,  \textit{The Brascamp--Lieb inequalities: finiteness, structure and extremals}, Geom. Funct. Anal. \textbf{17} (2008), 1343--1415.

%
\bibitem{BN} N. Bez, S. Nakamura, \textit{Regularised Brascamp--Lieb inequalities}, Anal. PDE. \textbf{18} (2025), 1567--1613. 
\bibitem{Borell} C. Borell, \textit{A Gaussian correlation inequality for certain bodies in $\R^n$}, Math. Ann. \textbf{256} (1981), 569--573. 
        %\bibitem{Bha} R. Bhatia, Positive Definite Matrices. 
        \bibitem{BraLi_Adv} H. J. Brascamp, E. H. Lieb, \textit{Best constants in Young's inequality, its converse, and its generalization to more than three functions}, Adv. Math. \textbf{20}  (1976), 151--173.
        \bibitem{BLJFA}  H. J. Brascamp, E. H. Lieb, {On extensions of the Brunn-Minkowski and Pr\'{e}kopa-Leindler theorems, including inequalities for log concave functions, and with an application to the diffusion equation}, J. Funct. Anal. {\bf 22} (1976), 366--389.
        %\bibitem{BGVV}S. Brazitikos, A. Giannopoulos, P. Valettas and B.-H. Vritsiou, Geometry of isotropic convex bodies, Mathematical Surveys and Monographs, 196. American Mathematical Society, Providence, RI, 2014.
        \bibitem{CDP} W. K. Chen, N. Dafnis, G. Paouris, \textit{Improved H\"{o}lder and reverse H\"{o}lder inequalities for Gaussian random vectors}, Adv. Math. \textbf{280} (2015), 643--689.
        \bibitem{CZ} X. Cheng, D. Zhou, \textit{Eigenvalues of the drifted Laplacian on complete metric measure spaces}, Commun. Contemp. Math. \textbf{19} (2017), 1650001, 17 pp. 
        \bibitem{CP} S. Chewi, A.-A. Pooladian. \textit{An entropic generalization of Caffarelli's contraction theorem via covariance inequalities}, C. R. Math. Acad. Sci. Paris, \textbf{361} (2023), 1471--1482. 
        \bibitem{C-E} D. Cordero-Erausquin, {\it Some applications of mass transport to Gaussian-type inequalities}, Arch. Ration. Mech. Anal. {\bf 161} (2002), 257--269. 
	%\bibitem{CFL} D. Cordero-Erausquin, Matthieu Fradelizi, Dylan Langharst, \textit{On a Santal\'{o} point for the Nakamura--Tsuji Laplace transport}, arXiv:2409.05541. 
        \bibitem{DEOPSS} S. Das Gupta, M. L. Eaton, I. Olkin, M. Perlman, L. J. Savage, M. Sobel, \textit{Inequalities on the probability content of convex regions for elliptically contoured distributions}, Proceedings of the Sixth Berkeley Symposium on Mathematical Statistics and Probability (Univ. California, Berkeley, Calif., 1970/1971), vol. 2, 1972, pp. 241--265.
        \bibitem{ENT}A. Eskenazis, P. Nayar, and T. Tkocz, {\it Gaussian mixtures: entropy and geometric inequalities}, Ann. Probab. {\bf 46} (2018), 2908--2945. 
	\bibitem{FraArchMath} M. Fradelizi, \textit{Sections of convex bodies through their centroid}, Arch. Math. {\bf 69} (1997), 515--522.  

        \bibitem{GKKO} N. Gigli, C. Ketterer, K. Kuwada, S. Ohta, \textit{Rigidity for the spectral gap on RCD$(K,\infty)$-spaces}, Amer. J. Math. {\bf 142} (2020), 1559--1594.
        
        \bibitem{Harge} G. Harg\'{e}, {\it A particular case of correlation inequality for the Gaussian measure}, Ann. Probab. {\bf 27} (1999), 1939--1951. 
        \bibitem{Harge2} G. Harg\'{e}, \textit{A convex/log-concave correlation inequality for Gaussian measure and an application to abstract Wiener spaces}, Probab. Theory Relat. Fields.  \textbf{130} (2004), 415--440. 
        \bibitem{Hu} Y. Hu, \textit{It\^{o}-Wiener chaos expansion with exact residual and correlation, variance inequalities}, J. Theoret. Probab. \textbf{10} (1997), 835--848.
        \bibitem{Kha} C. G. Khatri, {\it On certain inequalities for normal distributions and their applications to simultaneous confidence bounds}, Ann. Math. Statist. {\bf 38} (1967), 1853--1867. 
        \bibitem{LM}  R. Lata\l a and D. Matlak, {\it Royen's proof of the Gaussian correlation inequality}, In Geometric aspects of functional analysis, volume 2169 of Lecture Notes in Math., pages 265--275. Springer, Cham, 2017.
        \bibitem{Lieb} E. H. Lieb, \textit{Gaussian kernels have only Gaussian maximizers},  Invent.  Math. \textbf{102} (1990), 179--208. 
    \bibitem{Mil} E. Milman, {\it Gaussian Correlation via Inverse Brascamp-Lieb}, arXiv:2501.11018, to appear in Probab. Theory Relat. Fields. 
	\bibitem{NT3} S. Nakamura, H. Tsuji, \textit{A generalized Legendre duality relation and Gaussian saturation}, %arXiv:2409.13611v2, 
    Invent. math. (2025). https://doi.org/10.1007/s00222-025-01382-5. 
    \bibitem{Pitt} L. D. Pitt, {\it A Gaussian correlation inequality for symmetric convex sets}, Ann. Probab. {\bf 5} (1977), 470--474. 
    \bibitem{Roy14}T. Royen, {\it A simple proof of the Gaussian correlation conjecture extended to some multivariate gamma distributions}, Far East J. Theor. Stat. {\bf 48} (2014), 139--145. 
    %\bibitem{Roy24}T. Royen, {\it On improved Gaussian correlation inequalities for symmetrical $n$-rectangles extended to certain multivariate Gamma distributions and some further probability inequalities}, Far East J. Theor. Stat. {\bf 69} (2024), 1--38. 
    \bibitem{SSZ} G. Schechtman, Th. Schlumprecht, and J. Zinn, {\it On the Gaussian measure of the intersection}, Ann. Probab. {\bf 26} (1998), 346--357. 
    \bibitem{Sid}  Z. \v{S}id\'{a}k, {\it Rectangular confidence regions for the means of multivariate normal distributions}, J. Amer. Statist. Assoc. {\bf 62} (1967), 626--633. 
    \bibitem{SW} S. Szarek, E. Werner, {\it A non-symmetric correlation inequality for Gaussian measure}, J. Multivariate Anal. {\bf 68} (1999), 193--211.
    \bibitem{Tehr}M. R. Tehranchi, {\it Inequalities for the Gaussian measure of convex sets}, Electron. Commun. Probab., {\bf 22} (2017), 1--7.  
    \bibitem{Val} S. I. Valdimarsson, \textit{On the Hessian of the optimal transport potential}, Ann. Sc. Norm. Super. Pisa Cl. Sci. \textbf{6} (2007), 441--456. 
    \bibitem{Vill} C. Villani. Optimal Transport, Old and New. Springer, Berlin, 2009.
\end{thebibliography}
\end{document}